\documentclass[reqno,a4paper]{amsart}
\usepackage{amsthm,amsmath,amssymb}
\usepackage{ifthen}
\usepackage{datetime}
\usepackage{graphicx,color}
\usepackage[linktocpage=true]{hyperref}
\usepackage{comment}
\usepackage[normalem]{ulem}
\usepackage{caption}
\usepackage{subfigure}
\def\ub{\underline{u}}
\def\Lb{\underline{L}}
\def\Cb{\underline{C}} \def\Eb{\underline{E}}
\def\Fb{\underline{F}}
\def\Db{\bar{D}}
\def\chib{\underline{\chi}}
\def\gslash{\mbox{$g \mkern -8mu /$ \!}}

\def\doubleint{\int\!\!\!\!\!\int}
\def\nablaslash{\mbox{$\nabla \mkern -13mu /$ \!}}
\def\laplacianslash{\mbox{$\triangle \mkern -13mu /$ \!}}

\def\p{\partial}
\newcommand{\di}{\mathrm{d}} 

\newcommand{\D}{\mathcal{D}}
\newcommand{\R}{\mathcal{R}}
\newcommand{\Dp}{\mathcal{D}_{0, \ub}^{u_0, u}}
\newcommand{\Do}{\mathcal{D}_{\ub_1, \ub_2}^{u_1, u_2}}
\newcommand{\M}{\mathcal{M}}
\newcommand{\F}{\mathcal{F}}

\newcommand{\hori}{\mathcal{H}^+}

\newtheorem{theorem}{Theorem}[section]
\newtheorem{lemma}[theorem]{Lemma}
\newtheorem{proposition}[theorem]{Proposition}

\newtheorem{definition}[theorem]{Definition}
\newtheorem{remark}[theorem]{Remark}

\numberwithin{equation}{section}

\allowdisplaybreaks[4]

\begin{document}
\title[Large data for nonlinear wave in Schwarzschild]{A large data theory for nonlinear wave on the Schwarzschild background}

\author[S. Huo]{Saisai Huo} \email{11435024@zju.edu.cn}
\address{School of Mathematical Sciences, Zhejiang University, Hangzhou 310027, China}

\author[J. Wang]{Jinhua Wang} \email{wangjinhua@xmu.edu.cn}
\address{School of Mathematical Sciences, Xiamen University, Xiamen 31005, China}

\begin{abstract}
We study both of the scattering and Cauchy problems for the semilinear wave equation with null quadratic form on the Schwarzschild background. 
Prescribing the scattering data that are given by the short pulse data on the future null infinity and are trivial on the future event horizon, we construct a class of globally smooth solutions backwards up to any finite time and show that the wave travels in such a way that almost all of the (large) energy is focusing in an outgoing null strip, while little radiates out of this strip.
In reverse, considering a class of Cauchy data with large energy norms, there exists a unique and global solution  in the future development. And most of the wave packet is confined in an incoming null strip and reflected to the future event horizon, whereas little is transmitted to the future null infinity. 
\end{abstract}
\maketitle
\tableofcontents

\section{Introduction}\label{sec-intro}
We are concerned with the semilinear wave equation in the exterior region of Schwarzschild spacetime, of the form
\begin{equation}\label{Main Equation}
\Box_g \varphi = Q(\p \varphi, \p \varphi),
\end{equation}
where $\Box_g$ is the Laplace--Beltrami operator for the Schwarzschild metric, and $Q$ denotes non-linear term that is quadratic in the first order derivatives of the field $\varphi$ and satisfies the null condition (see Definition \ref{def-null-condition}). The data that we will consider for \eqref{Main Equation} will be some specific large data.

The small data theory for \eqref{Main Equation} has been well studied in the Minkowski spacetime $\mathbb{R}^{1+n}$. In dimension $n \geq 4$, the sufficiently fast decay rate of linear wave allows one to prove the global existence for the nonlinear wave equations with any quadratic nonlinearity for sufficiently small data \cite{Klainerman-80}. However, in $1+3$ dimension, F. John \cite{John-blowup} constructed a blowup example of nonlinear wave equations with certain quadratic nonlinearity. Nevertheless, if the quadratic nonlinearity satisfies the null condition, it has been proved independently by Christodoulou \cite{Christodoulou-86-wave} and Klainerman \cite{Klainerman-85-decay} that small data lead to solutions that are global in time. There has been an extensive literature on its applications \cite{Metcalfe-Nakamura-Sogge-05, Metcalfe-Sogge-07, Sideris-96, Sideris-Tu-01}. A far-reaching application of the idea of null condition in general relativity is the proof of nonlinear stability of the Minkowski spacetime \cite{C-K-90}, see also \cite{Lind-Rod-05, Lind-Rod-10}.

Based on the structure of null condition, Christodoulou  \cite{Christodoulou-09} initiated a large data theory for the Einstein vacuum equation. He introduced the \emph{short pulse} data, which is large in one certain null direction, and proved the formation of black holes  due to the focusing of gravitational waves. This work has been generalized and significantly simplified by Klainerman and Rodnianski \cite{K-R-12}. In addition, the ideas used in \cite{Christodoulou-09} and \cite{K-R-12} have been adapted to the wave equation \eqref{Main Equation} and the membrane equation in the Minkowski spacetime, see \cite{Miao-Pei-Yu, Wang-Wei-17, Wang-Yu-13, Wang-Yu-16}. 

We briefly recall some works on the linear and nonlinear wave equations in the asymptotically flat black hole spacetimes. The decay rate of linear wave has received intensive attention, see \cite{A-Blue-wavekerr, Blue-Soffer-wavesch, D-R-09, Dafermos-Rodnianski-smalla, D-R-13, D-R-S-R, D-S-A-12, Kerr-Smoller-decay, Kerr-Smoller-decay-Erratum, L-10, Luk-12-linear, Marzuola-Metcalfe-Tataru-Tohaneanu, MMDT-strich-sch-10, T-13, Tataru-Tohaneanu}. Closely related to this, there are quite a lot of results on the linearized gravity (related to the Regge-Wheeler equation, Teukolsky equation, etc.) \cite{Linear-Stability-Kerr, Andersson-Blue-Wang-17, D-H-R-16, Hartle-Wilkins, H-K-W-17, Ma17spin2Kerr, Press-Teukolsky-73, R-W-75}.
For the nonlinear wave, the global existence with power nonlinearity has been studied in \cite{Blue-Sterbenz, D-R-nonlinear-05, L-M-S-T-W-14, Nicolas-Sch-95, Nicolas-Kerr-02, Tohaneanu-12}; the small data global existence with null quadratic form in the slowly rotating Kerr spacetime has been demonstrated by Luk \cite{Luk-15-nonlinear} and the same theory for quasilinear wave equation in the spacetimes close to the Schwarzschild has been addressed by Lindblad and Tohaneanu \cite{L-T-quasilinear-sch-18},  We also mention some works on the scattering of  waves  (or gravity) in the black hole spacetimes \cite{D-H-D-scattering, D-R-S-Scattering, Dimock-84, Firedlander-80, Friedlander-01, Nicolas-16}, etc.

In the current work, we study the global-in-time behaviour of solutions to the semilinear wave equation \eqref{Main Equation} with the short pulse data in the Schwarzschild spacetime. 
\subsection{Main results}\label{Section main result}
To state our main theorem, we introduce some necessary concepts and notations on the Schwarzschild geometry.
The Schwarzschild spacetime is an $1+3-$dimensional Lorentzian manifold with
 the Lorentz metric taking the following form in the Boyer-Lindquist coordinates $(x^\alpha) = (t,r,\theta, \phi)$, 
\begin{align}
g_{\mu\nu} \di x^\mu \di x^\nu  ={}&{}
-\left(1-\frac{2m}{r}\right)\di t^2 +\left(1-\frac{2m}{r}\right)^{-1}\di r^2+r^2 \di \sigma_{S^2},
\label{eq:SchMetric}
\end{align}
where $\di \sigma_{S^2}$ is always the standard metric on the unit $2$-sphere $S^2$.
We consider the exterior region, which is given by $\mathcal{M}=\mathbb{R} \times [2m, \infty) \times S^2$.
For notational convenience, we set
\begin{align}\label{mu-eta}
\mu={}&\frac{2m}{r}, & \eta={}&1-\mu.
\end{align}
Let $r^\ast$ be the Regge-Wheeler tortoise coordinate 
\begin{equation}\label{Regge-Wheeler tortoise coordinate}
r^\ast = r + 2m \log(r-2m)-3m-2m\log m,
\end{equation}
and define the null coordinates $u = \frac{1}{2}(t-r^\ast), \,\, \ub = \frac{1}{2}(t+r^\ast)$.
The future null infinity $\mathcal{I}^{+}$ of $\M$ can be parametrized by $\{\ub = +\infty\}$. For any $c \in \mathbb{R}$, $C_c$ is used to denote the level surface $\{u=c\}$; Similarly, $\Cb_{\ub}$ denotes a level set of $\ub$. The intersection $C_u \cap \Cb_{\ub}$ is a $2$-sphere denoted by $S_{\ub,u}$, and $\Sigma_t$ is the constant $t$ hypersurface.

Define $L$, $\Lb$ and $Y$ by
\begin{equation*}
 L =\p_{\ub}= \partial_t + \partial_{r^\ast}, \quad \Lb=\p_u = \partial_t - \partial_{r^\ast}, \quad Y=\eta^{-1} \Lb.
\end{equation*}
Then $\{ L, Y \}$ is a normalized null frame. Let $\nablaslash$ be the induced covariant derivative on $S_{\ub,u}$. 
We can now define the ``good'' ($\Db$) and ``bad'' ($L$) derivatives,
\begin{equation*}
 \Db \in \{Y, \nablaslash \},  \quad  D \in \{ Y, L, \nablaslash\}.
\end{equation*}
Besides, let $\{ \Omega_i \}_{i=1}^3$ be a basis of the killing vectors spanning the Lie algebra $so(3)$. These are angular derivatives on $S_{\ub,u}$. We shall use the short cut: for any given function $\psi$, $\Omega \psi = \Omega_i \psi$, $\Omega^2 \psi = \Omega_i \Omega_j \psi, \, i, j \in \{1, 2, 3\}$, etc.

Near the horizon, we also use the Eddington-Finkelstein coordinates $(r, \ub, \theta, \phi)$, in which the metric reads $$g_{\mu \nu} \di x^\mu \di x^\nu = -\eta \di \ub^2 + 2 \di r \di \ub +r^2 \di \sigma_{S^2},$$ and extends across the event horizon.

We now define the null condition for a quadratic form \cite{Luk-15-nonlinear}.
\begin{definition}\label{def-null-condition}
Consider the quadratic form $Q(D\psi_1, D \psi_2)$. We say that $Q$ satisfies the null condition if
\begin{align*}
Q= \Lambda_1 (u, \ub,\theta, \phi) D \psi_1 \Db \psi_2 +  \Lambda_2 (u, \ub,\theta, \phi) D \psi_2 \Db \psi_1,
\end{align*}
and
\begin{align*}
|\p_t^{i_1} Y^{i_2} \Omega^{i_3} \Lambda_j |\lesssim t^{-i_1} r^{-i_2}, \quad  j =1,2.
\end{align*}
\end{definition}
The notation $x \lesssim y$ means $x \leq cy$ for a universal constant $c$, and $x \sim y$ means $x \lesssim y$ and $y \lesssim x$.  

Now we are ready to present our first theorem concerning the scattering problem.
The asymptotic characteristic data will be imposed on the future null infinity $\mathcal{I}^+$ and the future event horizon $\hori$. 
Let $\delta >0$ and let $\varphi_{+\infty}: \mathcal{I}^+ \rightarrow \mathbb{R}$ be such that
\begin{equation}\label{def-psi-0}
\varphi_{+\infty} (u, \theta, \phi)=
\begin{cases}
0, &\text{if} \quad u > 0 \,\, \text{or} \,\, u < - \delta, \\
\delta^{\frac{1}{2}} \psi_0 \left(\frac{u}{\delta}, \theta, \phi \right), &\text{if} \quad -\delta \leq u \leq 0,
\end{cases}
\end{equation} 
where $\psi_0: [-1,0] \times S^2 \rightarrow \mathbb{R}$ is a smooth, compactly supported function defined on $\mathcal{I}^+$. We recall that $\D^+(\Sigma) (\D^-(\Sigma))$ is the future (past) Cauchy development of $\Sigma$.

\begin{theorem}[Scattering Theorem]\label{Thm-scattering-region}
Consider on the Schwarzschild background the scattering problem (without contribution from $\hori$)  for the semilinear wave equation \eqref{Main Equation} with $Q$ satisfying the null condition. 
The asymptotic characteristic data are given by
\begin{align*}
\ub \varphi(u, \ub, \theta, \phi)\big|_{\mathcal{I}^+} & = \varphi_{+\infty} (u, \theta, \phi),\quad 
\varphi(u, \ub, \theta, \phi)\big|_{\hori} \equiv  0,
\end{align*}
where $\varphi_{+\infty} \in C^\infty ( \mathcal{I}^+)$ is defined in \eqref{def-psi-0}.
If $\delta$ is small enough, \eqref{Main Equation} has a globally smooth solution  in $ \D^{-}(\mathcal{I}^{+}) \cap \D^-(\hori) \cap \D^{+}(\Sigma_{0})$, whose radiation field on $\mathcal{I}^+$ is exactly $\varphi_{+\infty}$. 
And most of the wave energy is concentrated in the null strip $\mathcal{N}_1 := \D^{-}(\mathcal{I}^{+}) \cap \D^-(\hori) \cap \D^{+}(\Sigma_{0}) \cap \{-\delta \leq u \leq 0\}$, while little is dispersing out of $\mathcal{N}_1$ (see Figure \ref{fig:Scattering}).
\end{theorem}
\begin{remark}\label{rk-uniqueness}
The statement above does not assert the uniqueness of the global solution in $ \D^{-}(\mathcal{I}^{+}) \cap \D^-(\hori) \cap \D^{+}(\Sigma_{0})$ with prescribed scattering data at $\hori \cup \mathcal{I}^+$. This had been explained earlier in \cite[Section 1.3.4]{D-H-D-scattering}.

As a remark, the uniqueness for a solution of the scattering problem in the Minkowski spacetime is understood in a class of solutions whose asymptotic behaviour resembles the linear wave \cite[Main Theorem 2]{Wang-Yu-16}, namely,
$$|\varphi| = O(1/t ), \quad |L \varphi| = O(1/t), \quad |\Db \varphi| = O(1/t^{ \frac{3}{2}}).$$
However, this does not hold true in the Schwarzschild spacetime, for the decay of linear wave is generically not strong enough in these asymptotically flat black hole spacetimes \cite{D-R-09, L-10, Luk-12-linear}. Practically, we only prove $|D\varphi| = O(1/t)$ in the ``small data'' region $\D^-(C_{-\delta}) \cap \D^+(\Sigma_0)$, see Section \ref{sec-global-region-III-1}.
\end{remark}
\begin{figure}
\centering
\includegraphics[width=3.5in]{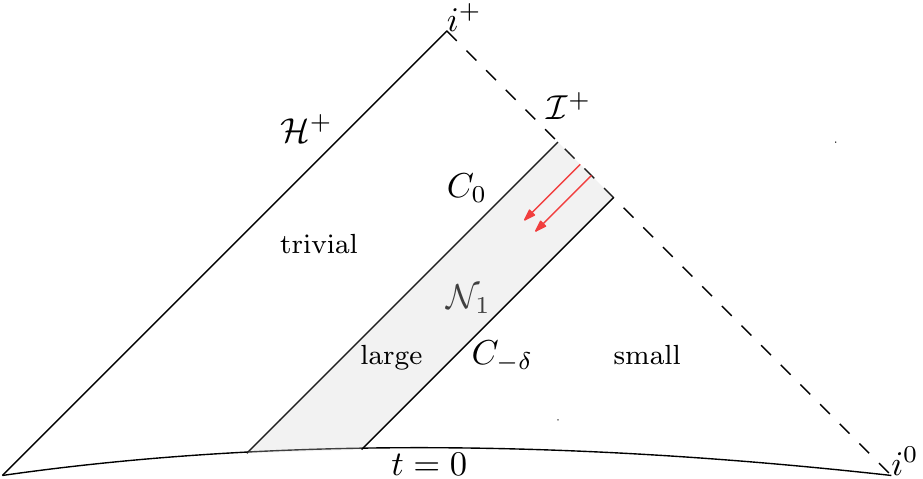}
\caption{Scattering Theorem}
\label{fig:Scattering}
\end{figure}

The global Cauchy development for the semilinear wave equation with large data is stated as follows. 
\begin{theorem}[Cauchy development]\label{Thm-froward-whole-region}
Consider the Cauchy problem for the semilinear equation \eqref{Main Equation} with  $Q$ satisfying the null condition, where the Cauchy data are given by $\left( \varphi|_{\Sigma_{1}}, \partial_t\varphi|_{\Sigma_{1}} \right) = (\psi_0, \psi_1)$. Fix an integer $N \in \mathbb{N}$, $N \geq 30$.
If $\delta$ is small enough, there is an initial data set $(\psi_0, \psi_1)$ verifying
\begin{align*}
E_k(\psi_0, \psi_1) \sim \delta^{-k+1}, \quad 1 \leq k\leq N,
\end{align*}
where $ E^2_k(\psi_0, \psi_1)= \int_{\Sigma_{1}} ( |D^k \psi_0|^2 + |D^{k-1} \psi_1|^2) \di x^3$,
so that a unique and global solution $\varphi$ exists in $\D^{+}(\Sigma_{1}) \cap \D^-(\hori) \cap \D^-(\mathcal{I}^+)$ (see Figure \ref{fig:Cauchy-development}). Moreover, the wave profile is mostly transmitted along the null strip $\mathcal{N}_2 := \D^{-}(\mathcal{I}^{+}) \cap \D^-(\hori) \cap \D^{+}(\Sigma_{1}) \cap \{0 \leq \ub \leq \delta\}$ to the future event horizon $\hori$, whereas little is propagated to the future null infinity $\mathcal{I}^+$. 
\end{theorem}
\begin{remark}\label{rk-Cauchy-theorem}
Theorem \ref{Thm-froward-whole-region} should be fundamentally distinguished from the cases in \cite{Miao-Pei-Yu, Wang-Yu-16}, where most of the wave profile 
disperses to the null infinity $\mathcal{I}^+$. Essentially, the wave  in Theorem \ref{Thm-froward-whole-region} is travelling without decay in $\mathcal{N}_2$ and the energies transmitted to the event horizon $\hori$ are large. This can be read off from the $L^\infty$ estimate in Region $\R_2$ in Theorem \ref{theorem-whole-R123}.
\end{remark}
\begin{figure}
\centering
\includegraphics[width=3.5in]{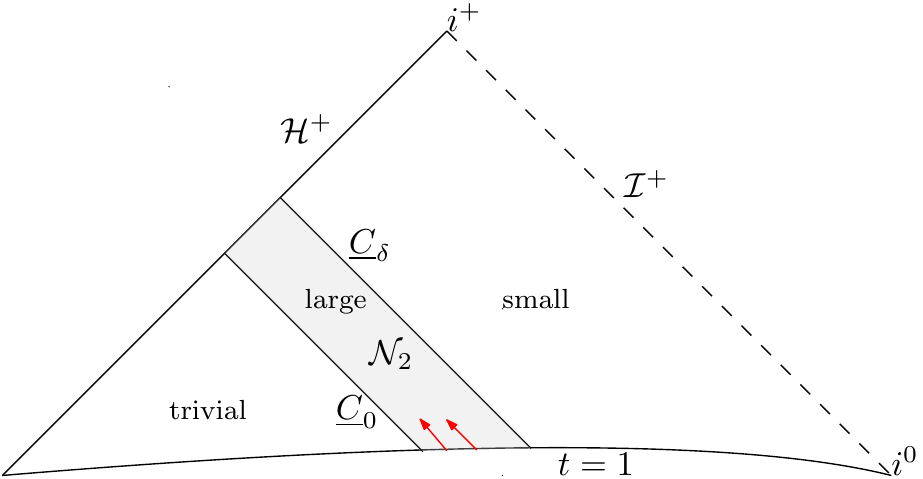}
\caption{Cauchy development}
\label{fig:Cauchy-development}
\end{figure}

\subsection{Outline of the proof}\label{sec-outline-proof}
The main body of this paper is devoted to proving the following semi-global statement.
Define
\begin{equation}\label{def-psi-0-past}
\varphi_{-\infty} (\ub, \theta, \phi)=
\begin{cases}
0, &\text{if} \quad \ub < 0 \,\, \text{or} \,\, \ub > \delta, \\
\delta^{\frac{1}{2}} \psi_0 \left(\ub / \delta,  \theta, \phi \right), &\text{if} \quad 0 \leq \ub \leq \delta,
\end{cases}
\end{equation} 
where $\psi_0: [0, 1] \times S^2 \rightarrow \mathbb{R}$ is a smooth, compactly supported function defined on $\mathcal{I}^-$.   Let $r_{NH}$ be close to $2m$, satisfying $2m<r_{NH} <1.2 r_{NH} <3m$.
\begin{theorem}\label{theorem-whole-R123}
Consider on the Schwarzschild background the semilinear wave equation \eqref{Main Equation} with $Q$ satisfying the null condition and with the asymptotic characteristic data 
\begin{align*}
|u|  \varphi(u, \ub, \theta, \phi)\big|_{\mathcal{I}^-} & = \varphi_{-\infty} (\ub, \theta, \phi), \quad  \varphi(u, \ub, \theta, \phi)\big|_{\mathcal{H}^-} \equiv 0,
\end{align*}
where $\varphi_{-\infty} \in C^\infty ( \mathcal{I}^-)$ is defined in  \eqref{def-psi-0-past}. There exists a constant $\delta_0$ such that if $\delta < \delta_0$, \eqref{Main Equation} has a global solution $\varphi$ in the null strip $\R_1 \cup \R_2 := \D^{+}(\mathcal{I}^{-}) \cap \D^-(\hori) \cap \{0 \leq \ub \leq \delta \}$, with $\R_1:=\D^-(\Sigma_1) \cap \D^+(\mathcal{I^-}) \cap  \{0\leq \ub \leq \delta\}$ and $\R_2:=\D^+(\Sigma_1) \cap \D^-(\hori)  \cap \{0\leq \ub \leq \delta\}$ (see Figure \ref{fig:global}). 

In particular, fix any integer $N \in \mathbb{N}$, $N \geq 30$,
 the solution $\varphi$ obeys the following estimates: for $k+l+j \leq N-2$,
 \begin{align*}
 |L^{1+k} \Lb^l \Omega^j \varphi| \lesssim{}& \delta^{-\frac{1}{2}-k} |t|^{-1}, \quad  |\bar D L^{k} \Lb^l \Omega^j  \varphi| \lesssim{}\delta^{\frac{1}{4}-k} |t|^{-\frac{3}{2}}, & \text{in} \, \, \R_1,\\
   |L^{1+k} Y^l \Omega^j\varphi| \lesssim{}& \delta^{-\frac{1}{2}-k}, \quad\quad\quad |\bar D L^k Y^l \Omega^j \varphi| \lesssim{} \delta^{\frac{1}{4}-k}, & \text{in} \,\,  \R_2.
\end{align*} 
And the energies of the solution $\varphi$ are of the size $\delta^{\frac{1}{2}}$ on the last cone $\Cb_\delta$.
\end{theorem}
\begin{remark}\label{rk-R1-R2}
Compared to \cite{Christodoulou-09} and \cite{Wang-Yu-16}, where the main estimates are merely obtained in the past region $\R_1$, our results are able to cover both of the past region $\R_1$ and the future region $\R_2$. Basically, the decaying mechanics in $\R_1$ and $\R_2$ are completely different. 
\end{remark}
\begin{figure}
\centering
\includegraphics[width=3.0in]{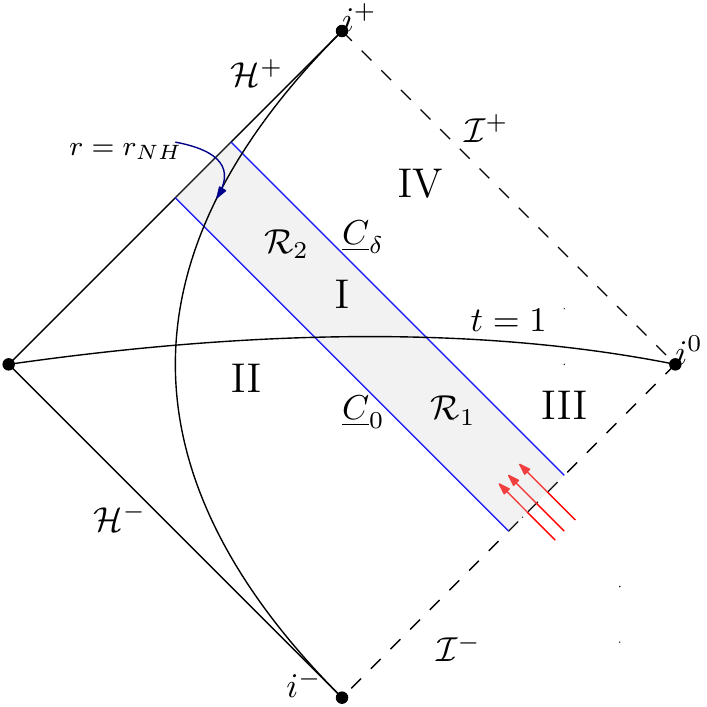}
\caption{Semi-global existence with large data}
\label{fig:global}
\end{figure}

The proof of Theorem \ref{theorem-whole-R123} is indicated in the sections \ref{sec-past} and \ref{sec-future}.
Additionally, the result of Theorem \ref{theorem-whole-R123} entails the scattering theorem \ref{Thm-scattering-region} and the Cauchy development theorem \ref{Thm-froward-whole-region}. We will give an overview for the proof in what follows.

The exterior region is divided into four parts: I $\,:=\R_1 \cup \R_2$, II $\,:=\D^+(\mathcal{I^-}) \cap \D^+(\mathcal{H}^-) \cap \D^-(\Cb_0)$ and III $\,:=\D^+(\mathcal{I^-}) \cap \D^+(\mathcal{H}^-) \cap \D^+(\Cb_\delta) \cap \D^-(\Sigma_1)$, IV $\,:=\D^-(\mathcal{I^+}) \cap \D^-(\mathcal{H}^+) \cap \D^+(\Cb_\delta) \cap \D^+(\Sigma_1)$ (see Figure \ref{fig:global}). Our main estimates are conducted in Region I (Theorem \ref{theorem-whole-R123}), where the energy norms of the solution are large, while in Region II, the solution is extended by zero.

To establish the energy argument in Region I, we employ the following multipliers:
\begin{align*}
\xi_1={}& \eta L +  \delta^{-1} \Lb, & \text{in} & \quad \R_1,\\
\xi_2={}&\eta L +  \delta^{-1} (1+\mu) \Lb,  &\text{in} &\quad \R_2,\\
\xi_3={}&\left( 1+y_2(r^\ast) \right) L + \delta^{-1}  y_1 (r^\ast)  Y,  &\text{in} &\quad \R_2\cap \{r < r_{NH}\}.
\end{align*}
The energy estimate in $\R_1$ is the Schwarzschild analogue of that in \cite{Christodoulou-09, Wang-Yu-16}, and it also provides an simplification for the proof of \cite{Wang-Yu-16}.
We remark that the decay rate in $\R_1$ is already determined by the asymptotic characteristic data (or radiation field on the past null infinity). However, in $\R_2$, one of the main difficulties is to figure out the quantitative decay rates for both of the degenerate and non-degenerate energies.
To begin with, we prove the following degenerate decay estimate in $\R_2$ by means of the multiplier $\xi_2$ which degenerates at the horizon $\hori$: for any $\beta \geq \frac{1}{2}$, $k+l+j \leq N-2$,
\begin{equation}\label{intr-decay-deg}
|\eta^{\frac{1}{2}} L^{1+k} \Lb^l \Omega^j\varphi| \lesssim \delta^{-\frac{1}{2}-k} |u|^{-\beta} ,  \,\, |\eta^{\frac{1}{2}}\bar D L^k \Lb^l \Omega^j\varphi| \lesssim{} \delta^{\frac{1}{4}-k} |u|^{-\beta}, \quad \text{in} \,\, \R_2.
\end{equation}
The idea beyond the choice of $\xi_2$ in $\R_2$ lies in the facts that, upon using $\xi_2$, there is a positive (i.e., with a favourable sign) contribution from the spacetime integral $$\doubleint  \left( \delta^{-1} | \Lb \varphi|^2 +\delta^{-1} |\nablaslash \varphi|^2 + |L \varphi|^2 \right) \eta r^2 \di u \di \ub \di \sigma_{S^2}$$ in the energy estimate, noting that $r$ is always finite in $\R_2$, see Section \ref{sec-deg-multiplier-II}. With this positive spacetime integral, we get an energy inequality taking the form of $$E^{deg}(u) + \int_{u_1}^{u} E^{deg}(u^\prime) \di u^\prime \lesssim E^{deg}(u_1) + \delta^{\frac{1}{2}} |u_1|^{-2\beta}, \quad \forall \, u_1 <u, \, \beta \geq \frac{1}{2},$$ where $E^{deg}(u) = \int_{C_u \cap \R_2} \left(\delta^{-1} |\nablaslash \varphi|^2 + |L \varphi|^2 \right) \eta r^2 \di \ub \di \sigma_{S^2}$ denotes the degenerate energy. Then a pigeon-hole argument can be applied to achieve the energy decay $E^{deg}(u) \lesssim |u|^{2\beta}$. On the other, the energy involving the transversal derivative $Y$, ${}^tF^{deg}(u) = \int_{C_u \cap \R_2}  |Y \varphi|^2 \eta r^2 \di \ub \di \sigma_{S^2}$, can be retrieved by integrating along $L$ and making use of the wave equation. As it stands, \eqref{intr-decay-deg} can be regarded as ``fake'' decay estimate  (recalling that $\eta = 1-\mu$ and $\eta \sim \exp(-u)$ in $\R_2$), and it degenerates at $\hori$. Nevertheless, based on \eqref{intr-decay-deg} (and the associated degenerate energy estimate), the non-degenerate energy estimate will be inferred with the help of $\xi_3$, which is actually the red-shift vector field \cite{D-R-09}, and is now well adapted to the sizes of the profiles $L\varphi$ and $Y\varphi$, see Section  \ref{sec-multiplier-horizon}. 

With more efforts, we can show that the solution in Region I is small on the last incoming cone $\Cb_{\delta}$, as proved in the sections \ref{sec-small-Cb-1} and \ref{sec-smallness-Cb-2}.

To accomplish the scattering theorem, it remains to prove the global existence in Regions III, which is reduced to be a small data problem. This is carried out by introducing some ideas in \cite{MMDT-strich-sch-10, L-T-quasilinear-sch-18} (without using the conformal multiplier  $K=u^2 \Lb + \ub^2 L$, for $K$ does not have a favourable sign near the spatial infinity in Region III and hence is not allowed when considering the scattering problem), see Section \ref{sec-global-region-III-1}. In the end, we reverse the time $t$ to conclude the scattering statement. 
When it comes to the Cauchy development, we are left with the global existence in Region IV, with data imposed on $(\Cb_\delta \cap \{t \geq 1\}) \cup (\Sigma_1\cap \{\ub > \delta\})$, for which we shall apply the small data theorem for the Cauchy problem in \cite{Luk-15-nonlinear}, see Section \ref{sec-small-data-Cauchy}. 

 The paper is organized as follows. In Section \ref{Section Preliminaries}, we introduce several notations and the energy estimates scheme in the Schwarzschild spacetime. In Section \ref{sec-past}, we show the global existence with scattering data at the past event horizon and the past null infinity. In Section \ref{sec-future}, the global existence for the Cauchy problem is stated. More background knowledge is collected in the appendix.

{\bf Acknowledgement}: The authors would like to thank Prof. Mihalis Dafermos for helpful comments on the issue of uniqueness for the scattering problem, and professors Boling Guo and Xiao Zhang, and the Academy of Mathematics and Systems Science, Chinese Academy of Science (AMSS, CAS), the Institute of Applied Physics and Computational Mathematics (IAPCM), for great support and hospitality during the preparation of this work. S.H. is also deeply indebted to her advisor---Prof. Dexing Kong for guidance. 
J.W. is grateful to Prof. Jianwen Zhang for great help and Siyuan Ma for enlightening discussions as well. And J.W. is supported by NSFC (Grant No. 11701482), 
and NSF of Fujian Province (Grant No. 2018J05010).

\section{Preliminaries}\label{Section Preliminaries}

\subsection{Notations}\label{sec-notation}
We clarify the measures: $\di \mu_{\D} =  r^2 \di u \di \ub \di \sigma_{S^2}$, 
$\di \mu_{C_u} = r^2 \di \ub \di \sigma_{S^2}$, $\di \mu_{\Cb_{\ub}} = r^2 \di u \di \sigma_{S^2},$ $\di \mu_{\Cb_{\ub}^{ND}} = r^2 \eta \di u  \di \sigma_{S^2}$. Here $\di \mu_{\Cb_{\ub}^{ND}}$ means the non-degenerate volume form on $\Cb_{\ub}$. In $(r, \ub)$ coordinate, $\di \mu_{\Cb_{\ub}^{ND}}= -  r^2  \di r \di \sigma_{S^2}$. And the spacetime volume form takes the form of $\eta \di \mu_{\D}$.
We denote $\|\cdot \|_{L^2(C_u)}, \| \cdot \|_{L^2(\Cb_{\ub})}$ and $\|\cdot \|_{L^2(\Cb_{\ub}^{ND})}$ the $L^2$ norm with the corresponding volume form respectively.

Define the following truncated cone: $C_{u_1}^{[\ub_1, \ub_2]}:= \{u_1\} \times {[\ub_1, \ub_2]} \times S^2$, $\Cb_{\ub_1}^{[u_1, u_2]}:= {[u_1, u_2]} \times \{\ub_1\} \times S^2$. The spacetime domain bounded by $C_{u_1}^{[\ub_1, \ub_2]}, C_{u_2}^{[\ub_1, \ub_2]}$ and  $\Cb_{\ub_1}^{[u_1, u_2]}, \Cb_{\ub_2}^{[u_1, u_2]}$ is denoted by $\D_{u_1, u_2}^{\ub_1, \ub_2}$.

Define the degenerate and non-degenerate null vector fields: $W \in \{L, \Lb\}$, $Z \in \{L, Y\}$. We shall introduce the following simplifications: $W^n_{p,q} := L^p \Lb^q,\, p+q =n$,  $\bar W^i_{p,q} \in \{L^k \Lb^l | k+l=i, k \leq p, l \leq q\}$ and $Z^n_{p,q} := L^p Y^q,\, p+q =n$, $\bar Z^i_{p,q} \in \{L^k Y^l | k+l=i, k \leq p, l \leq q\}$.

We use the notation $C, \, c$ to denote positive numerical constants that are free to vary from line to line. We allow $C, c$ to depend on 
the amount of Sobolev regularity that we assume on the initial data, but we always choose these constants so that they are independent of the solution.

We always use the notation $\langle x \rangle = \sqrt{1+x^2}$.

Throughout this paper, we set
\begin{equation}\label{def-varphi-k}
\varphi_i = \Omega^i \varphi, \quad |\varphi_k|^2=\sum_{i\leq k} |\Omega^i \varphi|^2.
\end{equation}

\subsection{Energy estimates scheme}\label{Energy estimates scheme}
We would like to briefly review the vector field method. In the case of wave equation on the Schwarzschild background, the energy momentum tensor associated to the wave equation for $\psi$ is defined to be
\begin{equation}\label{energy-tensor-1}
\mathcal{T}_{\alpha\beta}(\psi)=D_\alpha \psi D_\beta \psi-\frac{1}{2} g_{\alpha\beta} D^\gamma \psi D_\gamma \psi,
\end{equation}
where $D$ denotes the covariant derivative corresponding to the spacetime metric $g$.
We note that $\mathcal{T}_{\alpha\beta}$ is symmetric and there is the divergence identity for the energy-momentum tensor,
\begin{equation}\label{divergence of T}
 D^\alpha  \mathcal{T}_{\alpha \beta}(\psi) =\Box_g \psi \cdot D_\beta \psi.
\end{equation}
Given a vector field $\xi$, which is usually called a \emph{multiplier vector field}, the associated energy currents are defined as follows
\begin{align*}
P_{\alpha}^\xi(\psi) = \mathcal{T}_{\alpha \beta}(\psi) \cdot \xi^\beta, \quad K^\xi(\psi)=\mathcal{T}^{\mu\nu}(\psi)\, ^{\xi}\pi_{\mu \nu},
\end{align*}
where $^{\xi}\pi_{\mu \nu}$ is the deformation tensor defined by
\begin{equation*}
  ^{\xi}\pi_{\mu \nu} = \frac{1}{2} \mathcal{L}_\xi g_{\mu \nu} = \frac{1}{2}(D_\mu \xi_{\nu} + D_\nu \xi_{\mu}).
\end{equation*}
Due to \eqref{divergence of T}, we have
\begin{equation}\label{divergence of P}
D^\alpha P^{\xi}_{\alpha}(\psi)=  K^\xi(\psi) + \Box_g \psi \cdot \xi \psi.
\end{equation}
Integrating \eqref{divergence of P} on the spacetime domain, we then derive the energy identity,
\begin{equation}\label{energy-identity}
\begin{split}
& \int_{C_u^{[\ub_0, \ub]}} T_{\p_{\ub} \xi} (\psi)  r^2 \di \ub \di \sigma_{S^2} + \int_{\Cb_{\ub}^{[u_0, u]}}T_{\p_u \xi}(\psi)  r^2 \di u \di \sigma_{S^2}  \\
=& \int_{C_{u_0}^{[\ub_0, \ub]}} T_{\p_{\ub} \xi} (\psi) r^2 \di \ub \di \sigma_{S^2} + \int_{\Cb_{\ub_0}^{[u_0, u]}} T_{\p_u \xi} (\psi) r^2 \di u \di \sigma_{S^2} \\
& \quad \quad + \doubleint_{\D_{u_0, u}^{\ub_0, \ub}} \left( -2 K^\xi (\psi) -2 \Box_g \psi \cdot \xi \psi \right)  \eta r^2 \di u \di \ub \di \sigma_{S^2}.
\end{split}
\end{equation}
We also define the modified currents associated to a function $q$,
\begin{align*}
P_\alpha^\xi (\psi, q)&=P_{\alpha}^\xi(\psi)+q D_\alpha \psi \cdot \psi-\frac{1}{2} D_\alpha q \cdot \psi^2,\\
K^\xi(\psi, q)&=K^\xi(\psi) + q D^\gamma \psi D_\gamma \psi - \frac{1}{2} \Box_g q \cdot \psi^2.
\end{align*}
Then
\begin{equation*}
D^\alpha P^{\xi}_{\alpha}(\psi, q)= K^\xi(\psi, q) + \Box_g \psi \cdot (\xi \psi + q \psi).
\end{equation*}
An energy identity analogous to \eqref{energy-identity} holds true after integration by parts.

\subsection{Vector fields}\label{Vector fields}
In terms of the null frame $\{\Lb, L, e_A, A=1,2 \}$, where $\{e_A,A=1,2\}$ is an orthonormal basis on $S_{\ub ,u}$, the energy-momentum tensor \eqref{energy-tensor-1} reads $\mathcal{T}_{\ub \ub}(\psi) = |L \psi|^2, \,
 \mathcal{T}_{u\ub}(\psi) = \eta |\nablaslash \psi|^2, \, \mathcal{T}_{uu}(\psi) = |\Lb \psi|^2$.
The deformation tensor for $L$ is computed as
\begin{align*}
&  ^{L}\pi_{\ub \ub} = 0, &  ^{L}\pi_{uu} =0,\\
 &  ^L\pi_{u \ub} = - \frac{\mu \eta}{r}, & ^L\pi_{AB} = \frac{\eta}{r} \gslash_{AB},
\end{align*}
and at the same time, there is $^{\Lb}\pi_{\alpha \beta} = - ^L\pi_{\alpha \beta}.$

\subsubsection{The multiplier  vector fields}\label{sec-Multiplier-Vector fields}
We employ a class of multipliers of the following type
\begin{equation*}
X  = f_1(u,\ub)) L + f_2(u,\ub)) \Lb,
\end{equation*}
with $f_i(u,\ub), \, i=1,2$ being some functions to be determined.
The current is now calculated as
\begin{equation}\label{current of fL}
\begin{split}
 K^{X}(\psi) 
 &=\p_u f_1 g^{u\ub} |L\psi|^2 - \frac{1}{2}(\p_{\ub} f_1+\frac{\mu f_1}{r} ) |\nablaslash \psi|^2 - \frac{2\eta}{r} f_1 g^{u \ub} L\psi \Lb \psi \\
 &+ \p_{\ub} f_2 g^{u\ub} |\Lb \psi|^2 - \frac{1}{2}(\p_u f_2- \frac{\mu f_2}{r} ) |\nablaslash \psi|^2 + \frac{2\eta}{r} f_2 g^{u \ub} L\psi \Lb \psi.
\end{split}
\end{equation}

Let $K = u^2 \Lb + v^2 L$ be the conformal multiplier, and $q = \frac{(1-\mu)}{r}( v^2 -u^2) =  \frac{\eta t r^\ast}{r}$, then 
\begin{equation}\label{current-K}
K^{K}(\psi, q) = t \left(\frac{r^\ast}{r} \left(1- \frac{3m}{r} \right) - 1\right) |\nablaslash \psi|^2 - \frac{t}{2} \frac{\mu}{r^2} \left( 2 + \frac{r^\ast}{r} (4\mu -3) \right) \psi^2.
\end{equation}
If $r>2m$ is small enough, then 
$K^{K}(\psi, q) \sim - \frac{m r^\ast t}{r^2} |\nablaslash \psi|^2 -  \frac{t}{2r^2} \frac{r^\ast}{r} \psi^2$ would be non-negative ($t \geq 0$); If $r>R$ large enough, then $K^{K}(\psi, q) \sim \frac{r^\ast t}{r} |\nablaslash \psi|^2 + \frac{3t}{2} \frac{\mu}{r^2} \frac{r^\ast}{r} \psi^2$ would be non-negative ($t \geq 0$) as well.

\subsubsection{The commutators}\label{sec-commutator}
For most of the computations throughout this paper, we will need several commuting formulae. Here, we collect all of them as follows,
\begin{equation}\label{commutates}
\begin{split}
[\Omega, \nablaslash]&=0, \quad\,\, [\Lb, \nablaslash]= \frac{\eta}{r} \nablaslash, \quad\,\, [L, \nablaslash]=-\frac{\eta}{r} \nablaslash,\\
[\Box_g, \Omega]&=0, \quad\,\, [\Lb, \Omega]=0, \quad\quad\,\,\, [L, \Omega]=0, \\
 [\Lb, Y] &= \frac{\mu}{r} Y, \quad\quad\quad\quad\quad\quad\quad\,\,\,  [L, Y] =- \frac{\mu}{r} Y,
\end{split}
\end{equation}
and the commutator with the wave operator:
\begin{equation}\label{commutates-Y-L-Box}
\begin{split}
[\Box_g, Y]   &=\frac{2m}{r^2} Y^2 - \frac{2}{r} \laplacianslash + \frac{1}{r^2} Y - \frac{1}{r^2} L, \\
[\Box_g, L]  &= \frac{\eta - \mu}{r^2}(L-\Lb)+\frac{2\eta-\mu}{r}\laplacianslash +  \frac{\mu}{r} \Box_g, \\
[\Box_g, \Lb] &= \frac{\eta - \mu}{r^2}(\Lb-L)-\frac{2\eta-\mu}{r}\laplacianslash -  \frac{\mu}{r} \Box_g.
\end{split}
\end{equation}
We also provide the following derived commutator,
\begin{align*}
[\Box_g, u \Lb] =& \left(1- \frac{\mu u}{r}\right) \Box_g +\left(\frac{(\mu-2\eta) u}{r} - 1\right)  \laplacianslash  - \frac{1}{r} L  + \frac{(\eta - \mu) u}{r^2} (\Lb  - L ),
\end{align*}
which will be used in Region $\R_1$. We note that, $r \sim| u|$ in $\R_1$, and hence,
\begin{equation}\label{commutator-S-Box-R1}
[\Box_g, u \Lb] 
\sim  \pm \Box_g  \pm  \laplacianslash   \pm \frac{1}{r} (\Lb - L ), \quad \text{in} \,\, \R_1.
\end{equation}
In general, we conclude the following lemma.
\begin{lemma}\label{lemma-commuting}
Let $W \in \{L, \Lb\}$, $W^n_{p,q} := L^p \Lb^q$, $p+q =n$. Then
\begin{align*}
\Box_g W^n_{p,q} \varphi_k &= W^n_{p,q} \Omega^k \Box_g \varphi \pm W_{\leq n-1} (\varphi_k),
\end{align*}
where 
\begin{align*}
 W_{\leq n-1} (\varphi_k) &= \sum_{i\leq n-1} \frac{1}{r^{n+1-i}}(\Lb \bar W^i_{p,q} \varphi_k \pm L \bar W^i_{p,q}  \varphi_k \pm   r\laplacianslash \bar W^i_{p,q} \varphi_k) \\
 &\pm \frac{\mu}{r}\bar W^{n-1}_{p,q} \Omega^k \Box_g \varphi  + \text{l.o.t.},
\end{align*}
and $\bar W^l_{p,q} \in \{L^i \Lb^j | i+j=l, \, i \leq p, \, j \leq q\}$.

Denote $Z^{n}_{p,q} := L^{p} Y^{q}$, $p+q =n$ and $Z \in \{L, Y\}$. We have
\begin{align*}
\Box_g Z^n_{p,q} \varphi_k &= Z^n_{p,q} \Omega^k \Box_g \varphi + Z_n (\varphi_k) \pm Z_{\leq n-1} (\varphi_k),
\end{align*}
where 
\begin{align*}
Z_n (\varphi_k)& = (q-p) \frac{2m}{r^2} Y  Z^{n}_{p, q} \varphi_k, \\ 
 Z_{\leq n-1} (\varphi_k) &=   \sum_{i\leq n-1} \frac{1}{r^{n+1-i}}(\bar Z^i_{p,q} Y \varphi_k \pm \bar Z^i_{p,q}  L\varphi_k  \pm  r\laplacianslash \bar Z^i_{p,q} \varphi_k) + \text{l.o.t.},
\end{align*}
and $\bar Z^l_{p,q} \in \{L^i Y^j | i+j=j, \, i \leq p, \, j \leq q\}$. 

Here l.o.t. denotes lower order terms in terms of  derivatives and $r$ weight.
\end{lemma}

\subsection{Null condition}
We refer to Definition \ref{def-null-condition}  for the null condition.
There are several obvious examples of quadratic null forms: $Q_0= g^{\mu \nu} D_\mu \psi D_\nu \psi;$ $Q_{\mu \nu} = D_\mu \psi D_\nu \psi - D_\nu \psi D_\mu \psi.$ Without confusion, we denote these null forms by $Q$ as well. Given any vector field $X$, let $Q \circ X(D\psi_1, D \psi_2) = Q(DX\psi_1, D \psi_2) +  Q(D\psi_1, D X\psi_2)$ and $[Q, X] = XQ - Q \circ X$.  One has then
\begin{subequations}
\begin{align}
|[\Omega, Q] (D\psi_1, D \psi_2)| & \lesssim   | D \psi_1 \Db \psi_2| + | D \psi_2 \Db \psi_1|, \label{eq-null-condition-comm-omega}\\
| [D, Q](D\psi_1, D \psi_2)| & \lesssim  r^{-1} \left( | D \psi_1 \Db \psi_2| + | D \psi_2 \Db \psi_1|\right). \label{eq-null-condition-comm-D}
\end{align}
\end{subequations}
Implied by \eqref{eq-null-condition-comm-D}, there is
\begin{equation*}
|[u \Lb, Q] (D\psi_1, D \psi_2)| \lesssim |u| r^{-1} \left( | D \psi_1 \Db \psi_2| + | D \psi_2 \Db \psi_1|\right),
\end{equation*}
and hence
\begin{equation}\label{eq-null-uLb-comm}
|[u \Lb, Q] (D\psi_1, D \psi_2)| \lesssim  | D \psi_1 \Db \psi_2| + | D \psi_2 \Db \psi_1|, \quad \text{in} \,\,\, \R_1,
\end{equation}
since $r \sim| u|$ in $\R_1$.

By the formula \eqref{commutates}, we can calculate that for a general $Q$ satisfying the Definition \ref{def-null-condition}, \eqref{eq-null-condition-comm-omega}-\eqref{eq-null-condition-comm-D} and \eqref{eq-null-uLb-comm} are always valid. Theses are the inequalities needed in practice.

\section{Global existence for the scattering problem}\label{sec-past}
In this section, we will prove that the solution exists from the past event horizon and the past null infinity up to any finite $u=u_1 \sim 1$. Without lost of generality, we assume that $u_1=1$ in the following discussion.  
Hence, we shall allow us to abuse the notation $\R_1$ a little bit and let $\R_1$ be the null strip $\D^+(\mathcal{I}^-) \cap \{u \leq 1\} \cap \{0 \leq \ub \leq \delta \}$, if there is no confusion. Then in $\R_1$, $u \leq 1$, $r^\ast = \ub - u\geq -1$, and hence $\langle r^\ast \rangle \sim r$, $t \sim \langle u \rangle \sim r $. Moreover, $\R_1$ is away from the horizon and the  photon sphere $r=3m$.  We remind ourselves that $\langle u\rangle = \sqrt{|u|^2 +1}$. And we simplify the notation $C_u^{[0, \ub]}$ by $C_u$, $\Cb_{\ub}^{[u_0, u]}$ by $\Cb_{\ub}$, where $-\infty \leq u_0 \leq u \leq 1$.

\subsection{Initial data in $\R_1$}\label{sec-data-1}
We refer to \cite{Christodoulou-09} for the short pulse data, and also \cite[Section 3]{Wang-Yu-16} for such data in the setting of wave equation. 

Let $-\infty \leq u_0 \leq 0$ and $C_{u_0}= \{u = u_0\}$ be the initial outgoing light-cone. And $\mathcal{H}^- = \{\ub = -\infty\}$ denotes the past event horizon. The data will be imposed on $\mathcal{H}^- \cup C_{u_0}$.
Initially, we require that the data of \eqref{Main Equation} verify:
\begin{equation}\label{eq-data-trivial}
\varphi \equiv 0, \quad \text{on} \quad \mathcal{H}^- \cup C_{u_0}^{[-\infty, 0]}.
\end{equation}
Consequently, 
we can extend the solution of \eqref{Main Equation} to be trivial in the region $\D^+(\mathcal{H}^-) \cap \D^+(C_{u_0}) \cap \D^-(\Cb_0)$ , i.e., $\varphi \equiv 0$ in $\{ \ub \leq 0, u \geq u_0\}$. Secondly, we set
\begin{equation}\label{data-short-plus-1}
\varphi|_{C_{u_0}^{[0, \delta]}} = \frac{\delta^{\frac{1}{2}}}{|u_0|} \psi_0 \left(\ub/\delta, \theta, \phi \right),
\end{equation}
where $\psi_0: [0,1] \times S^2 \rightarrow \mathbb{R}$ is a smooth, compactly supported function. We remark that the factor $\frac{1}{|u_0|}$ manifests the decay of linear wave.

The data \eqref{data-short-plus-1} immediately entail that for all $l, k \in \mathbb{N}$,
\begin{equation}\label{initial-data-R1}
\begin{split}
|u_0| \delta^{\frac{1}{2}} \|L^{l+1} \Omega^k \varphi \|_{L^\infty (C_{u_0})} + \|L^{l+1} \Omega^k \varphi \|_{L^2 (C_{u_0})} & \lesssim \delta^{-l} ,\\
 |u_0|^{2}\|\nablaslash L^l \Omega^k \varphi \|_{L^\infty (C_{u_0})} +  |u_0| \delta^{-\frac{1}{2}} \|\nablaslash L^l \Omega^k \varphi \|_{L^2 (C_{u_0})} & \lesssim  \delta^{\frac{1}{2} -l}.
\end{split}
\end{equation}
Following \cite{Wang-Yu-16}, we commute \eqref{Main Equation} with $\Omega^k$, rewrite it as an ODE for $\Lb \Omega^k \varphi$ and integrate along $L$ to derive 
\begin{equation*}
\begin{split}
\|\Lb \Omega^k \varphi \|_{L^\infty (C_{u_0})} \lesssim \delta^{\frac{1}{2}} |u_0|^{-2}.
\end{split}
\end{equation*}
We expect that these initial informations will be preserved during the evolution of wave equation. For this purpose, we should relax $\nablaslash$ a little bit, namely, we only expect that the estimate $\|\nablaslash L^l \Omega^k \varphi \|_{L^2 (C_{u_0})} \lesssim \delta^{\frac{1}{2} -l}$, rather than the original $\|\nablaslash L^l \Omega^k \varphi \|_{L^2 (C_{u_0})} \lesssim \delta^{1 -l} |u_0|^{-1}$ in \eqref{initial-data-R1}, propagates along the flow of \eqref{Main Equation}. This will be reflected in the definitions of  energies $E_k(u, \ub)$ \eqref{def-energy-E-R1} and ${}^LF_{1+k}(u, \ub)$ \eqref{def-F-L-R1} below.

\subsection{Bootstrap argument in $\mathcal{R}_1$} 
To conduct the energy estimates in $\R_1$, we need the commutators: $L, \Lb, \, \Omega$ and $\tilde S$, where 
\begin{equation}\label{def-tilde-S}
\tilde S : = \langle u \rangle \Lb.
\end{equation}
Then a family of energy norms are defined as follows.
Given any fixed number $N\in \mathbb{N}$, $N\geq6$, 
we define  for $0\leq l\leq N$,
\begin{subequations}
\begin{align}
E_l(u,\ub) &= \|L \varphi_l\|^2_{L^2(C_u)} + \delta^{-1}  \| \nablaslash \varphi_l\|^2_{L^2(C_u)}, \label{def-energy-E-R1}\\
\Eb_l(u,\ub) &=\| \nablaslash \varphi_l\|^2_{L^2(\Cb_{\ub})} + \delta^{-1} \| \Lb \varphi_l\|^2_{L^2(\Cb_{\ub})}. \label{def-energy-Eb-R1}
\end{align}
\end{subequations}
And for $0\leq k\leq N-1$, 
 \begin{subequations}
\begin{align}
{}^{L}F_{1+k}(u,\ub) &=\delta^2 \|L^2 \varphi_k\|^2_{L^2(C_u)}+\delta \| \nablaslash L\varphi_k\|^2_{L^2(C_u)}, \label{def-F-L-R1} \\
{}^{L }\Fb_{1+k}(u,\ub) &=\delta^2 \| \nablaslash L \varphi_k\|^2_{L^2(\Cb_{\ub})} + \delta  \| \Lb L \varphi_k\|^2_{L^2(\Cb_{\ub})}, \label{def-Fb-L-R1}\\
{}^{\tilde S }F_{1+k}(u,\ub) &= \delta^{-1} \| \nablaslash  \tilde S \varphi_k\|^2_{L^2(C_u)}, \label{def-F-S-R1}\\
  {}^{\tilde S }\Fb_{1+k}(u,\ub) &= \delta^{-1}\| \Lb  \tilde S \varphi_k\|^2_{L^2(\Cb_{\ub})}, \label{def-Fb-S-R1}
 \end{align}
\end{subequations}
and
\begin{equation}\label{def-energy-E-T-R1}
{}^{t}{F}_{1+k}(u,\ub) = \delta^{-2}\langle u \rangle^2 \|\Lb \varphi_k\|^2_{L^2(C_{u})} + \langle u \rangle^2 \|\Lb L \varphi_k\|^2_{L^2(C_{u})}.
\end{equation}
Equivalently, 
\begin{equation*}
{}^{t}{F}_{1+k}(u,\ub) = \delta^{-2} \|\tilde S \varphi_k\|^2_{L^2(C_{u})} +  \|L \tilde S \varphi_k\|^2_{L^2(C_{u})}.
\end{equation*}
We also make the simplification for the flux \eqref{def-F-L-R1}-\eqref{def-Fb-S-R1}
\begin{align*}
F_{1+k}(u,\ub) &={}^{L}F_{1+k}(u,\ub) +{}^{\tilde S }F_{1+k}(u,\ub),\\
\Fb_{1+k}(u,\ub) &={}^{L }\Fb_{1+k}(u,\ub) + {}^{\tilde S }\Fb_{1+k}(u,\ub).
\end{align*}

\begin{remark}\label{rk-def-energy-R1}
Intuitively, one would expect ${}^{\tilde S }F_{1+k}(u,\ub)$ and  ${}^{\tilde S }\Fb_{1+k}(u,\ub)$ resemble ${}^{L}F_{1+k}(u,\ub)$ and  ${}^{L}\Fb_{1+k}(u,\ub)$, i.e., replacing \eqref{def-F-S-R1}-\eqref{def-Fb-S-R1} by
 \begin{subequations}
\begin{align}
& \|L \tilde S \varphi_k\|^2_{L^2(C_u)} +
 \delta^{-1} \| \nablaslash  \tilde S \varphi_k\|^2_{L^2(C_u)}, \label{def-F-S-R1-fake}\\
& \| \nablaslash \tilde S \varphi_k\|^2_{L^2(\Cb_{\ub})} +
 \delta^{-1}\| \Lb  \tilde S \varphi_k\|^2_{L^2(\Cb_{\ub})}. \label{def-Fb-S-R1-fake}
 \end{align}
\end{subequations}
However, we note that for $k \leq N-1$, $\|L \tilde S \varphi_k\|^2_{L^2(C_u)}$ is already covered by  ${}^{t}{F}_{1+k}(u,\ub)$ \eqref{def-energy-E-T-R1}, and $\| \nablaslash \tilde S \varphi_k\|^2_{L^2(\Cb_{\ub})}$ is also covered by $\delta^{-1} \| \Lb \varphi_l\|^2_{L^2(\Cb_{\ub})}, \, l \leq N$ in $\Eb_{l}(u,\ub)$ \eqref{def-energy-Eb-R1}. Hence, $\|L \tilde S \varphi_k\|^2_{L^2(C_u)}$ and  $\| \nablaslash \tilde S \varphi_k\|^2_{L^2(\Cb_{\ub})}$ are redundant in \eqref{def-F-S-R1-fake}-\eqref{def-Fb-S-R1-fake} and we use \eqref{def-F-S-R1}-\eqref{def-Fb-S-R1} for the sake of simplicity.
\end{remark}

With these definitions of energy norms, the data \eqref{eq-data-trivial}-\eqref{data-short-plus-1} satisfy
\begin{subequations}
\begin{align}
E_{l}(u_0,\delta)+  F_{1+k}(u_0,\delta)+ {}^{t}{F}_{1+k}(u_0,\delta)& \leq I^2_{N+1}, \label{initial bound-1}\\
\Eb_l(u,\ub) +  \Fb_{1+k}(u,\ub) & =0, \label{initial bound-2}
\end{align}
\end{subequations}
where \eqref{initial bound-1} has relaxed the initial bound \eqref{initial-data-R1} and $0\leq l \leq N$, $0\leq k\leq N-1$. Here $I_{N+1} \in \mathbb{R}^{+}$ is a universal constant manifesting the initial norm. The subindex $N+1$ in $I_{N+1}$ denotes the number of derivatives used in the energy norms.

The energy estimates in $\R_1$ will be based on a standard bootstrap argument. Fix $u_0 \leq u^*\leq 1$ and $0\leq \ub^* \leq \delta$. We assume that there is a large constant $M$ to be determined, such that the solution of \eqref{Main Equation} defined on the domain $\mathcal{D}^{u_0, u^\ast}_{0,\ub^\ast} \subset \mathcal{R}_1 $ enjoys the estimate
\begin{equation}\label{bootstrap assumption in R1}
 E_l(u^{\prime},\ub^\prime) + \Eb_l(u^{\prime},\ub^{\prime})+ F_{1+k}(u^{\prime},\ub^{\prime}) + \Fb_{1+k}(u^{\prime},\ub^{\prime})+ {}^{t}{F}_{1+k}(u^\prime,\ub^\prime) \leq M^2,
\end{equation}
for all $\ub^{\prime} \in [u_0, u]$ and $\ub^{\prime} \in [0, \ub]$, where $l \leq N,$ $k\leq N-1$,  and $u\leq u^*$ and all $0\leq \ub \leq \ub^*$. At the end of the current section, we aim to show that  the $M^2$ in \eqref{bootstrap assumption in R1} can be actually replaced by $\frac{M^2}{2}$, and the choice of $M$ depends only on the norm of the initial data but not the wave profile $\varphi$. Then the bootstrap argument will be closed and it yields the following estimates: There is a constant $C(I_{N+1})$ depending only on $I_{N+1}$ (in particular, not on $\delta$ and $u_0$), so that for all $u\leq 1$ and all $0\leq \ub \leq \delta$, we have
\begin{equation}\label{main estimates in R1}
E_l(u,\ub) + \Eb_l(u,\ub)+  F_{1+k}(u,\ub) +  \Fb_{1+k}(u,\ub)+{}^{t}{F}_{1+k}(u,\ub) \leq C(I_{N+1}).
\end{equation}

We first collect some preliminary estimates which follow from the bootstrap assumption \eqref{bootstrap assumption in R1} and the Sobolev inequalities.
\begin{proposition}\label{proposition L infinity and L4 estimates in R1} 
The bootstrap assumption \eqref{bootstrap assumption in R1} leads to the following estimates in $\R_1$,
\begin{align*}
\delta^{\frac{1}{2}}\langle u \rangle \|L\varphi_j\|_{L^\infty(\mathcal{R}_1)} + \delta^{-\frac{1}{4}}\langle u \rangle^{\frac{3}{2}}  \|\Db\varphi_j\|_{L^\infty(\mathcal{R}_1)} & \lesssim M, \quad 0\leq j\leq  N-2, \\
\delta^{\frac{1}{2}}\langle u \rangle^{\frac{1}{2}}\| L\varphi_k \|_{L^{4}(S_{\ub,u})} +\delta^{-\frac{1}{4}}\langle u \rangle \|\Db\varphi_k\|_{L^4(S_{\ub,u})} & \lesssim M, \quad 0\leq k\leq  N-1.
\end{align*}
\end{proposition}

\begin{proof}
The proof is based on the Sobolev inequalities on $C_u$ and $S_{\ub, u}$, \eqref{Sobolev Inequlities-S2}-\eqref{Sobolev Inequlities-Cu}. 
For simplicity, we will only address the case for $\Lb\varphi_k$. By \eqref{Sobolev Inequlities-Cu}, there is, for $k \leq N-1$,
\begin{align*}
 \langle u \rangle^{\frac{1}{2}} \| \Lb \varphi_k \|_{L^{4}(S_{\ub,u})} &\lesssim \| L \Lb \varphi_k \|^{\frac{1}{2}}_{L^{2}(C_{u})} (\| \Lb \varphi_k \|^{\frac{1}{2}}_{L^{2}(C_{u})} + \|\Omega \Lb \varphi_k \|^{\frac{1}{2}}_{L^{2}(C_{u})}) \\
 &\lesssim \langle u \rangle^{-\frac{1}{2}} M^{\frac{1}{2}} \cdot  \delta^{\frac{1}{4}} M^{\frac{1}{2}} \lesssim \delta^{\frac{1}{4}}  \langle u \rangle^{-\frac{1}{2}} M,
 \end{align*}
where we note that for $k \leq N-1$, $  \langle u \rangle \| L \Lb \varphi_k \|_{L^{2}(C_{u})}$ is controlled by ${}^tF_{1+k} (u, \ub)$, while the bound of $\|\Omega \Lb \varphi_k \|_{L^{2}(C_{u})} \sim \|\nablaslash \tilde S \varphi_k \|_{L^{2}(C_{u})}$ should be related to  the bootstrap assumption for ${}^{\tilde S}F_{1+k}$ \eqref{def-F-S-R1}.
For the $L^\infty$ estimate, we turn to \eqref{Sobolev Inequlities-S2}. Then for $j \leq N-2$,
\begin{align*}
\|\Lb \varphi_j \|_{L^\infty(S_{\ub, u})} &\lesssim r^{-\frac{1}{2}} (\| \Lb \varphi_j \|_{L^{4}(S_{\ub,u})} +\| \Omega \Lb \varphi_j \|_{L^{4}(S_{\ub,u})}) \lesssim \delta^{\frac{1}{4}}  \langle u \rangle^{-\frac{3}{2}} M.
\end{align*}
\end{proof}

\begin{remark}\label{bt-L4-Linfty-Lb-R1}
In contrast to \cite{Wang-Yu-16}, we here use the Sobolev inequality on $C_u$ instead of that on $\Cb_{\ub}$ for the $L^\infty$ and $L^4$ estimates of $\Lb \varphi_k$. In doing this, we have to employ the weighted commutator vector field $\tilde S = \langle u \rangle \Lb$ (rather than $\Lb$) to ensure good decay rates for $\Lb \varphi_k$. In other words, we introduce the energy norms ${}^{\tilde S} F_{1+k}, {}^{\tilde S} \Fb_{1+k}$ and ${}^tF_{1+k}$ rather than ${}^{\Lb} F_{1+k}, {}^{\Lb} \Fb_{1+k}$ as in \cite{Wang-Yu-16}.

There are the following stronger estimates for lower order derivatives of $\Lb\varphi_k$: when $0\leq k_1\leq  N-3$ and $0\leq k_2\leq  N-2$,
\begin{equation*}
 \delta^{-\frac{1}{2}}\langle u \rangle^{2}\|\Lb\varphi_{k_1}\|_{L^\infty(\mathcal{R}_1)}+ \delta^{-\frac{1}{2}}\langle u \rangle^{\frac{3}{2}}\|\Lb\varphi_{k_2} \|_{L^{4}(S_{\ub,u})}\lesssim M.
\end{equation*}
Compared to the top order case, the difference lies in that for the lower order case $0\leq k_2\leq  N-2$, $$\| \Lb \varphi_{k_2} \|^{\frac{1}{2}}_{L^{2}(C_{u})} + \|\Omega \Lb \varphi_{k_2} \|^{\frac{1}{2}}_{L^{2}(C_{u})} \lesssim \delta^{\frac{1}{2}} M^{\frac{1}{2}}   \langle u \rangle^{-\frac{1}{2}},$$ where we use the bootstrap assumption for ${}^tF_{1+k}(u, \ub), \, k \leq N-1$, instead of that for ${}^{\tilde S}F_{1+k}, \, k \leq N-1$ (in the top order case). It will be more obvious if we list these two bootstrap assumptions: ${}^tF_{1+k}(u, \ub) \leq M^2$ gives $ \|\Lb \varphi_k\|_{L^2(C_{u})} \leq \delta M \langle u \rangle^{-1},$ $k \leq N-1$; while ${}^{\tilde S}F_{1+k} (u, \ub) \leq M^2$, i.e., $ \| \nablaslash \tilde S \varphi_k\|_{L^2(C_u)} \leq  \delta^{\frac{1}{2}} M$, $k \leq N-1$, leads to $ \| \Lb \varphi_l\|_{L^2(C_u)} \lesssim  \delta^{\frac{1}{2}} M$, $l \leq N$. 
We can see that, there is an extra $\delta^{\frac{1}{2}}  \langle u \rangle^{-1}$ in the estimates for the lower order cases. However, the weak estimates presented in Proposition \ref{proposition L infinity and L4 estimates in R1} is good enough for our proof.
\end{remark}

\subsection{Energy estimates in $\mathcal{R}_1$}\label{sec-energy-estimate-R1}
\subsubsection{The multiplier in $\mathcal{R}_1$}\label{sex-multiplier-R1}
Consider the multiplier $\xi_1:= \eta L +  \delta^{-1}  \Lb$. That is, choose $f_1= \eta$, $f_2 = \delta^{-1}$. In view of \eqref{current of fL},  we have,
\begin{align}
\p_u f_1g^{u\ub}  |L \psi|^2 &=  \frac{m}{r^2} |L \psi|^2 >0, \label{eq-positive-L-R1} \\
-\frac{1}{2} \left( \p_u f_2 - \frac{\mu f_2}{r} \right)   |\nablaslash \psi|^2 &= \delta^{-1} \frac{m}{r^2}    |\nablaslash \psi|^2 >0. \label{eq-positive-angular-R1}
\end{align}
We apply the scheme in Section \ref{Energy estimates scheme} to the wave equation for $\psi$,  the energy identity \eqref{energy-identity}-\eqref{current of fL}  yields that, for  $u_0 \leq u \leq 1$,
\begin{align*}
&\int_{C_u}  ( |L \psi|^2 + \delta^{-1} |\nablaslash \psi|^2)   \di \mu_{C_u} + \int_{\Cb_{\ub}} (  |\nablaslash \psi|^2+\delta^{-1} |\Lb \psi|^2) \di \mu_{\Cb_{\ub}} \\
\quad &+\doubleint_{\Dp} \langle u \rangle^{-2} ( |L \psi|^2+ \delta^{-1} |\nablaslash \psi|^2 )   \di \mu_{\D}\\
\lesssim {}& I_{1}^2 (\psi) + \mathcal{C}^k_1(\psi)+ \mathcal{C}^k_2(\psi)+ \mathcal{F}_L^k(\psi) +  \mathcal{F}_{\Lb}^k(\psi),
\end{align*}
where we note that $r \sim \langle u \rangle, \, \eta \sim 1$ in $\R_1$ and $I^2_1(\psi)$ denotes the first order initial energy of $\psi$, and
\begin{align*}
 \mathcal{C}^k_1(\psi)& = \doubleint_{\Dp}\langle u \rangle^{-1} |\nablaslash \psi|^2  \di \mu_{\D}, \quad \quad \mathcal{C}^k_2(\psi)  = \doubleint_{\Dp} \delta^{-1}\langle u \rangle^{-1}|L\psi \Lb \psi|   \di \mu_{\D}, \\
 \mathcal{F}_L^k(\psi) &= \doubleint_{\Dp} |L \psi| |\Box_g\psi|  \di \mu_{\D}, \quad
 \quad  \mathcal{F}_{\Lb}^k(\psi) = \doubleint_{\Dp}  \delta^{-1} |\Lb  \psi| |\Box_g\psi|  \di \mu_{\D}.
\end{align*}
For $\mathcal{C}^k_1(\psi), \,  \mathcal{C}^k_2(\psi)$,
\begin{align*}
 \mathcal{C}^k_1(\psi) \lesssim & \int^{\ub}_{0}  \di \ub^{\prime}\int_{\Cb_{\ub^{\prime}}}  |\nablaslash \psi|^2  \di \mu_{\Cb_{\ub}}, \\
\mathcal{C}^k_2(\psi) \lesssim &  \int_{u_0}^u  \langle u^\prime \rangle^{-2}  \di u^\prime \int_{C_{u^\prime}} | L \psi|^2  \di \mu_{C_u} +\int^{\ub}_{0} \delta^{-1}\di \ub^{\prime}\int_{\Cb_{\ub^{\prime}}}\delta^{-1}  |\Lb \psi|^2\di \mu_{\Cb_{\ub}},
\end{align*}
where both of them can be handled by the Gr\"{o}nwall's inequality, see Lemma \ref{lemma-Gronwall}.
For $ \mathcal{F}_L^k(\psi), \, \mathcal{F}_{\Lb}^k(\psi)$, 
\begin{align*}
 \mathcal{F}_L^k(\psi) \lesssim & \doubleint_{\Dp}\langle u \rangle^{\frac{3}{2}}  |\Box_g\psi|^2  \di \mu_{\D} + \int_{u_0}^u  \langle u^\prime \rangle^{-\frac{3}{2}}    \di u^\prime \int_{C_{u^\prime}} | L \psi|^2  \di \mu_{C_u}, \\ 
\mathcal{F}_{\Lb}^k(\psi) \lesssim & \doubleint_{\Dp}   |\Box_g\psi|^2  \di \mu_{\D} +\int^{\ub}_{0} \delta^{-1}\di \ub^{\prime}\int_{\Cb_{\ub^{\prime}}}\delta^{-1}   |\Lb \psi|^2 \di \mu_{\Cb_{\ub}}.  
\end{align*}
Hence, after the Gr\"{o}nwall's inequality, we derive the energy inequality: for $u_0 \leq u \leq 1$, $0 \leq \ub \leq \delta$,
\begin{equation}\label{energy-ineq-phi-R1-summary}
\begin{split}
&\quad \int_{C_u}  ( |L \psi|^2 + \delta^{-1} |\nablaslash \psi|^2) \di \mu_{C_u} + \int_{\Cb_{\ub}} ( |\nablaslash \psi|^2+\delta^{-1} |\Lb \psi|^2) \di \mu_{\Cb_{\ub}} \\
& \lesssim  I_{1}^2 (\psi)
+\doubleint_{\Dp}\langle u \rangle^{\frac{3}{2}}  |\Box_g\psi|^2 \di \mu_{\D}.
\end{split}
\end{equation}
We should always remind ourselves that $\eta \sim 1$ in $\R_1$. Note also that, if we take $\xi=\delta^{-1} \Lb$ as the multiplier, then another energy inequality will be derived 
\begin{equation}\label{energy-ineq-Lb-phi-R1-summary}
\quad \int_{C_u}  \delta^{-1} |\nablaslash \psi|^2 \di \mu_{C_u} + \int_{\Cb_{\ub}} \delta^{-1} |\Lb \psi|^2 \di \mu_{\Cb_{\ub}} \lesssim  I_{1}^2 (\psi)
+\doubleint_{\Dp} |\Box_g\psi|^2 \di \mu_{\D}.
\end{equation}

\subsubsection{Energy estimates for $E_{k}(u,\ub), \Eb_{k}(u,\ub),$ $k \leq N-1$}\label{sec-energy-esti-Ek-Ebk}
We take $\psi = \varphi_k$, $k \leq N-1$ in \eqref{energy-ineq-phi-R1-summary}, then
\begin{equation}\label{energy-ineq-varphi_k-R1}
 E_{k}(u,\ub)+ \Eb_{k}(u,\ub) 
\lesssim I_{N}^2  + \doubleint_{\Dp}\langle u \rangle^{ \frac{3}{2}}  |\Box_g\varphi_k|^2 \di \mu_{\D},
\end{equation}
where  by the null condition, the spacetime integral can be decomposed as: \\ $\iint_{\Dp}\langle u \rangle^{ \frac{3}{2}}  |\Box_g\varphi_k|^2 \di \mu_{\D} = H_1^k + \cdots + H^k_4$, with  $k_1+k_2\leq k \leq N-1, \, k_1\leq k_2$ and
\begin{subequations}
\begin{align}
H^k_{1} &= \doubleint_{\Dp}  \langle u \rangle^{\frac{3}{2}}  | D \varphi_{k_1}|^2  | \Lb \varphi_{k_2}|^2 \di \mu_{\D}; \label{eq-H-k-1-R1}\\
H^k_{2} &=  \doubleint_{\Dp}  \langle u \rangle^{\frac{3}{2}}   |  \Db\varphi_{k_1}|^2 |L \varphi_{k_2}|^2 \di \mu_{\D};\label{eq-H-k-2-R1} \\
H^k_{3} &=  \doubleint_{\Dp}  \langle u \rangle^{\frac{3}{2}}   |  \Db\varphi_{k_1}|^2  |\nablaslash  \varphi_{k_2}|^2\di \mu_{\D}; \label{eq-H-k-3-R1} \\
H^k_{4} &= \doubleint_{\Dp}  \langle u \rangle^{\frac{3}{2}}  | L\varphi_{k_1}|^2   |\nablaslash  \varphi_{k_2}|^2 \di \mu_{\D}. \label{eq-H-k-4-R1}
  \end{align}
\end{subequations}
Noting that, $ N\geq 6$, $k_1\leq [\frac{N}{2}] \leq N-3$, we can apply $L^\infty$ to $D \varphi_{k_1}$, see Proposition \ref{proposition L infinity and L4 estimates in R1}. Then by the bootstrap assumptions,
\begin{align*}
H^k_{1}& \lesssim \int^{\ub}_{0}  \di \ub^\prime \int_{\Cb_{\ub^{\prime}}}  \langle u \rangle^{\frac{3}{2}} \delta^{-1} \langle u \rangle^{-2}M^2 | \Lb\varphi_{k_2}| ^2\di \mu_{\Cb_{\ub}}  \lesssim \delta  M^4.
\end{align*}
In an analogous fashion, there is
\begin{align*}
H^k_{2} + H^k_{3} & \lesssim \int^{u}_{u_0}  \langle u^\prime \rangle^{\frac{3}{2}} \delta^{\frac{1}{2}} \langle u^\prime \rangle^{-3} M^2 \left(\|L\varphi_{k_2}\|^2_{L^2(C_{u^\prime})} + \|\nablaslash \varphi_{k_2}\|^2_{L^2(C_{u^\prime})} \right) \di u^\prime \\
&\lesssim \delta^{\frac{1}{2}} \langle u \rangle^{-\frac{1}{2}} M^4.
\end{align*}
For the last term $H^k_{4} $, we should notice that, $k \leq N-1$. Thus, we are allowed to manipulate $L^4, L^4, L^4, L^4$ (instead of $L^\infty, L^\infty, L^2, L^2$) on the four factors and gain some positive power of $\delta$,
\begin{equation}\label{esti-Qk-4N-1}
\begin{split}
H^k_{4} &\lesssim\int_{u_0}^{u}\int_{0}^{\ub}  \langle u^\prime \rangle^{\frac{3}{2}} \|L\varphi_{k_1}\|^2_{L^4(S_{\ub^{\prime},u^{\prime}})}\| \nablaslash\varphi_{k_2}\|^2_{L^4(S_{\ub^{\prime},u^{\prime}})}\di  \ub^{\prime}\di u^{\prime}\\
&\lesssim  \delta^{\frac{1}{2}}\langle u \rangle^{-\frac{1}{2}}M^4.
\end{split}
\end{equation}
We remark that the estimate \eqref{esti-Qk-4N-1} is not valid in the top order case: $k=N$.

These results are summarized as,
\begin{equation}\label{energy-esti-Qk}
\doubleint_{\Dp}\langle u \rangle^{ \frac{3}{2}}|\Box_g \varphi_k|^2\di \mu_{\D}\lesssim \delta^{\frac{1}{2}} M^4, \quad k \leq N-1.
\end{equation}
Therefore, we infer that 
 \begin{equation}\label{energy-esti-Ek}
 E_k(u,\ub)+\Eb_k(u,\ub) \lesssim I^2_{N}+\delta^{\frac{1}{2}} M^4, \quad k \leq N-1.
 \end{equation}

With of the improved $E_k(u, \ub), \, k \leq N-1$ \eqref{energy-esti-Ek}, we can proceed to ${}^LF_{1+k}(u,\ub)$, ${}^{L}\Fb_{1+k}(u,\ub), \, k \leq N-1$. 
\subsubsection{Energy estimates for ${}^LF_{1+k}(u,\ub)$, ${}^{L}\Fb_{1+k}(u,\ub), \, k \leq N-1$}
In this section, we take $\psi=\delta L\varphi_k, \, k \leq N-1$ in \eqref{energy-ineq-phi-R1-summary} to obtain the following energy inequality,
\begin{equation}\label{energy-ineq-Lvarphi-R1}
{}^LF_{1+k}(u,\ub) + {}^{L}\Fb_{1+k}(u,\ub) 
 \lesssim I_{N+1}^2 + \doubleint_{\Dp}\delta^2 \langle u \rangle^{ \frac{3}{2}} |\Box_gL\varphi_k|^2\di \mu_{\D},
\end{equation}
where the source term is split as: $ \iint_{\Dp}\delta^2 \langle u \rangle^{ \frac{3}{2}} |\Box_gL\varphi_k|^2\di \mu_{\D} =  {}^L\mathcal{G}^k + {}^L\mathcal{H}^k + {}^L\mathcal{J}^k +{}^LW^k$, with $k_1+k_2\leq k \leq N-1, \, k_1\leq k_2$ and
\begin{align*}
 {}^L\mathcal{G}^k&= \doubleint_{\Dp}\delta^2\langle u\rangle^{ \frac{3}{2}}|Q(\p L\varphi_{k_1},\p \varphi_{k_2})|^2\di \mu_{\D}, \\
  {}^L\mathcal{H}^k &=  \doubleint_{\Dp}\delta^2 \langle u \rangle^{ \frac{3}{2}}|Q(\p\varphi_{k_1},\p L \varphi_{k_2})|^2\di \mu_{\D}, \\
   {}^L\mathcal{J}^k &= \doubleint_{\Dp}\delta^2 \langle u \rangle^{ \frac{3}{2}}|Q(\p\varphi_{k_1},\p \varphi_{k_2})|^2 \di \mu_{\D}, \\
  {}^LW^k&
 =\doubleint_{\Dp}\delta^2 \langle u\rangle^{ \frac{3}{2}}|[\Box_g,L]\varphi_k|^2\di \mu_{\D}.
\end{align*}
In what follows, we will focus on estimating these four terms.

At first,  \eqref{energy-esti-Qk} tells that ${}^L\mathcal{J}^k \lesssim \delta^{2+\frac{1}{2}} M^4$. 

Next, for ${}^L\mathcal{G}^k$,  we make the splitting: $ {}^L\mathcal{G}^k= {}^LG^k_{1} +{}^LG^k_{2} +{}^LG^k_{3}$, where for $k_1+k_2\leq k \leq N-1, \, k_1\leq k_2$,
\begin{align*}
{}^LG^k_{1}  &= \doubleint_{\Dp}\delta^2 \langle u \rangle^{\frac{3}{2}}  | D\varphi_{k_2}|^2  |\Lb L\varphi_{k_1}|^2 \di \mu_{\D};\\
{}^LG^k_{2}  &= \doubleint_{\Dp} \delta^2  \langle u \rangle^{\frac{3}{2}} | D\varphi_{k_2}|^2 |\nablaslash L\varphi_{k_1}|^2 \di\mu_{\D};\\
{}^LG^k_{3}  &=  \doubleint_{\Dp}\delta^2  \langle u \rangle^{\frac{3}{2}}  | \Db\varphi_{k_2}|^2  |L^2\varphi_{k_1}|^2 \di\mu_{\D}.
\end{align*}
Since $k_1\leq [\frac{N}{2}] \leq N-3, \, k_2 \leq N-1$, we can apply $L^4$ to all of the four factors in ${}^LG^k_i$. By Proposition \ref{proposition L infinity and L4 estimates in R1}, $ \| D\varphi_{k_2}\|_{L^4(S_{\ub,u})} \lesssim \delta^{- \frac{1}{2}}M  \langle u \rangle^{-\frac{1}{2}},$
then 
\begin{align*}
{}^LG^k_{1} 
&\lesssim \int_{u_0}^{u}\int_{0}^{\ub} \delta M^2  \langle u^\prime \rangle^{-1+\frac{3}{2}} \cdot \|\Lb L \varphi_{k_1}\|^2_{L^4(S_{\ub^{\prime},u^{\prime}})}\di  \ub^{\prime} \di u^{\prime}\\
&\lesssim \int_{u_0}^{u}\delta M^2   \langle u^\prime \rangle^{\frac{1}{2}} \sum_{k_1 \leq i \leq k_1+1} \langle u^\prime \rangle^{-1}\|\Lb L \varphi_{i}\|^2_{L^2(C_{u^{\prime}})}\di u^{\prime}\\
&\lesssim\delta \langle u \rangle^{-\frac{3}{2}} M^4, \quad k_1 \leq N-3.
\end{align*}
Here we have used the bootstrap assumption for ${}^tF_{1+k}(u, \ub), \, k\leq N-1$ and the Sobolev inequality on the sphere $S_{\ub,u}$:
\begin{equation}\label{Sobolev-L2-L4-S2}
\|\phi\|_{L^4(S_{\ub,u})}\lesssim r^{-\frac{1}{2}}\|\phi\|_{L^2(S_{\ub,u})}+ r^{\frac{1}{2}}\|\nablaslash\phi\|_{L^2(S_{\ub,u})}.
\end{equation}
Similarly, there is
\begin{align*}
{}^LG^k_{2} 
&\lesssim \int_{u_0}^{u}\delta M^2   \langle u^\prime \rangle^{\frac{1}{2}} \sum_{k_1 \leq i \leq k_1+1} \langle u^\prime \rangle^{-1}\|\nablaslash L \varphi_{i}\|^2_{L^2(C_{u^{\prime}})}\di u^{\prime}\\
&\lesssim\delta \langle u \rangle^{-\frac{3}{2}} M^4, \quad k_1 \leq N-3,
\end{align*}
where we have used the bootstrap assumption for $E_l(u, \ub), \, l\leq N$. Finally, we come to ${}^LG^k_{3}$, noting that, by Proposition \ref{proposition L infinity and L4 estimates in R1}, $ \| \Db \varphi_{k_2}\|_{L^4(S_{\ub,u})} \lesssim \delta^{\frac{1}{4}}M  \langle u \rangle^{-1},$
\begin{align*}
{}^LG^k_{3} 
&\lesssim \int_{u_0}^{u}\int_{0}^{\ub} \delta^{\frac{5}{2}} M^2  \langle u^\prime \rangle^{-2+\frac{3}{2}} \cdot \|L^2 \varphi_{k_1}\|^2_{L^4(S_{\ub^{\prime},u^{\prime}})}\di  \ub^{\prime} \di u^{\prime}\\
&\lesssim \int_{u_0}^{u}  \delta^{\frac{5}{2}} M^2 \langle u^\prime \rangle^{-\frac{1}{2}} \sum_{k_1 \leq i \leq k_1+1} \langle u^\prime \rangle^{-1}\|L^2 \varphi_{i}\|^2_{L^2(C_{u^{\prime}})}\di u^{\prime}\\
&\lesssim \delta^{\frac{1}{2}} \langle u \rangle^{-\frac{1}{2}} M^4, \quad k_1 \leq N-3.
\end{align*}
Here the bootstrap assumption for ${}^LF_{1+k}(u, \ub), \, k\leq N-1$ is used.

For $  {}^L\mathcal{H}^k$, it can be decomposed into the following terms: for $k_1+k_2\leq k \leq N-1, \, k_1\leq k_2$,
\begin{align*}
{}^LH^k_{1} &= \doubleint_{\Dp}\delta^2  \langle u \rangle^{\frac{3}{2}}  | D \varphi_{k_1}|^2  | \Lb L\varphi_{k_2}|^2  \di \mu_{\D};\\
 {}^LH^k_{2} &=  \doubleint_{\Dp} \delta^2 \langle u \rangle^{\frac{3}{2}}   |  D \varphi_{k_1}|^2  |\nablaslash L\varphi_{k_2}|^2  \di \mu_{\D};\\
 {}^LH^k_{3} &= \doubleint_{\Dp}\delta^2  \langle u \rangle^{\frac{3}{2}}   |  \Db\varphi_{k_1}|^2 |L^2\varphi_{k_2}|^2  \di \mu_{\D}.
\end{align*}
Knowing that $k_1\leq [\frac{N}{2}] \leq N-3,\, k_2 \leq N-1$, we can apply $L^\infty, L^\infty, L^2, L^2$ to the four factors in ${}^LH^k_i$. By Proposition \ref{proposition L infinity and L4 estimates in R1}, $\|D \varphi_{k_1}\|_{L^\infty(S_{\ub,u})} \lesssim \delta^{-\frac{1}{2}}M  \langle u \rangle^{-1}$, and the bootstrap assumption on ${}^{t}F_{1+k} (u, \ub), \, k \leq N-1$, 
\begin{align*}
{}^LH^k_{1} &\lesssim \doubleint_{\Dp}\delta^2 \langle u^\prime \rangle^{\frac{3}{2}}  \| D \varphi_{k_1}\|^2_{L^\infty} |\Lb L \varphi_{k_2}|^2\di \mu_{\D}\\
&\lesssim \int_{u_0}^u \delta M^2  \langle u^\prime \rangle^{-\frac{1}{2}} \|\Lb L \varphi_{k_2}\|^2_{L^2(C_{u^\prime})}\di u^\prime \lesssim \delta \langle u \rangle^{-\frac{3}{2}} M^4, \quad k_2 \leq N-1.
\end{align*}
In the same way, taking advantage of the bootstrap assumption on $E_{l} (u, \ub), \, k \leq N$, there is
\begin{equation*}
{}^LH^k_{2} 
\lesssim \int_{u_0}^u \delta M^2  \langle u^\prime \rangle^{-\frac{1}{2}} \|\nablaslash L \varphi_{k_2}\|^2_{L^2(C_{u^\prime})}\di u^\prime \lesssim \delta \langle u \rangle^{-\frac{3}{2}} M^4, \quad k_2 \leq N-1.
\end{equation*}
As for ${}^LH^k_{3}$, noting that $\| \Db \varphi_{k_1}\|_{L^\infty(S_{\ub,u})} \lesssim \delta^{\frac{1}{4}}M  \langle u \rangle^{-\frac{3}{2}}$ (see Proposition \ref{proposition L infinity and L4 estimates in R1}), we have
\begin{align*}
{}^LH^k_{3} 
&\lesssim \int_{u_0}^u \delta^{\frac{5}{2}} M^2  \langle u^\prime \rangle^{-\frac{3}{2}} \|L^2 \varphi_{k_2}\|^2_{L^2(C_{u^\prime})}\di u^\prime \lesssim \delta^{\frac{1}{2}} \langle u \rangle^{-\frac{1}{2}} M^4, \quad k_2 \leq N-1,
\end{align*}
where the bootstrap assumption on ${}^LF_{1+k} (u, \ub), \, k \leq N-1$ is used.

In summary, we have obtained
\begin{equation*}
{}^L\mathcal{G}^k + {}^L\mathcal{H}^k  \lesssim\delta^{\frac{1}{2}} M^4, \quad k \leq N-1.
\end{equation*}

Next, we turn to ${}^LW^k $. In view of \eqref{commutates-Y-L-Box},
$[\Box_g,L]\varphi_k\sim \frac{1}{r^2}(L\varphi_k-\Lb\varphi_k)+\frac{1}{r} \laplacianslash\varphi_k + \frac{1}{r}\Box_g \varphi_k,$
 then 
 \begin{equation}\label{esti-W-R1}
\begin{split}
  {}^LW^k \lesssim & \doubleint_{\Dp} \delta^2 \langle u \rangle^{-4+\frac{3}{2}} \left( |L \varphi_{k}|^2+|\Lb \varphi_{k}|^2+|\nablaslash \varphi_{k+1}|^2 \right) \di \mu_{\D} \\
  &+\doubleint_{\Dp} \delta^2 \langle u \rangle^{-2+\frac{3}{2}}  |\Box_g \varphi_k|^2 \di \mu_{\D}, \quad k \leq N-1.
\end{split}
\end{equation}
By the improved result on $E_k(u, \ub), \, k \leq N-1$ \eqref{energy-esti-Ek}, there is 
$$\doubleint_{\Dp} \delta^2|u|^{-4+\frac{3}{2}}\left(|L \varphi_{k}|^2+|\Lb \varphi_{k}|^2\right) \di \mu_{\D} \lesssim \delta^2 I^2_{N}+\delta^{\frac{5}{2}} M^4, \quad k \leq N-1.$$
For the third term associated to $\nablaslash\varphi_{k+1}$, we should note that $k+1\leq N$, and hence it hits the top order derivative. Thus, we should make use of the bootstrap assumption on $E_N(u, \ub)$, then
$$\doubleint_{\Dp} \delta^2 \langle u \rangle^{-4+\frac{3}{2}}|\nablaslash \varphi_{k+1}|^2 \di \mu_{\D} \lesssim  \delta^3 \langle u \rangle^{-\frac{3}{2}}M^2, \quad k \leq N-1.$$
For the last term on the right hand of \eqref{esti-W-R1}, we appeal to \eqref{energy-esti-Qk}, then
\begin{equation*}
\doubleint_{\Dp}\langle u \rangle^{-2 + \frac{3}{2}}|\Box_g \varphi_k|^2\di \mu_{\D}\lesssim \delta^{\frac{1}{2}} M^4, \quad k \leq N-1.
\end{equation*}

Now, we conclude that
\begin{equation}\label{energy-esti-F-L-R1}
{}^LF_{1+k}(u,\ub) + {}^{L}\Fb_{1+k}(u,\ub) 
\lesssim I^2_{N+1}+\delta^{\frac{1}{2}} M^4, \quad k \leq N-1.
\end{equation}

Combining \eqref{energy-esti-F-L-R1} and the previously enhanced results \eqref{energy-esti-Ek}, we can improve the $L^\infty$ and $L^4$ estimates for $L\varphi_k$.
Define 
\begin{equation}\label{def-I-mathbb-R1}
\mathbb{I}^2_{k} := I^2_{k}+\delta^{\frac{1}{2}} M^4.
\end{equation}
\begin{proposition}\label{improved L infinity and L4 estimates in R1} 
 In Region $\mathcal{R}_1$, we have
\begin{align*}
 \delta^{\frac{1}{2}}\langle u \rangle \|L\varphi_p\|_{L^\infty(\mathcal{R}_1)} + \delta^{-\frac{1}{4}} \langle u \rangle^{\frac{3}{2}}\|\nablaslash\varphi_p\|_{L^\infty(\mathcal{R}_1)} & \lesssim \mathbb{I}_{N}, \quad 0\leq p\leq  N-3, \\
\delta^{\frac{1}{2}} \langle u \rangle^{\frac{1}{2}}\| L\varphi_q \|_{L^{4}(S_{\ub,u})} +\delta^{-\frac{1}{4}}\langle u \rangle  \|\nablaslash \varphi_q\|_{L^{4}(S_{\ub,u})} & \lesssim \mathbb{I}_{N}, \quad 0\leq q\leq  N-2.
\end{align*}
\end{proposition}

With the help of Proposition \ref{improved L infinity and L4 estimates in R1}, we can continue with the case of $E_{N}(u,\ub), \, \Eb_{N}(u,\ub)$ immediately.
\subsubsection{Energy estimates for $E_{N}(u,\ub), \, \Eb_{N}(u,\ub)$}\label{sec-energy-esti-EN-EbN}

In this top order case: $k=N$, we can proceed along the lines of Section \ref{sec-energy-esti-Ek-Ebk}, except that now \eqref{esti-Qk-4N-1} is not valid for $H^N_4$, due to the restriction of regularity. Alternatively, taking advantage of the improved result of Proposition \ref{improved L infinity and L4 estimates in R1} (noting that $k_1 \leq [\frac{N}{2}] \leq N-3$),
\begin{align*} 
H^N_4& \lesssim \sum_{k_1+k_2\leq N, k_1 \leq k_2} \doubleint_{\Dp}  \langle u \rangle^{\frac{3}{2}} |L\varphi_{k_1}|_{L^\infty}^2  |  \nablaslash \varphi_{k_2} |^2\di\mu_{\D}\\
&\lesssim \int_{0}^{\ub} \delta^{-1}\langle u \rangle^{ -\frac{1}{2}}\mathbb{I}^2_{N}  \di {\ub^\prime} \int_{\Cb_{\ub^\prime}} | \nablaslash\varphi_{k_2}|^2\di\mu_{\Cb_{\ub}},
\end{align*}
where the Gr\"{o}nwall's inequality works. In conclusion, there is
 \begin{equation}\label{energy-esti-EN}
 E_N(u,\ub)+\Eb_N(u,\ub) \lesssim \mathbb{I}^2_{N+1}. 
 \end{equation}

\subsubsection{Energy estimates for ${}^{\tilde S}F_{1+k}(u,\ub),{}^{\tilde S}\Fb_{1+k}(u,\ub), \, k \leq N-1$}\label{sec-tilde-S-energy-estimates-R1}
In this section, we shall make use of Proposition \ref{improved L infinity and L4 estimates in R1} and also the improved $E_k(u, \ub), \, k \leq N-1$ \eqref{energy-esti-Ek} to estimate ${}^{\tilde S}F_{1+k}(u,\ub),{}^{\tilde S}\Fb_{1+k}(u,\ub)$, $k \leq N-1$.
Taking the multiplier $\xi = \delta^{-1} \Lb$ and $\psi =\tilde S \varphi_k$ yields the energy inequality \eqref{energy-ineq-Lb-phi-R1-summary} with $\psi $ being replaced by $\tilde S \varphi_k$. That is,
\begin{equation}\label{energy-ineq-Lb-S-varphi-R1}
{}^{\tilde S}F_{1+k}(u,\ub) + {}^{\tilde S}\Fb_{1+k}(u,\ub)  
\lesssim  I_{N+1}^2 + \doubleint_{\Dp} |\Box_g \tilde S \varphi_k|^2\di \mu_{\D},
\end{equation}
where the double integrated term is split as $\iint_{\Dp} |\Box_g \tilde S \varphi_k|^2\di \mu_{\D} = {}^{\tilde S}\mathcal{G}^k +{}^{\tilde S}\mathcal{H}^k + {}^{\tilde S}W^k$, with $k_1+k_2\leq k \leq N-1, \, k_1\leq k_2$,
\begin{align*}
{}^{\tilde S}\mathcal{G}^k&=  \doubleint_{\Dp} |Q(\p \tilde S \varphi_{k_1},\p \varphi_{k_2})|^2\di\mu_{\D}, \\
{}^{\tilde S}\mathcal{H}^k&= \doubleint_{\Dp} \sum_{i \leq 1}  |Q(\p \varphi_{k_1},\p \tilde S^i \varphi_{k_2})|^2\di{\mu_{\D}}, \\
{}^{\tilde S}W^k&=
\doubleint_{\Dp} |[\Box_g,\tilde S]\varphi_k|^2\di{\mu_{\D}}.
\end{align*}

By analogy with ${}^{L}\mathcal{G}^k$ and ${}^{L}\mathcal{H}^k$, ${}^{\tilde S}\mathcal{G}^k, {}^{\tilde S}\mathcal{H}^k$ can be divided into ${}^{\tilde S}\mathcal{G}^k = {}^{\tilde S}G^k_{1} + \cdots + {}^{\tilde S}G^k_{3}$ and $ {}^{\tilde S}\mathcal{H}^k = {}^{\tilde S}H^k_{1} + \cdots + {}^{\tilde S}H^k_{3}$ respectively. Each of them can be bounded in a similar way as ${}^{L}G^k_{i}$, ${}^{L}H^k_{i}$.

${}^{\tilde S}\mathcal{G}^k = {}^{\tilde S}G^k_{1} + \cdots + {}^{\tilde S}G^k_{3}$, where for $k_1+k_2\leq k \leq N-1, \, k_1\leq k_2$,
\begin{align*}
{}^{\tilde S}G^k_{1}  &= \doubleint_{\Dp} | D\varphi_{k_2}|^2  |\Lb \tilde S \varphi_{k_1}|^2 \di \mu_{\D};\\
{}^{\tilde S}G^k_{2}  &= \doubleint_{\Dp} | D\varphi_{k_2}|^2 |\nablaslash \tilde S \varphi_{k_1}|^2 \di\mu_{\D};\\
{}^{\tilde S}G^k_{3}  &=  \doubleint_{\Dp} | \Db\varphi_{k_2}|^2  |L \tilde S \varphi_{k_1}|^2 \di\mu_{\D}.
\end{align*}
In the same way as ${}^{L}\mathcal{G}^k$, we can apply $L^4$ to all of the four factors in ${}^{\tilde S} G^k_i$ and we should always note that $k_1\leq [\frac{N}{2}] \leq N-3$. Knowing that $ \| D\varphi_{k_2}\|_{L^4(S_{\ub,u})} \lesssim \delta^{- \frac{1}{2}}M  \langle u \rangle^{-\frac{1}{2}}$, $k_2 \leq N-1$,
\begin{align*}
{}^{\tilde S}G^k_{1} 
&\lesssim \int_{0}^{\ub} \di \ub^{\prime} \int_{\Cb_{\ub^\prime}} \delta^{-1} M^2 \langle u^\prime \rangle^{-1} \sum_{k_1 \leq i \leq k_1+1} \langle u^\prime \rangle^{-1} |\Lb \tilde S \varphi_{i}|^2 \di \mu_{\Cb_{\ub}} \\
&\lesssim \delta M^4, \quad k_1 \leq N-3,
\end{align*}
where the Sobolev inequality $S_{\ub,u}$ \eqref{Sobolev-L2-L4-S2} and the bootstrap assumption for ${}^{\tilde S} \Fb_{1+k} (u, \ub)$, $k \leq N-1$ are used.
Similarly, 
\begin{align*}
{}^{\tilde S} G^k_{2} 
&\lesssim \int_{0}^{\ub}  \di \ub^{\prime}  \int_{\Cb_{\ub^\prime}} \delta^{-1} M^2 \langle u^\prime \rangle^{-1}  \sum_{k_1 \leq i \leq k_1+1} \langle u^\prime \rangle^{-1} |\nablaslash \tilde S \varphi_{i} |^2 \di \mu_{\Cb_{\ub}}  \\
&\lesssim \int_{0}^{\ub}\delta^{-1} M^2 \sum_{k_1 \leq j \leq k_1+2} \|\Lb \varphi_{j}\|^2_{L^2(\Cb_{\ub^{\prime}})}\di \ub^{\prime}\\
&\lesssim \delta M^4, \quad k_1 \leq N-3,
\end{align*}
where we have made use of the bootstrap assumption for $\Eb_{l} (u, \ub), \, l \leq N$. For ${}^{\tilde S}G^k_{3}$, noting that, $ \| \Db \varphi_{k_2}\|_{L^4(S_{\ub,u})} \lesssim \delta^{\frac{1}{4}}M  \langle u \rangle^{-1},$
\begin{align*}
{}^{\tilde S} G^k_{3} 
&\lesssim \int_{u_0}^{u}  \delta^{\frac{1}{2}} M^2 \langle u^\prime \rangle^{-2} \sum_{k_1 \leq i \leq k_1+1} \langle u^\prime \rangle^{-1}\|L \tilde S \varphi_{i}\|^2_{L^2(C_{u^{\prime}})}\di u^{\prime}\\
&\lesssim \delta^{\frac{1}{2}} \langle u \rangle^{-2} M^4, \quad k_1 \leq N-3.
\end{align*}
Here we have used the bootstrap assumption for ${}^tF_{1+k}(u, \ub), \, k\leq N-1$.

For ${}^{\tilde S}\mathcal{H}^k$, it can be split into the following terms: for $k_1+k_2\leq k \leq N-1, \, k_1\leq k_2$,
\begin{align*}
{}^{\tilde S}H^k_{1} &= \doubleint_{\Dp} \sum_{i \leq 1} | D \varphi_{k_1}|^2  | \Lb \tilde S^i \varphi_{k_2}|^2  \di \mu_{\D};\\
{}^{\tilde S}H^k_{2} &=  \doubleint_{\Dp}  \sum_{i \leq 1} |D \varphi_{k_1}|^2  |\nablaslash \tilde S^i \varphi_{k_2}|^2  \di \mu_{\D}; \\
{}^{\tilde S}H^k_{3} &= \doubleint_{\Dp}  \sum_{i \leq 1} |\Db\varphi_{k_1}|^2 |L \tilde S^i \varphi_{k_2}|^2  \di \mu_{\D}.
\end{align*}
As in the case of ${}^{L}\mathcal{H}^k$, we can apply $L^\infty, L^\infty, L^2, L^2$ to the four factors in ${}^LH^k_i$ and we should always note that $k_2 \leq N-1$.
For ${}^{\tilde S}H^k_{1}$, in view of $\| D \varphi_{k_1}\|^2_{L^\infty} \lesssim  \delta^{-1} M^2  \langle u \rangle^{-2}$,
\begin{align*}
{}^{\tilde S}H^k_{1} 
&\lesssim \int_{0}^{\ub}  \di \ub^\prime  \int_{\Cb_{\ub^\prime}} \sum_{i \leq 1} \delta^{-1} M^2  \langle u^\prime \rangle^{-2} |\Lb \tilde S^i \varphi_{k_2}|^2 \di \mu_{\Cb_{\ub}} \lesssim \delta M^4, \quad k_2 \leq N-1.
\end{align*}
Here we have used the bootstrap assumptions for $\Eb_{k}(u, \ub), \, k\leq N-1$ and ${}^{\tilde S} \Fb_{1+k}, \, k\leq N-1$.
As for ${}^{\tilde S}H^k_{2}$,  we should use the improved $L^\infty$ estimate (Proposition \ref{improved L infinity and L4 estimates in R1}): $\| D \varphi_{k_1}\|_{L^\infty(S_{\ub,u})} \lesssim \delta^{-\frac{1}{2}} \mathbb{I}_{N}  \langle u \rangle^{-1} + \delta^{\frac{1}{4}} M  \langle u \rangle^{-\frac{3}{2}}$ and then
\begin{align*}
{}^{\tilde S}H^k_{2} 
&\lesssim \int_{0}^{\ub} \di \ub^\prime  \int_{\Cb_{\ub^\prime}} \sum_{i \leq 1} \delta^{-1} \mathbb{I}^2_{N} \langle u \rangle^{-2}  |\nablaslash \tilde S^i \varphi_{k_2}|^2 \di \mu_{\Cb_{\ub}} \\
&\lesssim \int_{0}^{\ub} \delta^{-1} \mathbb{I}^2_{N} \left(\|\nablaslash \varphi_{k_2}\|^2_{L^2(\Cb_{\ub^\prime})} +  \sum_{j \leq k_2 + 1} \|\Lb \varphi_{j}\|^2_{L^2(\Cb_{\ub^\prime})} \right) \di \ub^\prime \\
&\lesssim  \int_{0}^{\ub}\delta^{-1} \mathbb{I}^2_{N} (\mathbb{I}^2_{N} + \delta M^2) \di \ub^\prime \lesssim \mathbb{I}^2_{N} (\mathbb{I}^2_{N}+\delta M^2), \quad k_2 \leq N-1,
\end{align*}
where we have used the improved $\Eb_{k}(u, \ub), \, k\leq N-1$ \eqref{energy-esti-Ek} and the bootstrap assumption for $\Eb_{l}(u, \ub), \, l \leq N$.
At last, noting that $\| \Db \varphi_{k_1}\|_{L^\infty(S_{\ub,u})} \lesssim \delta^{\frac{1}{4}}M  \langle u \rangle^{-\frac{3}{2}}$, ${}^{\tilde S}H^k_{3}$ enjoys the estimate
\begin{equation*}
{}^{\tilde S}H^k_{3} 
\lesssim \int_{u_0}^u \delta^{\frac{1}{2}} M^2  \langle u^\prime \rangle^{-3} \|L \tilde S \varphi_{k_2}\|^2_{L^2(C_{u^\prime})}\di u^\prime \lesssim \delta^{\frac{1}{2}} \langle u \rangle^{-2} M^4, \quad k_2 \leq N-1,
\end{equation*}
where we make use of the bootstrap assumption on ${}^tF_{1+k} (u, \ub), \, k \leq N-1$.

Putting all these estimates together, we conclude
\begin{equation*}
 {}^{\tilde S}\mathcal{H}^k+ {}^{\tilde S}\mathcal{G}^k \lesssim \mathbb{I}^2_{N}+ \delta^{\frac{1}{2}} M^4. 
\end{equation*}

For $ {}^{\tilde S}W^k$, we recall \eqref{commutator-S-Box-R1} and utilize \eqref{energy-esti-Ek}, \eqref{energy-esti-EN} and \eqref{energy-esti-Qk} to derive 
\begin{align*}
 {}^{\tilde S}W^k &\lesssim\doubleint_{\Dp} \langle u \rangle^{-2}\left(|L \varphi_{k}|^2+|\Lb \varphi_{k}|^2+|\nablaslash \varphi_{k+1}|^2\right) + |\Box_g \varphi_k|^2 \di \mu_{\D}\\
&\lesssim \mathbb{I}^2_{N+1}+\delta^{\frac{1}{2}} M^4, \quad k \leq N-1.
\end{align*}

We now summarize the above estimates as
\begin{equation}\label{energy-esti-F-2-R1}
{}^{\tilde S}F_{1+k}(u,\ub)+{}^{\tilde S}\Fb_{1+k}(u,\ub) \lesssim I^2_{N+1}+\delta^{\frac{1}{2}} M^4, \quad k \leq N-1.
\end{equation}

\subsubsection{Energy estimates for ${}^tF_{1+k}(u,\ub), \, k\leq N-1$}\label{sec-T-E}
In order to retrieve the estimate for ${}^tF_{1+k}(u,\ub)$, $k\leq N-1$, we will use the improved $L^\infty$ estimate (Proposition \ref{improved L infinity and L4 estimates in R1}) and additionally the upgraded $E_l (u, \ub)$, $l \leq N$ \eqref{energy-esti-Ek}, \eqref{energy-esti-EN}. Besides, since ${}^tF_{1+k}(u,\ub)$ concerns the transversal derivative $\Lb$ on $C_u$, our estimate will be done through integrating along $L$ and making use of the wave equation. 

Letting
\begin{equation}\label{def-chib}
\chib^2[\psi](u,\ub)=\int_{S_{\ub, u}}|\Lb \psi|^2 r^2 \di \sigma_{S^2},
\end{equation}
we derive with the aid of the expression of wave operator,
\begin{align*} 
\p_{\ub}\chib^2[\psi]&=\int_{S_{\ub, u}}2\Lb \psi(L\Lb \psi+\frac{\eta}{r}\Lb \psi)r^2 \di \sigma_{S^2} \\
                                  &=\int_{S_{\ub, u}}2\Lb \psi(\frac{\eta}{r}L\psi+\eta\laplacianslash \psi-\eta\Box_g\psi) r^2 \di \sigma_{S^2} \\
                                  &\lesssim \delta^{-1}\chib^2[\psi]+ \int_{S_{\ub, u}} \delta \left( \langle u\rangle^{-2}|L\psi|^2+|\laplacianslash \psi|^2+|\Box_g\psi|^2\right) r^2 \di \sigma_{S^2}.
\end{align*}
Suppose $\psi \equiv 0$ on $\Cb_{0}$. We then integrate along $\p_{\ub}$ to obtain,
\begin{equation}\label{esti-ineq-Lpsi-Cu}
\chib^2[\psi] \lesssim \int_0^{\ub} \delta^{-1}\chib^2[\psi] \di \ub^\prime + \int_{C_u}  \delta \left( \langle u\rangle^{-2}|L\psi|^2+|\laplacianslash \psi|^2+|\Box_g\psi|^2 \right) \di \mu_{C_{u}}.
\end{equation}

Taking $\psi= \varphi_k, \, k \leq N-1$ in \eqref{esti-ineq-Lpsi-Cu}, knowing that $\varphi \equiv 0$ on $\Cb_{0}$, we have
\begin{equation}\label{esti-Lb-varphik-ineq}
\chib^2[\varphi_k] \lesssim  \int_0^{\ub} \delta^{-1}\chib^2[\varphi_k] \di \ub^\prime+ \int_{C_u}  \delta \left( \langle u\rangle^{-2}|L\varphi_k|^2+|\laplacianslash \varphi_k|^2+|\Box_g\varphi_k|^2 \right)  \di \mu_{C_{u}}.
\end{equation}
Using the improved results for $E_l(u, \ub), \, l \leq N$ \eqref{energy-esti-Ek}, \eqref{energy-esti-EN}, we obtain
\begin{equation}\label{esti-Lb-varphik-leading-Cu}
 \int_{C_u}  \delta \left( \langle u\rangle^{-2}|L\varphi_k|^2+|\laplacianslash \varphi_k|^2 \right) \di \mu_{C_{u}} \lesssim\delta\langle u\rangle^{-2}\mathbb{I}^2_{N+1}, \quad k \leq N-1.
\end{equation}
For the term $\delta \int_{C_u} |\Box_g\varphi_k|^2\di \mu_{C_u}$,  we make the splitting: for $k_1+k_2\leq k \leq N-1, \, k_1 \leq k_2$,
 \begin{align*}
 S^k_1 &= \delta \int_{C_u}| D \varphi_{k_1}|^2  |\Lb\varphi_{k_2}|^2\di \mu_{C_u}, \quad & S^k_2 = \delta \int_{C_u}  | \Db\varphi_{k_1}|^2 |L\varphi_{k_2} |^2\di \mu_{C_u},\\
 S^k_3 &= \delta \int_{C_u} | \Db\varphi_{k_1}|^2  |\nablaslash\varphi_{k_2} | ^2\di\mu_{C_u}, \quad & S^k_4 = \delta \int_{C_u}|L\varphi_{k_1}|^2 |\nablaslash\varphi_{k_2}|^2\di \mu_{C_u}.
 \end{align*}
 We now treat these error terms one by one.
In view of the improved $\|L \varphi_{k_1} \|_{L^\infty(\R_1)}$ (Proposition \ref{improved L infinity and L4 estimates in R1}), $\| \Db \varphi_{k_1} \|_{L^\infty(\R_1)} \lesssim \delta^{\frac{1}{4}}M \langle u \rangle^{-\frac{3}{2}}$,  $k_1 \leq [\frac{N}{2}] \leq N-3$, and the enhanced $E_k(u, \ub), \, k \leq N-1$ \eqref{energy-esti-Ek},
\begin{align*}
|S^1_k|  &  \lesssim  \mathbb{I}^2_{N} \langle u \rangle^{ -2} \int_0^{\ub} \chib^2 [\varphi_k] (u,\ub^\prime) \di \ub^\prime, \\
 |S^2_k| + |S^3_k| & \lesssim \delta^{ \frac{3}{2}} M^2 \langle u \rangle^{ -3} \left( \| L \varphi_{k_2}\|^2_{L^2(C_{u})}+ \| \nablaslash \varphi_{k_2}\|^2_{L^2(C_{u})} \right) \lesssim \delta^{ \frac{3}{2}} M^4 \langle u \rangle^{ -3},\\
|S^4_k| & \lesssim  \mathbb{I}^2_{N} \langle u \rangle^{-2} \|\nablaslash \varphi_{k_2}\|^2_{L^2(C_{u})} \lesssim \delta \mathbb{I}^4_{N} \langle u\rangle^{ -2}.
\end{align*}
Therefore, for $k \leq N-1$,
\begin{equation}\label{esti-Qk-Cu}
\delta \int_{C_u}|\Box_g\varphi_k|^2\di \mu_{C_u} 
\lesssim   \mathbb{I}^2_{N} \langle u \rangle^{ -2} \int_0^{\ub} \chib^2 [\varphi_k] (u,\ub^\prime) \di \ub^\prime+ \delta \mathbb{I}^4_{N} \langle u\rangle^{ -2}.
\end{equation}
The Gr\"{o}nwall's inequality together with \eqref{esti-Lb-varphik-ineq}-\eqref{esti-Qk-Cu} leads to
 \begin{equation}\label{esti-Lbvarphik-ineq}
\|\Lb  \varphi_k\|^2_{L^2(S_{\ub,u})} \lesssim \delta \langle u\rangle^{-2}\mathbb{I}^2_{N+1}, \quad k \leq N-1.
\end{equation}
Integrating \eqref{esti-Lbvarphik-ineq} over the interval $\ub \in [0, \delta]$ yields
\begin{equation}\label{energy-esti-chib-Cu}
 \| \Lb \varphi_k \|^2_{L^2(C_u)} \lesssim  \delta^{2}  \langle u\rangle^{-2}\mathbb{I}^2_{N+1}, \quad k \leq N-1.
\end{equation}

As for $\|L\Lb \varphi_k\|_{L^2({C_u})}$, we shall make use of the wave equation, which reads
$\eta^{-1} L \Lb\varphi_k = \laplacianslash \varphi_k + \frac{L \varphi_k}{r}- \frac{\Lb \varphi_k}{r} - \Box_g \varphi_k$. Taking \eqref{esti-Lb-varphik-leading-Cu}-\eqref{energy-esti-chib-Cu} into account, we deduce that $$\|L\Lb \varphi_k\|^2_{L^2(C_u)} \lesssim  \langle u \rangle^{-2} \mathbb{I}^2_{N+1}, \quad k\leq N-1.$$
And hence (noting that, $|L\Lb \varphi_k|= \langle u \rangle^{-1}|L\tilde S \varphi_k|$)
\begin{equation}\label{eq-energy-L-S-varphik}
 \|L\tilde S \varphi_k\|_{L^2({C_u})}\lesssim \mathbb{I}_{N+1}, \quad k \leq N-1.
 \end{equation}

For the sake of clarity, we assemble these results with regard to the transversal derivative $\Lb$ on $C_u$ in the following proposition.
\begin{proposition}\label{prop-Lb-S-R1}
For $k \leq N-1$, we have in $\R_1$,
\begin{equation}\label{energy-esti-chib}
 \delta^{-\frac{1}{2}}  \langle u\rangle  \|\Lb \varphi_k\|_{L^2(S_{\ub,u})} + \langle u\rangle  \|L \Lb \varphi_k\|_{L^2(C_u)} +  \|L \tilde S \varphi_k\|_{L^2(C_u)}   \lesssim  \mathbb{I}_{N+1}.
\end{equation}
\end{proposition}

\subsubsection{End of the bootstrap argument in Region $\mathcal{R}_1$.}\label{sec-close-bt-R1}
Putting the estimates \eqref{energy-esti-Ek}, \eqref{energy-esti-F-L-R1}, \eqref{energy-esti-EN}, \eqref{energy-esti-F-2-R1} and Proposition \ref{prop-Lb-S-R1} together, we have arrived at, for $ l \leq N,\, k\leq N-1$,
\begin{equation}\label{Main-derived-estimates-R1}
 E_l(u,\ub) + \Eb_l(u,\ub)+ F_{1+k}(u,\ub) + \Fb_{1+k}(u,\ub)+{}^{t}F_{1+k}(u,\ub) \leq C \mathbb{I}^2_{N+1}, \quad \text{in} \,\, \R_1.
\end{equation}
By choosing $M$ (which depends on the initial data) large enough such that $C I^2_{N+1} \leq \frac{M^2}{4}$, and $\delta$ small enough such that $ C \delta^{\frac{1}{2}}M^{4} \leq \frac{M^2}{4}$, we can replace the $C \mathbb{I}^2_{N+1}$ in \eqref{Main-derived-estimates-R1}  by $\frac{M^2}{2}$, and hence the $M^2$ in \eqref{bootstrap assumption in R1} is replaced by $\frac{M^2}{2}$. The bootstrap argument is closed, which gives rise to the estimate \eqref{main estimates in R1} as well.

\subsubsection{Energy estimates for general derivatives in $\mathcal{R}_1$.}
To continue with general derivatives,  we define for $i=0,1$ and $i+l+k\leq N$,
\begin{align*}
E_{i+l+k}(u,\ub)&= \sum_{p+q=l} \delta^{2p}  \|L\tilde S^i W^l_{p,q}\varphi_k\|^{2}_{L^2(C_u)}+\delta^{2p-1}   \|\nablaslash \tilde S^iW^l_{p,q}\varphi_k\|^{2}_{L^2(C_u)},\\
\Eb_{i+l+k}(u,\ub)&= \sum_{p+q=l} \delta^{2p} \|  \nablaslash \tilde S^iW^l_{p,q}\varphi_k\|^{2}_{L^2(\Cb_{\ub)}}+\delta^{2p-1} \|  \Lb \tilde S^iW^l_{p,q}\varphi_k\|^{2}_{L^2(\Cb_{\ub})},
\end{align*}
and for $l+k\leq N-1$,
\begin{equation*}
{}^{t}F_{1+l+k}(u,\ub)=\sum_{p+q=l} \delta^{2p} \left(\delta^{-2}  \| \tilde S W^l_{p,q}\varphi_k\|^{2}_{L^2(C_u)}  + \|L \tilde S W^l_{p,q}\varphi_k\|^{2}_{L^2(C_u)} \right).
\end{equation*}
The energy estimate \eqref{main estimates in R1} can be extended to general energy norms.
\begin{theorem}\label{Thm-highorder-enengy-esti-R1}
Letting $N\geq6$, we have in  $\mathcal{R}_1$: $u_0 \leq u\leq 1, $ $0\leq \ub\leq \delta$,
\begin{align*}
E_{i+l+k}(u,\ub)+\Eb_{i+l+k}(u,\ub) & \lesssim I^2_{N+1},~~i=0,1,~i+l+k\leq N, \\
{}^{t}F_{1+l+k}(u,\ub) & \lesssim I^2_{N+1},~~~~~~~~i+l+k\leq N-1,
\end{align*}
provided that the initial energy is bounded by $I^2_{N+1}$.
\end{theorem}
This theorem can be easily proved by an inductive argument on $l$, i.e., the numbers of $W$ derivative and thus we will omit the details here. 
In the proof, the following $L^\infty$ and $L^4$ estimates (in $\R_1$) can be inferred as well:
\begin{align*}
\delta^{p+\frac{1}{2}}\langle u \rangle \|L W^l_{p,q}\varphi_k\|_{L^\infty(\R_1)}+\delta^{p-\frac{1}{4}}\langle u \rangle^{ \frac{3}{2}}   \|\Db W^l_{p,q}\varphi_k\|_{L^{\infty}(\R_1)} & \lesssim I_{N+1},\quad l+k\leq N-2, \\
\delta^{p+\frac{1}{2}}\langle u \rangle^{ \frac{1}{2}}\|LW^l_{p,q}\varphi_k\|_{L^4(S_{\ub,u})}+\delta^{p-\frac{1}{4}}\langle u \rangle  \|\Db W^l_{p,q}\varphi_k\|_{L^4(S_{\ub,u})}  & \lesssim I_{N+1},\quad l+k\leq N-1.
\end{align*}
Moreover, an analogous version of Proposition \ref{prop-Lb-S-R1} can be derived as follows.
\begin{proposition}\label{prop-Lb-S-high-R1}
For $l+ k \leq N-1$, we have in $\R_1$,
\begin{equation*} 
 \delta^p ( \delta^{-\frac{1}{2}}  \langle u\rangle \|\Lb W^l_{p,q}\varphi_k \|_{L^2(S_{\ub,u})} + \langle u\rangle \|L \Lb W^l_{p,q}\varphi_k\|_{L^2(C_u)} +  \|L \tilde S W^l_{p,q}\varphi_k\|_{L^2(C_u)})  \lesssim I_{N+1}.
\end{equation*}
\end{proposition}

\subsection{Smallness on the last cone in $\R_1$}\label{sec-small-Cb-1}

\begin{theorem}\label{small-W-R1}
Given any fixed $N\geq 6$, we have, on the last cone $\Cb_\delta \cap \R_1$,
\begin{align*}
\|r^{i-1} \Db^i W^{l}_{p,q}\varphi_k\|_{L^2(\Cb_{\delta}\cap\R_1)} & \lesssim \delta^{\frac{1}{2}}, \quad & i \leq 1, \, 2l  +k+i \leq N; \\
\|r^{i-1} \Db^i W^{l}_{p,q}\varphi_k\|_{L^\infty(S_{\delta,u} \cap \R_1)} & \lesssim  \delta ^{\frac{1}{2}}\langle u \rangle^{ -2}, \quad & i \leq 1, \, 2l +k+i\leq N-2.
\end{align*}
And
\begin{align*}
 \|L W^{l}_{pq}\varphi_k\|_{L^2(\Cb_{\delta}\cap\R_1)} & \lesssim \delta ^{\frac{1}{2}}, &\quad 2l +k \leq N-2; \\
\|L W^{l}_{pq}\varphi_k\|_{L^\infty(S_{\delta,u} \cap \R_1)} & \lesssim  \delta ^{\frac{1}{2}}\langle u \rangle^{-\frac{3}{2}}, &\quad 2l +k\leq N-4.
\end{align*}
\end{theorem}

For the proof, we begin with the cases involving merely good derivatives.
\begin{proposition}\label{prop-small-nabla-R1}
We have in Region $\mathcal{R}_1$, for any fixed $N\geq 6$ and $\ub \in [0, \delta]$,
\begin{align*}
\|r^{i-1} \Db^i \Lb^{l}\varphi_k\|_{L^2(\Cb_{\ub} \cap \R_1)} & \lesssim \delta ^{\frac{1}{2}},  &\quad  l+k+i \leq N, \, i \leq 1; \\
\|r^{i-1} \Db^i \Lb^l\varphi_k\|_{L^\infty(S_{\ub,u} \cap \R_1)} & \lesssim  \delta ^{\frac{1}{2}}\langle u \rangle^{ -2}, &\quad l+k+i\leq N-2, \, i \leq 1.
\end{align*}
\end{proposition}

\begin{proof}
Firstly,  considering $\Db$ to be $\nablaslash$, we define 
\begin{equation}\label{def-omega-R1}
\omega^2[\psi](u,\ub)=\int_{S_{\ub, u}} |\nablaslash \psi|^2 r^2 \di \sigma_{S^2}.
\end{equation}
We take $\psi= \Lb^l\varphi_k,\, (l+k\leq N-1)$ and derive the transport equation
\begin{align*}
\p_{\ub} \omega^2 [ \Lb^l\varphi_k] (u, \ub)
&=\int_{S_{\ub,u}}2 \nablaslash \Lb^l\varphi_k\cdot \nablaslash L   \Lb^l\varphi_kr^2\di \sigma_{S^2}\\
&\lesssim \delta^{-1} \omega^2 [   \Lb^l\varphi_k]  (u, \ub)+ \int_{S_{\ub,u}}   \delta \langle u \rangle^{-2} |L   \Lb^l\varphi_{k+1}|^2 r^2\di \sigma_{S^2},
\end{align*}
where in the last inequality, $| \nablaslash L  \Lb^l\varphi_k|^2 \sim \langle u \rangle^{-2} |L  \Lb^l\varphi_{k+1}|^2$ in $\R_1$ is used.
Now that $\nablaslash \Lb^l\varphi_k \equiv 0$ on the incoming cone $\Cb_{0}$, by the Gr\"{o}nwall's inequality, there is
\begin{equation*}
\omega^2[  \Lb^l\varphi_k] (u,\ub)   \lesssim \delta \| \langle u \rangle^{-1} L \Lb^{l}\varphi_{k+1}\|^2_{L^2(C_{u})}\lesssim  \delta \langle u \rangle^{ -2}.
\end{equation*}
Integrating over the interval $[u_0, u]$ leads to
\begin{equation*}
 \| \nablaslash \Lb^l\varphi_{k}\|^2_{L^2(\Cb_{\ub})} \lesssim \delta.
\end{equation*}
Define $h^2[\psi](u,\ub)=\int_{S_{\ub, u}} |\psi|^2 \di \sigma_{S^2}$. Then, it follows in the same manner that
\begin{align*}
h^2[  \Lb^l\varphi_k] (u,\ub) & \lesssim \delta \|r^{ -1}  L \Lb^{l}\varphi_{k}\|^2_{L^2(C_{u})} \lesssim  \delta \langle u \rangle^{ -2}, \\
 \|\Lb^l\varphi_{k}\|^2_{L^2(\Cb_{\ub})} & \lesssim \delta.
\end{align*}

When $\Db$ is taken as $\Lb$, the smallness follows straightforwardly as a consequence of Theorem \ref{Thm-highorder-enengy-esti-R1} and Proposition \ref{prop-Lb-S-high-R1}.

Eventually, the $L^\infty$ estimates follow from the Sobolev inequality on $S_{\ub,u}$. 
\end{proof}

For  $L \Lb^l \varphi_k$,  the smallness will take place on the last incoming cone $\Cb_{\delta} \cap \R_1$.
\begin{proposition}\label{prop-small-L-R1}
On $\Cb_\delta \cap \R_1$, we have,  for any fixed $N\geq 6$,  
\begin{align*}
\| L \Lb^l \varphi_k\|_{L^2(\Cb_{\delta} \cap \R_1)} & \lesssim \delta ^{\frac{1}{2}}, &\quad l+k \leq N-2; \\
\|L  \Lb^l \varphi_k\|_{L^\infty(S_{\delta,u} \cap \R_1)} & \lesssim  \delta ^{\frac{1}{2}}\langle u\rangle^{ -\frac{3}{2}}, &\quad l+k\leq N-4.
\end{align*}
\end{proposition}
\begin{proof}
To illustrate the idea, we will carry out the estimates for $l=0$ in detail.
Define
\begin{equation}\label{def-chi-R1}
\chi^2[\psi](u,\ub)=\int_{S_{\ub, u}} |L \psi|^2 r^3  \di \sigma_{S^2},
\end{equation}
and take $\psi = \varphi_k$, with $k \leq N-2$. There is the transport equation
\begin{align*}
\p_{u}\chi^2[ \varphi_k] (u, \ub) =&\int_{S_{\ub,u}} 2 r L  \varphi_k\left(\Lb L   \varphi_k- \frac{3\eta}{2r}L   \varphi_k \right)r^2 \di \sigma_{S^2}\\
=& \int_{S_{\ub,u}}  2 r L  \varphi_k \eta \left(\laplacianslash  \varphi_k - \Box_g \varphi_k - \frac{1}{r} \Lb \varphi_k \right)r^2 \di \sigma_{S^2} - \int_{S_{\ub,u}} \eta |L \varphi_k|^2 r^2 \di \sigma_{S^2}.
\end{align*}
That is,
\begin{align*}
&\,\, \p_{u}\chi^2[ \varphi_k](u, \ub) + \int_{S_{\ub,u}}   \eta |L \varphi_k|^2 r^2\di \sigma_{S^2} \\  
=& \int_{S_{\ub,u}} 2 \eta L  \varphi_k \cdot r \left(\laplacianslash  \varphi_k - \Box_g  \varphi_k - \frac{1}{r} \Lb   \varphi_k \right)r^2 \di \sigma_{S^2}.
\end{align*}
Integrating over $[u_0, u]$,  using the Cauchy-Schwarz inequality and absorbing terms which can be bounded by the positive term $ \int_{S_{\ub,u}}   \eta |L \varphi_k|^2 r^2\di \sigma_{S^2}$ on the left hand side after a small change in constant, we derive (note that $\langle u \rangle \sim r$, $\eta \sim 1$ in $\R_1$) 
\begin{equation}\label{eq-small-chi-1-R1}
\begin{split}
&\chi^2 [  \varphi_k](u,\ub)+\int_{\Cb_{\ub}} |L \varphi_k|^2\di \mu_{\Cb_{\ub}}\\
 \lesssim {}& \chi^2 [ \varphi_k](u_0,\ub)+\int_{\Cb_{\ub}} \left(|\Lb   \varphi_k|^2+|\nablaslash \varphi_{k+1}|^2\right)\di \mu_{\Cb_{\ub}} + \int_{\Cb_{\ub}}\langle u\rangle^{2} |\Box_g \varphi_k|^2  \di \mu_{\Cb_{\ub}}.
\end{split}
\end{equation}
Indicated by Proposition \ref{prop-small-nabla-R1},
\begin{equation*}
\int_{\Cb_{\ub}} \left(|\Lb \varphi_k|^2+|\nablaslash  \varphi_{k+1}|^2\right) \di \mu_{\Cb_{\ub}}\lesssim\delta, \quad k \leq N-2.
\end{equation*}
Therefore, we are left with
\begin{equation}\label{inequa-L-small-R1}
\begin{split}
&\chi^2 [ \varphi_k](u,\ub)+\int_{\Cb_{\ub}} |L  \varphi_k|^2\di \mu_{\Cb_{\ub}}\\
\lesssim {}&\chi^2 [ \varphi_k](u_0,\ub) + \delta + \int_{\Cb_{\ub}}\langle u\rangle^{2} |\Box_g \varphi_k|^2  \di \mu_{\Cb_{\ub}}, \quad k \leq N-2.
\end{split}
\end{equation}
By the null condition, the remaining error term can be decomposed as $\int_{\Cb_{\ub}}\langle u\rangle^{ 2} |\Box_g \varphi_k|^2  \di \mu_{\Cb_{\ub}} = Er_1+ Er_2 + Er_3$,  where for $k_1+k_2\leq k \leq N-2$, $k_1 \leq k_2$,
\begin{align*}
Er_1&= \int_{\Cb_{\ub}} \langle u\rangle^{2} |\bar D \varphi_{k_1} \bar D\varphi_{k_2}|^2 \di \mu_{\Cb_{\ub}},\\
Er_2&= \int_{\Cb_{\ub}} \langle u\rangle^{ 2} | \Db\varphi_{k_1} L\varphi_{k_2} |^2 \di \mu_{\Cb_{\ub}},\\
 Er_3&= \int_{\Cb_{\ub}} \langle u\rangle^{ 2}|L\varphi_{k_1} \Db\varphi_{k_2}|^2 \di \mu_{\Cb_{\ub}}. 
\end{align*}
In view of Proposition \ref{prop-small-nabla-R1}, it is easy to check that $Er_1 \lesssim \delta ^2$ and 
\begin{align*}
Er_2 & \lesssim \int_{\Cb_{\ub}}\langle u\rangle^{ 2}\|\Db\varphi_{k_1}\|^2_{L^\infty}|L\varphi_{k_2} |^2 \lesssim\int_{\Cb_{\ub}} \delta  \langle u\rangle^{-2}|L\varphi_{k_2} |^2,
\end{align*}
which can be absorbed by the left hand side of \eqref{inequa-L-small-R1}. Finally, for $Er_3$, if $k=0$, it is the same as $Er_2$; if $k \geq 1$,  then $ k_1 +1 \leq [\frac{k}{2}] +1 \leq k$ (recalling that
$k_1 + k_2 \leq k, \, k_1 \leq k_2$) and we can perform $L^4$ on all the four factors in $Er_3$. Now, analogous to ${}^{L}\mathcal{G}^k$,
\begin{align*}
 Er_3 &\lesssim \int_{u_0}^{u}\langle u^\prime \rangle^{ 2} \|L\varphi_{k_1}\|^2_{L^4(S_{\ub,u^{\prime}})} \|\Db \varphi_{k_2}\|^2_{L^4(S_{\ub,u^{\prime}})}\di u^{\prime}\\
  &\lesssim\int_{u_0}^{u}\delta \langle u^\prime \rangle^{-2}\sum_{j \leq k_1+1} \|L\varphi_{j}\|^2_{L^2(S_{\ub,u^{\prime}})}\di u^{\prime}, \quad k_1 +1 \leq k,
\end{align*}
which can be absorbed by the left hand side of \eqref{inequa-L-small-R1} as well.
In a word, we deduce that
\begin{equation*}
\chi^2[\varphi_k](u,\ub) + \| L\varphi_k\|_{L^2(\Cb_{\ub})}^2 \lesssim  \chi^2[\varphi_k](u_0,\ub)+\delta, \quad k \leq N-2.
\end{equation*}
Letting $\ub = \delta$, and knowing that $\chi^2[\varphi_k](u_0,\delta)=0$ for the data are compactly supported in $C_{u_0}^{[0, \delta]}$, we have
 \begin{equation}\label{eq-small-L-0}
 \|\langle u\rangle^{ \frac{1}{2}}L \varphi_k\|^2_{L^2(S_{\delta,u})}  +   \| L\varphi_k\|_{L^2(\Cb_{\delta})}^2  \lesssim  \delta, \quad k \leq N-2.
\end{equation}
At last, the $L^\infty$ estimate is implied by the Sobolev inequality on $S_{\ub, u}$.

Thus, we accomplish the proof for the $l=0$ case, while the argument for $l >0$ is similar and hence omitted. 
\end{proof}

\begin{proof}[\bf{Proof of Theorem \ref{small-W-R1}:}] 
We will prove this theorem by an inductive argument on $l$.  Theorem \ref{small-W-R1} with $l=0$ has been verified by the propositions \ref{prop-small-nabla-R1} and \ref{prop-small-L-R1}. Suppose Theorem \ref{small-W-R1} holds true for $l\leq n$, we wish to prove that it holds true as well when $l=n+1$. That is, we shall prove the smallness for $L W^{n+1}_{p,q} \varphi_k$, $2(n+1) +k \leq N-2$ and $\Db W^{n+1}_{p,q} \varphi_k$, $2(n+1) +k \leq N-1$, $ W^{n+1}_{p,q} \varphi_k$, $2(n+1)+k\leq N$.

In the case of $p=0$, the smallness holds true by virtue of the propositions \ref{prop-small-nabla-R1} and \ref{prop-small-L-R1}.

For the case of $p\geq1$, it can be argued  in two steps as below.

{\bf Step I}: $\Db W^{n+1}_{p,q} \varphi_k$, with  $1\leq p$, $2(n+1)+k \leq N-1$ and $W^{n+1}_{p,q} \varphi_k$, with  $1\leq p$, $2(n+1)+k \leq N$. We note that
\begin{align*}
\nablaslash W^{n+1}_{p,q} \varphi_k &\sim r^{-1} L  W^{n}_{p-1,q} \varphi_{k+1}, \quad 2n+k+1 \leq N-2,\\ 
W^{n+1}_{p,q} \varphi_k &\sim  L  W^{n}_{p-1,q} \varphi_{k}, \quad 2n+k \leq N-2,
\end{align*}
both of which reduce to the $l=n$ case, and hence the smallness holds true by the inductive assumption.

When it comes to $\Lb W^{n+1}_{p,q} \varphi_k$ with  $1\leq p$, and $2(n+1)  +k \leq N-1$, we proceed by an analogous idea: $\Lb W^{n+1}_{p,q} \varphi_k = \Lb L W^{n}_{p-1,q} \varphi_k,$ and use the wave equation,
\begin{equation}\label{eq-Lb-L-Wn-phi}
 \eta^{-1} \Lb L W^{n}_{p-1,q} \varphi_k  =   - \Box_g W^{n}_{p-1,q}  \varphi_k +  \laplacianslash W^{n}_{p-1,q} \varphi_k + \frac{1}{r} L W^{n}_{p-1,q} \varphi_k -   \frac{1}{r} \Lb W^{n}_{p-1,q} \varphi_k,
\end{equation}
where $|\Box_g W^{n}_{p-1,q}  \varphi_k| \lesssim |\bar W^{n}_{p-1,q} \Box_g \varphi_k| + |W_{\leq n-1} (\varphi_k)|$.
All the terms on the right hand side of \eqref{eq-Lb-L-Wn-phi} are of lower order in $W$ derivative, and can be reduced to the $l \leq n$ case, noting that $2n+k\leq N-3$. Hence, the smallness for $\Lb W^{n+1}_{p,q} \varphi_k$ with  $1\leq p$, and $ 2( n+1)+k \leq N-1$ follows by induction.

{\bf Step II}: $L W^{n+1}_{p,q}\varphi_k$, with  $1\leq p$, and $2( n+1) +k \leq N-2$. 
Recall the definition \eqref{def-chi-R1}: $\chi^2[ W^{n}_{p,q}\varphi_k](u,\ub)=\int_{S_{\ub,u}} |L  W^{n}_{p,q}\varphi_k|^2 r^3\di \sigma_{S^2}.$ Following the proof of Proposition \ref{prop-small-L-R1}, we deduce a general version of \eqref{eq-small-chi-1-R1}:
\begin{align*}
&\chi^2[ W^{n+1}_{p,q}\varphi_k](u,\ub)+\int_{\Cb_{\delta}} |L W^{n+1}_{p,q}\varphi_k|^2\di \mu_{\Cb_{\ub}}\\
\lesssim  {}&\chi^2[ W^{n+1}_{p,q}\varphi_k](u_0,\ub)+ \int_{\Cb_{\delta}}\langle u\rangle^{ 2}|\Box_g   W^{n+1}_{p,q}\varphi_k|^2\di \mu_{\Cb_{\ub}}\\
&+\int_{\Cb_{\delta}} \left(|\Lb W^{n+1}_{p,q} \varphi_k|^2+|\nablaslash W^{n+1}_{p,q}\varphi_{k+1}|^2\right)\di \mu_{\Cb_{\ub}},
\end{align*}
where by the results obtained in Step I, the last line admits the estimate
\begin{equation*}
\int_{\Cb_{\delta}} \left(|\Lb W^{n+1}_{p,q} \varphi_k|^2+|\nablaslash W^{n+1}_{p,q}\varphi_{k+1}|^2\right)\di \mu_{\Cb_{\ub}}\lesssim \delta.
\end{equation*}
Here we notice that, for the second term above, $2(n+1)+k +1\leq N-1$ does not exceed the regularity. 

Now we turn to the error term $\int_{\Cb_{\delta}}\langle u\rangle^{ 2}|\Box_g W^{n+1}_{p,q}\varphi_k|^2\di \mu_{\Cb_{\ub}}$.
Recall that $|\Box_g W^{n+1}_{p,q} \varphi_k|$ $\lesssim |\bar W^{n+1}_{p,q} \Box_g    \varphi_k| + |W_{\leq n} (\varphi_k)|$. Then $ \int_{\Cb_{\delta}}\langle u\rangle^{ 2}|\Box_g W^{n+1}_{p,q}\varphi_k|^2\di \mu_{\Cb_{\ub}}$ can be bounded by $\F^{n+1,k}_{L1} + \F^{n+1,k}_{L2}$, with
\begin{align*}
\F^{n+1,k}_{L1} &=  \int_{\Cb_{\delta}}\langle u\rangle^{ 2} |W_{\leq n} (\varphi_k) |^2, \quad \F^{n+1,k}_{L2} =   \int_{\Cb_{\delta}}\langle u\rangle^{ 2}  |\bar W^{n+1}_{p,q} \Box_g \varphi_k|^2.
\end{align*}
By induction, $\F^{n+1,k}_{L1} \lesssim \delta$. 
In addition, combining the inductive assumption and the results in Step I, $\F^{n+1,k}_{L2}$ which has null structure, can be bounded in an analogous way as $\int_{\Cb_{\delta}}\langle u\rangle^{ 2} | \Box_g    \varphi_k|^2  \di \mu_{\Cb_{\ub}}$, see Proposition \ref{prop-small-L-R1}.
Thus we conclude that in $\R_1$,
$$ \|\langle u\rangle^{ \frac{1}{2}}LW^{n+1}_{p,q}\varphi_k\|^2_{L^2(S_{\delta,u})} + \|LW^{n+1}_{p,q}\varphi_k\|^2_{L^2(\Cb_\delta)} \lesssim \delta,\quad 2(n+1) +k\leq N-2.$$

We complete the inctive argument.
In the end, the $L^\infty$ estimate follows by the Sobolev inequality.
\end{proof}

\subsection{Small data problem in Region III}\label{sec-global-region-III-1}
Due to Theorem \ref{small-W-R1}, the global existence of a solution to \eqref{Main Equation} in Region III, is reduced to a small data problem with characteristic data prescribed on $C_{u_0}^{[\delta, +\infty]}$ and $\Cb_{\delta}$, for which the local existence is ensured by Rendall's theorem \cite{Rendall-90}. One can also refer to \cite[Section 5.1]{Wang-Yu-16} or \cite[Chapter $16$]{Christodoulou-09} for detailed argument.
In regard of the global existence, it is remarkable that in this region (where $t$ might be negative) the conformal multiplier $K=u^2 \Lb + \ub^2 L$ does not offer a favourable sign even near the spatial infinity, referring to \eqref{current-K}. In the same way, $K$ is not allowed if we are considering the scattering problem. It means that we need to prove a small data global existence theorem without using $K$ in Region III.  As a remark, Luk's theorem \cite{Luk-15-nonlinear} where $K$ is crucial in the proof does not apply here.

We first consider the subregion: $\{ t \leq 0\} \cap \text{Region III}$, which is always away from the event horizon and the photon sphere $r=3m$, because $t \leq 0$ and $\ub \geq \delta$ imply that $r^\ast \geq 2 \delta >0$, and hence there is $r\geq R_0>3m$, for some $R_0 >3m$. By virtue of these crucial facts, we will make use of the approach in \cite{MMDT-strich-sch-10, L-T-quasilinear-sch-18} (without using $K$) to show the global existence in this subregion. 

Let $t_0 \leq t_1 \leq 0$ and $\M_{[t_0,t_1]}^{ext}$ denote the domain bounded by $\{t=t_0\}$, $\{t=t_1\}$, $C_{u_0}$ and $\Cb_\delta$. We know that $r\geq R_0 >3m$ in $\M_{[t_0,t_1]}^{ext}$. We shall use some notations in \cite{MMDT-strich-sch-10, L-T-quasilinear-sch-18}. Consider a partition of $\mathbb{R}^3$ into dyadic sets $A_R = \{\langle r \rangle  \approx R\}$ for some $R \geq 1$, with the obvious change for $R =1$: $A_1 = \{0 \leq r\leq 1\}$. The local energy norm $LE$ in $\M_{[t_0,t_1]}^{ext}$ is defined as
\begin{equation}\label{de-LE}
\|\psi\|_{LE} = \sup_{R \geq R_0} \|\langle r \rangle^{-\frac{1}{2}} \psi\|_{L^2(\mathbb{R} \times A_R)}, \quad \|\psi\|_{LE[t_0, t_1]} = \sup_{R \geq R_0} \|\langle r \rangle^{-\frac{1}{2}} \psi\|_{L^2([t_0, t_1] \times A_R)},
\end{equation}
and its $H^1$ counterpart $LE^1$ in $\M_{[t_0,t_1]}^{ext}$:
\begin{equation}\label{de-LE1}
\|\psi\|_{LE^1[t_0, t_1]} =\|\p \psi\|_{LE[t_0, t_1]} +  \|\langle r \rangle^{-1}  \psi\|_{LE[t_0, t_1]}.
\end{equation}
Let $\Sigma_\tau := \{t =\tau\} \cap \{\ub \geq \delta\}$, $\tau \leq 0$, the energy is defined by
\begin{equation}\label{de-E-tau}
E[\psi] (\tau)= \int_{\Sigma_\tau} |\p \psi|^2 r^2 \di r \di \sigma_{S^2}.
\end{equation}
 
Primarily, we follow the method of \cite{MMDT-strich-sch-10, L-T-quasilinear-sch-18} to derive the following estimate:
\begin{equation}\label{energy-inequality-small-data-1}
E[\psi] (t_1) + \|\psi\|^2_{LE^1[t_0, t_1]}  \lesssim E[\psi] (t_0) + \doubleint_{\M_{[t_0,t_1]}^{ext}}  |\Box_g \psi | \left(|\p \psi| + \frac{|\psi|}{r} \right) + \int_{\Cb_\delta} |\p \psi|^2 + \frac{\psi^2}{r^2}.
\end{equation}
This is achieved by 
taking a multiplier of the form 
 $$X_\rho = C \p_t + f(r) \p_r, \quad q_\rho = \frac{f(r)}{r}, \quad f(r) = \frac{r}{r + \rho}.$$
Then a computation shows that (see \cite[Proposition 8]{M-T-09-decay-exterior} or \cite{L-T-quasilinear-sch-18})
$$\| \langle r \rangle^{-\frac{1}{2}} \psi\|^2_{L^2([t_0, t_1] \times A_\rho)} + \| \langle r \rangle^{-\frac{3}{2}} \psi\|^2_{L^2([t_0, t_1] \times A_\rho)} \lesssim \doubleint_{\M_{[t_0,t_1]}^{ext}} K^{X_\rho}(\psi,q_\rho),$$
where we refer to Section \ref{Energy estimates scheme} for the definition of the modified current $K^{X_\rho}(\psi,q_\rho)$.
Eventually, \eqref{energy-inequality-small-data-1} follows by applying the energy estimates scheme in the domain $\M_{[t_0,t_1]}^{ext}$ and taking the supremum over the dyadic sequence $\rho$.

We now outline the bootstrap argument. Let $\varphi$ be the unknown solution of the wave equation \eqref{Main Equation}. Given an integer $\Lambda \in \mathbb{N}$, we denote $\varphi_\Lambda :=\p^i \Omega^j \varphi, \, i+j=\Lambda$ and $\varphi_{\leq \Lambda} :=\p^i \Omega^j \varphi, \, i+j\leq \Lambda$. Picking a large integer $N \geq 36$, we assume for some large constant $M$ and $t \leq 0$,
\begin{equation}\label{bt-small-1}
E[\varphi_{\leq N}] (t) + \|\varphi_{\leq N}\|^2_{LE^1[t,0]} \leq M^2 \delta.
\end{equation}
Then, appealing to the standard Sobolev inequality $$ \langle r \rangle^2 |\p \psi_\Lambda|^2 \lesssim  \sum_{j \leq 2} E[\Omega^j \psi_{\leq \Lambda}] (t) + \sum_{i \leq 1} E[\p_r\Omega^i \psi_{\leq \Lambda}] (t),$$ we have
\begin{equation}\label{Sobolev-L-infty-small-1}
 |\p \varphi_\Lambda| \lesssim \langle r \rangle^{-1} M \delta^{\frac{1}{2}}, \quad \Lambda \leq N-2.
\end{equation}
Apply the energy inequality \eqref{energy-inequality-small-data-1} to $\varphi$, and take Theorem \ref{small-W-R1}  into account,
\begin{equation}\label{energy-inequality-solu-small-data-1}
E[\varphi] (t) + \|\varphi\|^2_{LE^1[t_0, t]}  \lesssim \delta + \doubleint_{\M_{[t_0,t]}^{ext}}  |\Box_g \varphi | \left(|\p \varphi| +|r^{-1} \varphi| \right).
\end{equation}
Due to the $L^\infty$ estimate \eqref{Sobolev-L-infty-small-1}, the last double integral in \eqref{energy-inequality-solu-small-data-1} is bounded by  (we only use the fact that $\Box_g \varphi$ is quadratic in $\p \varphi$) 
\begin{equation}\label{estimate-main-integral-small-1}
\begin{split}
\doubleint_{\M_{[t_0,t]}^{ext}}  |\p\varphi \p \varphi|  \left(|\p \varphi| +|r^{-1} \varphi| \right) & \lesssim \doubleint_{\M_{[t_0,t]}^{ext}} \langle r \rangle^{-1}  M \delta^{\frac{1}{2}} |\p \varphi| \left(|\p \varphi| +|r^{-1} \varphi| \right) \\
& \lesssim \delta^{\frac{1}{2}} M \|\varphi\|^2_{LE^1[t_0, t]},
\end{split}
\end{equation}
which can be absorbed by the left hand side of \eqref{energy-inequality-solu-small-data-1}, since $\delta$ is small enough. Thus we prove
\begin{equation}\label{energy-inequality-solu-0-small-data}
E[\varphi] (t) + \|\varphi\|^2_{LE^1[t_0, t]}  \lesssim \delta.
\end{equation}
Similar estimate holds with the $\varphi$ in \eqref{energy-inequality-solu-0-small-data} being replaced by $\p_t^i \Omega^j \varphi$, $i+j \leq N$. Besides, combining these with the ellliptic estimates, we have
\begin{align*}
E[\p \varphi] (t) &\lesssim \delta + E[\p_t \varphi](t) + E[\varphi](t) + \|\Box_g \varphi\|^2_{L^2(\Sigma_t)} \lesssim \delta, \\
 \|\p \varphi\|^2_{LE^1[t_0, t]}  & \lesssim \delta + E[\p \varphi] (t) +  \|\p_t\varphi\|^2_{LE^1[t_0, t]} +  \|\varphi\|^2_{LE^1[t_0, t]} + \|\Box_g \varphi\|^2_{LE[t_0, t]} \lesssim \delta,
\end{align*}
where we have used the estimates $ \|\Box_g \varphi\|^2_{L^2(\Sigma_t)} \lesssim \delta^{\frac{1}{2}}M E[\varphi] (t) $ and $ \|\Box_g \varphi\|^2_{LE[t_0, t]} \lesssim \delta^{\frac{1}{2}} M \|\varphi\|^2_{LE^1[t_0, t]}$. As a result, we obtain
\begin{equation}\label{energy-inequality-solu-1-killing-small-data}
E[\p^i  \Omega^j\varphi] (t_1) + \|\p^i \Omega^j\varphi\|^2_{LE^1[t_0, t_1]}  \lesssim \delta, \quad i+j \leq 1.
\end{equation}
The higher order energy bound can be carried out by induction and then the bootstrap argument is closed.

We next come to the subregion: $\{ 0\leq t \leq 1\} \cap \text{Region III}$, where the problem considered is reduced to a small data finite time existence theorem. This is of course well-known. Moreover, the multiplier $K$ and the Morawetz estimate are not needed for the proof. Note that, $\{ 0\leq t \leq 1\} \cap \text{Region III}$ is away from the event horizon as well, however it hits the photon sphere, which is nevertheless not an issue, for the Morawetz estimate is avoided and hence the trapping phenomenon does not take affect in the proof here.

We would also like to remark that in the above proofs, we only require the non-linearity to be quadratic. In other word, the null structure is not necessary for the proof of global existence in Region III.

By the Arzela-Ascoli Lemma, we can let $u_0 \rightarrow - \infty$ and prove that there exists a global (but not necessarily unique) solution in the region $\{ t \leq 1\} \cap \D^+(\mathcal{I}^-) \cap \D^+(\mathcal{H}^-)$, i.e., from the past null infinity and past event horizon up to $t=1$, see \cite[Section 5.3]{Wang-Yu-16} or \cite[Chapter $17$]{Christodoulou-09}. 


Reversing the time $t$, we conclude the scattering statement, i.e., Theorem \ref{Thm-scattering-region}.
\begin{remark}\label{rk-reverse-time}
If we reverse the time function $t$ to be $-t$, then the multiplier for the energy estimates is replaced by $ \delta^{-1} L +  \eta \Lb$. That is, taking $f_1 = \delta^{-1}$, $f_2 =  \eta.$ In view of \eqref{current of fL},
\begin{align*}
- \frac{1}{2}(\p_{\ub} f_1+\frac{\mu f_1}{r} ) |\nablaslash \psi|^2 &= - \delta^{-1} \frac{m}{r^2} |\nablaslash \psi|^2 \leq 0,\\
 \p_{\ub} f_2 g^{u \ub} |\Lb \psi|^2 & = - \frac{m}{r^2} |\Lb \psi|^2 \leq 0.
\end{align*}
Define the corresponding energy 
\begin{align*}
E_k(u) &:= \int_{C_u^{[\ub_0,\ub]}} \left(\delta^{-1} |L\varphi_k|^2 +  \eta |\nablaslash \varphi_k|^2 \right) \di \mu_{C_u},\\ \Eb_k(\ub) &:= \int_{\Cb_{\ub}^{[0, u]}} \left( \delta^{-1} |\nablaslash \varphi_k|^2 + \eta |\Lb \varphi_k|^2 \right) \di \mu_{\Cb_{\ub}}.
\end{align*}
 Let $-\delta \leq u_1 \leq u_2 \leq 0$ and $-1 \leq \ub_1 \leq \ub_2 \leq + \infty$ (thus $r \sim \langle \ub \rangle$ in this region). Then,
\begin{align*}
&E_k(u_2) +\Eb_k(\ub_2) - \doubleint_{\Do}  \langle \ub \rangle^{-2} \left( \delta^{-1}  |\nablaslash \varphi_k|^2  +  |\Lb \varphi_k|^2 \right) \eta  \di \mu_{\D} \\
\lesssim {}& E_k(u_1) +\Eb_k(\ub_1) + \doubleint_{\Do}  \langle \ub \rangle^{-1}   \left( \delta^{-1} |L \varphi_k \Lb \varphi_k |+ |\nablaslash \varphi_k|^2 \right) \eta  \di \mu_{\D}  \\
&+  \doubleint_{\Do} |\Box_g \varphi_k| \cdot  \left( \delta^{-1}  | L \varphi_k| + | \eta \Lb \varphi_k| \right) \eta \di \mu_{\D},
\end{align*}
where the current on the left hand side has the wrong sign. Nevertheless, if we consider the scattering problem: impose the short pulse data $\varphi_{+\infty}$ \eqref{def-psi-0} on $\mathcal{I}^+ (\ub \rightarrow +\infty)$ and on $\hori (u \rightarrow +\infty)$, $\varphi|_{\hori} \equiv0$, the above energy inequality turns out to be
\begin{align*}
&E_k(u) +\Eb_k(\ub) + \doubleint_{\D_{\ub, +\infty}^{u, +\infty}} \langle \ub \rangle^{-2}  \left( \delta^{-1} |\nablaslash \varphi_k|^2  +  |\Lb \varphi_k|^2 \right) \eta  \di \mu_{\D} \\
\lesssim {}& E_k(+\infty) +\Eb_k(+\infty) + \doubleint_{{\D_{\ub, +\infty}^{u, +\infty}}} \langle \ub \rangle^{-1} \left( \delta^{-1} |L \varphi_k \Lb \varphi_k |+  |\nablaslash \varphi_k|^2 \right) \eta \di \mu_{\D}  \\
&+  \doubleint_{{\D_{\ub, +\infty}^{u, +\infty}}} |\Box_g \varphi_k| \cdot \left( \delta^{-1}  | L \varphi_k| +   | \eta \Lb \varphi_k| \right) \eta \di \mu_{\D}.
\end{align*}
We can then prove along the line of Section \ref{sec-past} to show the global existence of a solution to the scattering problem. And indicated by the energy estimates, the scattering map $\mathcal{F}^+: C^\infty(\mathcal{I}^+) \times C^\infty(\hori) \, \rightarrow \, C^\infty(\Sigma_{0}) \times C^\infty(\Sigma_0)$ with $\mathcal{F}^+ (\varphi_{+\infty}, 0) = (\varphi|_{\Sigma_0}, \p_t \varphi|_{\Sigma_0})$ is bounded.
\end{remark}

\section{Global existence for the Cauchy problem}\label{sec-future}

Let $\R_2$ be the null strip $\D^+(\Sigma_{1}) \cap \D^-(\mathcal{I}^+) \cap \D^-(\hori) \cap  \{0 \leq \ub \leq \delta \}$.  In $\R_2$, $u \geq 1-\delta$, $u \sim t$ and $r$ is finite. 
Letting $1-\delta \leq u_1 < u$, $0\leq \ub_1 \leq \ub \leq \delta$, we define the degenerate energy and flux,
\begin{subequations}
\begin{align}
E^{deg} [\psi] (u; [\ub_1,\ub]) : =& \int_{C_u^{[\ub_1, \ub]}} \eta \left( |L \psi|^2 + \delta^{-1}  |\nablaslash \psi|^2 \right) \di \mu_{C_u}, \label{def-energy-Cu}\\
\Eb^{deg}[\psi](\ub; [u_1,u]) : = & \int_{\Cb_{\ub}^{[u_1, u]}} \left( \delta^{-1}   |\Lb \psi|^2 + \eta^2  |\nablaslash \psi|^2 \right)  \di \mu_{\Cb_{\ub}}, \label{def-energy-Cub} \\
{}^tF^{deg}[\psi](u; [\ub_1,\ub]) : = & \int_{C_{u}^{[\ub_1, \ub]}} \eta \left( |Y L \psi |^2 + \delta^{-2} |Y \psi |^2 \right) \di \mu_{C_u}. \label{def-energy-Cu-t}
\end{align}
\end{subequations}
At the same time, the non-degenerate energy and flux are defined as
\begin{subequations}
\begin{align}
E^{ndeg} [\psi] (u; [\ub_1,\ub]) : = & \int_{C_u^{[\ub_1, \ub]}} \left( |L \psi |^2 + \delta^{-1}  |\nablaslash \psi |^2 \right) \di \mu_{C_u}, \label{def-energy-ndeg-E} \\
\Eb^{ndeg} [\psi] (\ub; [u_1,u]) : = &  \int_{\Cb_{\ub}^{[u_1, u]}} \left( \delta^{-1} \eta^{-1}  |\Lb \psi |^2 + \eta  |\nablaslash \psi |^2 \right)  \di \mu_{\Cb_{\ub}}, \label{def-energy-ndeg-Eb} \\
{}^tF^{ndeg} [\psi] (u; [\ub_1,\ub])  : = &  \int_{C_{u}^{[\ub_1, \ub]}} \left( |Y L \psi |^2 + \delta^{-1} |Y \psi|^2 \right) \di \mu_{C_u}. \label{def-energy-ndeg-tF}
\end{align}
\end{subequations}

\begin{remark}\label{rk-degeneracy-energy}
Note that, the degenerate energy and flux vanish at the event horizon $\hori$. It will be apparent to see the degeneracy of $\Eb^{deg}$ and non-degeneracy of $\Eb^{ndeg}$ if we rewrite \eqref{def-energy-Cub} and \eqref{def-energy-ndeg-Eb} with the non-degenerate measure $\di \mu_{\Cb_{\ub}^{ND}} = r^2 \eta \di u \di \sigma_{S^2} = - r^2 \di r \di \sigma_{S^2}$:
\begin{align*}
\Eb^{deg}[\psi](\ub; [u_1,u]) = & \int_{\Cb_{\ub}^{[u_1, u]}}  \eta \left( \delta^{-1}   |Y \psi|^2 + |\nablaslash \psi|^2 \right)  \di \mu_{\Cb^{ND}_{\ub}}, \\
\Eb^{ndeg} [\psi] (\ub; [u_1,u]) = &  \int_{\Cb_{\ub}^{[u_1, u]}} \left( \delta^{-1}   |Y\psi |^2 + |\nablaslash \psi |^2 \right)  \di \mu_{\Cb^{ND}_{\ub}}.
\end{align*}
\end{remark}

Fixing $N \in \mathbb{N}$, $N\geq 6$, we set: for $l \leq N$, 
\begin{subequations}
\begin{align}
E^{deg}_l (u; [\ub_1,\ub]) &:=  E^{deg} [\varphi_l] (u; [\ub_1,\ub]), \label{def-energy-fL}\\
\Eb^{deg}_l (\ub; [u_1,u]) &:=  \Eb^{deg}[\varphi_l](\ub; [u_1,u]), \label{def-energy-fLb}\\
E^{ndeg}_l (u; [\ub_1,\ub]) &:=  E^{ndeg} [\varphi_l] (u; [\ub_1,\ub]), \label{def-energy-fL-fLb-ndeg-E}\\
\Eb^{ndeg}_l (\ub; [u_1,u]) &= \Eb^{ndeg} [\varphi_l] (\ub; [u_1,u]). \label{def-energy-fL-fLb-ndeg-Eb}
\end{align}
\end{subequations}
The degenerate flux is denoted by: for $k \leq N-1$,
\begin{subequations}
\begin{align}
{}^LF^{deg}_{k+1}  (u; [\ub_1, \ub]) &:= E^{deg} [\delta L \varphi_k] (u; [\ub_1,\ub]), \label{def-energy-flux-Cu-LF} \\
{}^{t}F^{deg}_{k+1} (u; [\ub_1, \ub]) &:={}^tF^{deg}[\varphi_k](u; [\ub_1,\ub]), \label{def-energy-flux-Cu-tF} \\
{}^{\Lb}F^{deg}_{k+1} (u; [\ub_1, \ub]) &:=E^{deg} [\Lb \varphi_k] (u; [\ub_1,\ub]), \label{def-energy-flux-Cu-LbF} \\ 
{}^L\Fb^{deg}_{k+1} (\ub; [u_1,u]) &:= \Eb^{deg}[\delta L \varphi_k](\ub; [u_1,u]),  \label{def-energy-flux-Cbub-LFb} \\ 
{}^{\Lb}\Fb^{deg}_{k+1} (\ub; [u_1,u]) &:=  \Eb^{deg}[\Lb \varphi_k](\ub; [u_1,u]). \label{def-energy-flux-Cbub-LbFb} 
\end{align}
\end{subequations}
And the non-degenerate flux is given by: for $k\leq N-1$,
\begin{subequations}
\begin{align}
{}^LF^{ndeg}_{k+1} (u; [\ub_1,\ub])  &:=  E^{ndeg} [\delta L \varphi_k] (u; [\ub_1,\ub]), \label{def-energy-flux-Cu-ndeg-LF}\\  
{}^tF^{ndeg}_{k+1} (u; [\ub_1,\ub])  &:= {}^tF^{ndeg} [\varphi_k] (u; [\ub_1,\ub]), \label{def-energy-flux-Cu-ndeg-tF}\\ 
{}^YF^{ndeg}_{k+1} (u; [\ub_1,\ub])  &:=\Eb^{ndeg} [Y \varphi_k] (u; [\ub_1, \ub]), \label{def-energy-flux-Cu-ndeg-YF} \\ 
{}^L\Fb^{ndeg}_{k+1} (\ub; [u_1,u]) &:=\Eb^{ndeg} [\delta L \varphi_k] (\ub; [u_1,u]), \label{def-energy-flux-Cbub-ndeg-LFb}\\
{}^{Y}\Fb^{ndeg}_{k+1} (\ub; [u_1,u]) &:=\Eb^{ndeg} [Y \varphi_k] (\ub; [u_1,u]). \label{def-energy-flux-Cbub-ndeg-YFb}
\end{align}
\end{subequations}

Define the degenerate integrated energy
\begin{align}
\mathcal{S}^{deg} [\psi] (\D) := & \doubleint_{\D}  \left( \delta^{-1} | \Lb \psi |^2 +\delta^{-1} |\nablaslash \psi |^2 + |L \psi |^2 \right) \eta \di \mu_{\D}, \label{def-energy-spacetime-fLb-fL-deg-1} 
\end{align}
and the non-degenerate integrated energy
\begin{align}
 \mathcal{S}^{ndeg} [\psi] (\D) := & \doubleint_{\D}  \left( \delta^{-1} | Y \psi |^2 +\delta^{-1} |\nablaslash \psi |^2 + |L \psi |^2 \right) \eta \di \mu_{\D}. \label{def-energy-spacetime-fLb-fL-deg}
\end{align}

We set for $l \leq N, \, k \leq N-1$,  
\begin{align*}
 \mathcal{S}^{deg}_l (\D) := &  \mathcal{S}^{deg} [\varphi_l] (\D), \quad \quad\quad \,\,\,\, \mathcal{S}^{ndeg}_l (\D): = \mathcal{S}^{ndeg} [\varphi_l] (\D),\\
{}^L\mathcal{S}^{deg}_{k+1} (\D) :=  & \mathcal{S}^{deg} [\delta L \varphi_k] (\D),  \quad\quad {}^{\Lb}\mathcal{S}^{deg}_{k+1} (\D) : = \mathcal{S}^{deg} [\Lb \varphi_k] (\D),\\
{}^L\mathcal{S}^{ndeg}_{k+1} (\D) :=  & \mathcal{S}^{ndeg} [\delta L \varphi_k] (\D),  \quad\, {}^{Y}\mathcal{S}^{ndeg}_{k+1} (\D) := \mathcal{S}^{ndeg} [Y \varphi_k] (\D).
\end{align*}

At the end of this section, we will prove the degenerate energy  decay  and the non-degenerate energy bound in $\R_2$. 
\begin{theorem}\label{Main-Thm-R2}
There are the degenerate decay estimates in $\R_2$: $1-\delta \leq u \leq +\infty$, $0\leq \ub \leq \delta$,
\begin{align*}
\mathcal{S}^{deg}_l (\D_{0, \ub}^{u, +\infty}) + E^{deg}_l (u; [0, \ub])+ \Eb^{deg}_l (\ub; [u, +\infty])  
&\lesssim I_{N+1}^2  |u|^{-2\beta}, \quad  l \leq N, \\
{}^LF^{deg}_{k+1} (u; [0, \ub]) + {}^L\Fb^{deg}_{k+1} (\ub; [u, +\infty]) & \lesssim I_{N+1}^2  |u|^{-2\beta}, \quad k \leq N-1,\\
 {}^tF^{deg}_{k+1} (u; [0, \ub]) +  {}^{\Lb}F_{k+1}^{deg} (u; [0, \ub] )  + {}^{\Lb}\Fb_{k+1}^{deg}(\ub; [u, +\infty] )   &\lesssim  I_{N+1}^2  |u|^{-2\beta}, \quad  k \leq N-1,\\
 {}^L\mathcal{S}^{deg}_{k+1} (\D_{0, \ub}^{u, +\infty}) +  {}^{\Lb}\mathcal{S}^{deg}_{k+1} (\D_{0, \ub}^{u, +\infty}) 
& \lesssim I_{N+1}^2|u|^{-2\beta}, \quad   k \leq N-1,
\end{align*}
where $\beta \geq \frac{1}{2}$. Besides,  we have the non-degenerate energy bound in $\R_2$:
\begin{align*}
\mathcal{S}^{ndeg}_l (\D_{0, \ub}^{u, +\infty}) + E^{ndeg}_l (u; [0,\ub]) + \Eb^{ndeg}_l (\ub; [u, +\infty]) & \lesssim I_{N+1}^2, \quad  l \leq N, \\
 {}^LF^{ndeg}_{k+1} (u; [0,\ub]) + {}^L\Fb^{ndeg}_{k+1} (\ub; [u, +\infty]) & \lesssim I_{N+1}^2, \quad k \leq N-1, \\
   {}^tF^{ndeg}_{k+1} (u; [0,\ub])  + {}^YF^{ndeg}_{k+1} (u; [0, \ub]) +  {}^Y\Fb^{ndeg}_{k+1} (\ub; [u, +\infty])  & \lesssim I^2_{N+1}, \quad  k \leq N-1,\\
   {}^L\mathcal{S}^{ndeg}_{k+1} (\D_{0, \ub}^{u, +\infty}) +  {}^{Y}\mathcal{S}^{ndeg}_{k+1} (\D_{0, \ub}^{u, +\infty}) 
& \lesssim I_{N+1}^2, \quad   k \leq N-1.
\end{align*}
\end{theorem}

Before the energy argument, we will introduce a wealth of notations.
For a vector field $V$, let $\iint_{\D}  |V \Box_g \varphi_k|^2 \eta^i \di \mu_{\D} \lesssim {}^V\mathcal{S}^k + {}^V\mathcal{G}^k + \mathcal{L}^k$ where
\begin{align}
{}^V\mathcal{S}^k &=  \doubleint_{\D} \sum_{p+q \leq k, p \leq q}  |Q(\p V \varphi_q, \p \varphi_p)|^2 \eta^i \di \mu_{\D}, \label{def-V-mS-k}\\
{}^V\mathcal{G}^k &=  \doubleint_{\D} \sum_{p+q \leq k, p<q}  | Q(\p V \varphi_p, \p \varphi_q)|^2  \eta^i \di \mu_{\D}, \label{def-V-mG-k}
\end{align}
and the lower order term $ \mathcal{L}^k$ takes the form of \eqref{def-V-mS-k} with $V=1$, and ${}^V\mathcal{S}^k =  {}^VS^k_1+ \cdots +  {}^VS^k_4$, with 
\begin{subequations}
\begin{align}
{}^VS^k_1 &=  \doubleint_{\D} \sum_{p+q\leq k, p\leq q} |D \varphi_{p}|^2 |Y V  \varphi_{q}|^2  \eta^i \di \mu_{\D},\label{def-V-S-1} \\
{}^VS^k_2 &=  \doubleint_{\D}  \sum_{p+q\leq k, p\leq q} |\Db \varphi_{p}|^2 |LV  \varphi_{q}|^2 \eta^i \di \mu_{\D},\label{def-V-S-2} \\
{}^VS^k_3 &=  \doubleint_{\D}\sum_{p+q\leq k, p\leq q} |\Db \varphi_{p}|^2 |\nablaslash V  \varphi_{q}|^2 \eta^i \di \mu_{\D},\label{def-V-S-3} \\
{}^VS^k_4 &=  \doubleint_{\D} \sum_{p+q\leq k, p\leq q} |L\varphi_{p}|^2 |\nablaslash V \varphi_{q}|^2 \eta^i \di \mu_{\D},\label{def-V-S-4}
\end{align}
\end{subequations}
and ${}^V\mathcal{G}^k= {}^VG^k_1 + \cdots + {}^VG^k_4$, where for $p+q \leq k, \, p<q$,
\begin{subequations}
\begin{align}
{}^VG^k_1 &=  \doubleint_{\D}  \sum_{p+q \leq k, p < q}  |D\varphi_{q}|^2  |Y V\varphi_{p}|^2  \eta^i \di \mu_{\D}, \label{def-V-G-i-1} \\
{}^VG^k_2 &=  \doubleint_{\D}  \sum_{p+q \leq k, p < q}    |\Db \varphi_{q}|^2 | L V\varphi_{p}|^2 \eta^i  \di \mu_{\D},\label{def-V-G-i-2} \\
{}^VG^k_3 &=  \doubleint_{\D}  \sum_{p+q \leq k, p < q}   |\Db\varphi_{q}|^2 |\nablaslash V\varphi_{p}|^2  \eta^i \di \mu_{\D},\label{def-V-G-i-3} \\
{}^VG^k_4 &=  \doubleint_{\D}   \sum_{p+q \leq k, p < q}  |L\varphi_{q}|^2 |\nablaslash V\varphi_{p} |^2  \eta^i \di \mu_{\D}. \label{def-V-G-i-4}
\end{align}
\end{subequations}

\subsection{Initial data in $\R_2$}\label{sec-data-2}
We restrict the solution $\varphi$ obtained in $\R_1$ (Section \ref{sec-past}) to the portion of the Cauchy surface $\Sigma^{[0, \delta]}_{1} = \{t=1\} \cap \R_1$. Note that, here $\R_1 = \D^{+}(\mathcal{I}^{-}) \cap \D^+(\mathcal{H}^-) \cap \{u \leq 1\} \cap \{0 \leq \ub \leq \delta \}$.
In regard of the Cauchy problem of \eqref{Main Equation} with the initial data $( \varphi|_{\Sigma^{[0,\delta]}_1}, \partial_t \varphi|_{\Sigma^{[0,\delta]}_1})$, its solution restricted to $\{t \geq 1\} \cap \R_1$ exactly coincides with $\varphi |_{\{t \geq 1\} \cap \R_1}$ (obtained in Section \ref{sec-past}) by the uniqueness.
And hence on $\{u=1| 0 \leq \ub \leq \delta\}$, the solution of the Cauchy problem, which we also denote by $\varphi$, obeys the following estimates: 
\begin{align*}
E^{deg}[\delta^p W^l_{p,q}\varphi_k](1; [0,\delta]) & \lesssim  I^2_{l+k}, \quad \text{on} \quad C_{1}, \\
 \delta^{-1} \|D W^l_{p,q} \varphi_k\|^2_{L^2(S_{\delta, 1})} & \lesssim  I^2_{l+k}, \quad \text{on} \quad S_{\delta,1}.
\end{align*}
In addition, the data on $\Sigma_1 \cap \{\ub \leq 0\}$ is set to be trivial, and we know that $\varphi \equiv 0$ in $\{\ub \leq 0\} \cap \{t \geq 1\}$. 

For any $1 \leq u_1 \leq u, \, 0 \leq \ub_1 \leq \ub \leq \delta$, we shall also use the short cut $C_u$ for  $C_u^{[\ub_1, \ub]}$ and $\Cb_{\ub}$ for $\Cb_{\ub}^{[u_1, u]}$.

\subsection{Bootstrap argument in $\R_2$}\label{sec-bt-argument-forward}
\subsubsection{Bootstrap assumptions in $\R_2$}\label{sec-bt-forward}
We now address the bootstrap assumptions. Given any number $\beta \geq  \frac{1}{2}$ and $N \in \mathbb{N}, \, N \geq 6$, we assume that, there is a large constant $M$ to be determined, such that for $0 \leq \ub \leq \delta$, $1 \leq u \leq +\infty$,
\begin{align*}
E^{deg}_l (u; [0, \ub]) & \leq  M^2  |u|^{-2\beta}, \quad l \leq N,\\
{}^LF^{deg}_{k+1} (u; [0, \ub]) + {}^{t}F^{deg}_{k+1} (u; [0, \ub])+{}^{\Lb}F^{deg}_{k+1} (u; [0, \ub]) & \leq M^2  |u|^{-2\beta}, \quad k \leq N-1,\\
E^{ndeg}_l(u; [0,\ub]) &\leq M^2, \quad\quad \quad\quad l \leq N,\\
{}^LF^{ndeg}_{k+1}(u; [0,\ub]) + {}^tF^{ndeg}_{k+1}(u; [0,\ub]) + {}^YF^{ndeg}_{k+1}(u; [0,\ub]) &\leq M^2, \quad\quad \quad\quad k \leq N-1.
\end{align*}
That is, for the degenerate energy and flux: let $ l \leq N$, $k \leq N-1$,
\begin{subequations}
\begin{align}
  \|\eta^{\frac{1}{2}} L \varphi_l\|_{L^2(C_u)} +  \delta^{ - \frac{1}{2} } \|\eta^{\frac{1}{2}}\nablaslash \varphi_l\|_{L^2(C_u)} &\leq M  |u|^{-\beta},   \label{bt-forward-energy-Cu-deg}\\
  \delta^{-1} \| \eta^{ \frac{1}{2}} Y \varphi_k\|_{L^2(C_u)}  + \|\eta^{ \frac{1}{2}} Y L  \varphi_k\|_{L^2(C_u)} + \delta^{-\frac{1}{2} } \|\eta^{\frac{1}{2}} \nablaslash \Lb \varphi_k\|_{L^2(C_u)} & \leq  M  |u|^{-\beta}, \label{bt-forward-flux} \\
  \delta \|\eta^{\frac{1}{2}} L^2 \varphi_k\|_{L^2(C_u)} +  \delta^{\frac{1}{2} } \|\eta^{\frac{1}{2}}\nablaslash L \varphi_k\|_{L^2(C_u)} & \leq M  |u|^{-\beta}; \label{bt-forward-flux-out-deg}
\end{align}
 \end{subequations}
And for the non-degenerate energy and flux:  let $ l \leq N$, $k \leq N-1$,
 \begin{subequations}
\begin{align}
\|L \varphi_l\|_{L^2(C_u)} +  \delta^{ - \frac{1}{2} } \|\nablaslash \varphi_l\|_{L^2(C_u)} & \leq M, \label{bt-forward-energy-Cu-nondeg} \\
 \delta^{-1} \|Y\varphi_k\|_{L^2(C_u)} + \|Y L \varphi_k\|_{L^2(C_u)} + \delta^{-\frac{1}{2} } \| \nablaslash Y \varphi_k\|_{L^2(C_u)} & \leq M, \label{bt-forward-flux-nondeg}  \\
 \delta \|L^2 \varphi_k\|_{L^2(C_u)} +  \delta^{ \frac{1}{2} } \|\nablaslash L \varphi_k\|_{L^2(C_u)} & \leq M.   \label{bt-forward-flux-out-nondeg}
\end{align}
 \end{subequations} 
 In addition, we make the following bootstrap assumption for the degenerate integrated energy: letting $ l \leq N, \, k \leq N-1$,
\begin{equation}\label{bt-decay-energy-spacetime-fLb-fL-deg}
\mathcal{S}^{deg}_l (\D_{0, \ub}^{u, +\infty})  +   {}^L \mathcal{S}^{deg}_{k+1} (\D_{0, \ub}^{u, +\infty})  +   {}^{\Lb} \mathcal{S}^{deg}_{k+1} (\D_{0, \ub}^{u, +\infty})  
\leq  M^2  |u|^{-2\beta}.
\end{equation}
As a remark, bootstrap assumption for non-degenerate integrated energy is not needed for our proof.

\subsubsection{Close the bootstrap argument in $\R_2$}\label{sec-close-bt}
We also let 
\begin{equation}\label{def-mathbb-I}
\mathbb{I}_k^2 = I^2_k +   \delta^{\frac{1}{2}} M^{4}.
\end{equation}
As in Section \ref{sec-close-bt-R1},
we will finally choose $M$ (which depends on the initial data) large enough such that $C I^2_{N+1} \leq \frac{M^2}{4}$, and $\delta$ small enough such that $\delta^{\frac{1}{2}}M^{{2}} \ll 1$, hence $C\mathbb{I}^2_{N+1} \leq \frac{M^2}{2},$
which will close the bootstrap argument (refer to the discussions below) and we will complete the proof for Theorem \ref{Main-Thm-R2}.

To close the bootstrap argument in $\R_2$, we will start with the following degenerate decay estimates.
\begin{theorem}\label{Thm-decay-low-deg-energy}
Suppose $\beta \geq \frac{1}{2}$ and $1 \leq u \leq +\infty, \, 0\leq \ub \leq \delta$. There are the decay estimates for  the degenerate energies
\begin{align}
E^{deg}_l (u; [0, \ub]) + \mathcal{S}^{deg}_l (\D_{0, \ub}^{u, +\infty})
\lesssim {}& \mathbb{I}_{N+1}^2  |u|^{-2\beta}, \quad l \leq N, \label{decay-deg-estimate-fLb-fL-thm} \\
{}^LF^{deg}_{k+1} (u; [0, \ub])  + {}^L\mathcal{S}^{deg}_{k+1} (\D_{0, \ub}^{u, +\infty})  \lesssim {}& \mathbb{I}_{N+1}^2  |u|^{-2\beta}, \quad k \leq N-1, \label{decay-deg-estimate-fLb-fL-L-thm} \\
\Eb^{deg}_l (\ub; [u, +\infty]) 
\lesssim {}&  \mathbb{I}_{N+1}^2|u|^{-2\beta}, \quad  l \leq N, \label{decay-estimate-Lb-Cub} \\
 {}^L\Fb^{deg}_{k+1} (\ub; [u, +\infty]) 
\lesssim {}&  \mathbb{I}_{N+1}^2  |u|^{-2\beta}, \quad  k \leq N-1.\label{decay-deg-Lb-L-Cub}
\end{align}
\end{theorem}
The proof of Theorem \ref{Thm-decay-low-deg-energy} will be presented in the sections \ref{sec-deg-multiplier-II}, \ref{sec-deg-energy-Cu} and \ref{sec-deg-flux}.

We will take advantage of Theorem \ref{Thm-decay-low-deg-energy} to obtain the estimate for the flux associated to $\Lb$.
\begin{theorem}\label{prop-Lb2-deg}
There is, for any $\beta \geq \frac{1}{2}$ and $1 \leq u \leq +\infty, \, 0\leq \ub \leq \delta$,
\begin{align}
  {}^{\Lb}F_{k+1}^{deg} (u; [0, \ub] ) +  {}^{\Lb}\mathcal{S}^{deg}_{k+1} (\D_{0, \ub}^{u, +\infty}) &\lesssim \mathbb{I}_{N+1}^2  |u|^{-2\beta}, \quad  k \leq N-1. \label{decay-Lb-Cu-Lb-phi}\\
{}^{\Lb}\Fb_{k+1}^{deg}(\ub; [u, +\infty] ) 
&\lesssim  \mathbb{I}_{N+1}^2  |u|^{-2\beta}, \quad  k \leq N-1. \label{decay-estimate-Lb-Cub-Lb-phi}
\end{align}
\end{theorem}
Theorem \ref{prop-Lb2-deg} will be inferred in Section \ref{sec-Lb-energy-away-hori}.

\begin{remark}\label{rk-thm-deg-low}
Integrating \eqref{decay-estimate-Lb-Cub}, \eqref{decay-deg-Lb-L-Cub} and \eqref{decay-estimate-Lb-Cub-Lb-phi}  along $\p_{\ub}$ leads to
\begin{equation}\label{decay-spacetime-estimate-Lb-LbL-Lb2}
\doubleint_{\D_{0, \ub}^{u, +\infty}} \left( \delta^{-1} |\Lb \varphi_l|^2  + \delta |\Lb L \varphi_k|^2 +  \delta^{-1} |\Lb^2 \varphi_k|^2 \right) \di \mu_{\D} \lesssim \delta \mathbb{I}_{N+1}^2|u|^{-2\beta}, 
\end{equation}
where $l \leq N, \, k \leq N-1$.
Compared to the spacetime integrals $\mathcal{S}^{deg}_l (\D_{0, \ub}^{u, +\infty})$ and ${}^L\mathcal{S}^{deg}_{k+1} (\D_{0, \ub}^{u, +\infty})$, ${}^{\Lb}\mathcal{S}^{deg}_{k+1} (\D_{0, \ub}^{u, +\infty})$, there is no $\eta$ in the integrands of \eqref{decay-spacetime-estimate-Lb-LbL-Lb2} and the $\delta$-size of the spacetime decay estimates for $\Lb \varphi_l, \Lb L \varphi_k , \Lb^2 \varphi_k$ are improved.

The improved spacetime estimate \eqref{decay-spacetime-estimate-Lb-LbL-Lb2} will be useful in bounding the exterior currents (spacetime integrals away from the horizon) ${}^e\mathcal{C} (\varphi_l)$, ${}^e\mathcal{C} (L\varphi_k)$ and ${}^e\mathcal{C} (Y\varphi_k)$, which arise when conducting the non-degenerate energy estimates near the horizon, see Section \ref{sec-horizon}.
\end{remark}

Finally, the estimate for ${}^tF^{deg}_{k+1} (u; [0,\ub]), \, k \leq N-1$ will be retrieved in Section \ref{sec-recover-Lb-deg}.

Next, we will proceed to the non-degenerate energy estimates near the horizon.
Denote the region near the horizon by 
\begin{equation}\label{def-R2-NH}
\R_2^{NH} :=\R_2 \cap \{ 2m \leq r \leq r_{NH}\},
\end{equation}
where $r_{NH}$ satisfying $2m<r_{NH} <1.2 r_{NH} <3m$, is close to $2m$. 

\begin{theorem}\label{Thm-decay-low-hori-energy}
In $\R_2^{NH}$, $0 - r^\ast_{NH} < u \leq + \infty$, $0 \leq \ub \leq \delta$, there is
\begin{align}
E^{ndeg}_l (u; [0,\ub]) + \mathcal{S}^{ndeg}_l(\R_2^{NH}) \lesssim {}&  \mathbb{I}_{N+1}^2, \quad l \leq N, \label{decay-hori-estimate-lower-flux-thm} \\
 {}^LF^{ndeg}_{k+1} (u; [0,\ub]) + {}^L\mathcal{S}^{ndeg}_{k+1}(\R_2^{NH})  \lesssim {}& \mathbb{I}_{N+1}^2, \quad k \leq N-1. \label{decay-hori-estimate-L-lower-flux-thm}
\end{align}
And letting $u^{NH} = \ub - r_{NH}^\ast$, we have
\begin{align}
\Eb^{ndeg}_l (\ub; [u^{NH}, +\infty])  
\lesssim {}& \mathbb{I}_{N+1}^2, \quad l \leq N, \label{decay-hori-estimate-Eb-thm} \\
{}^L\Fb^{ndeg}_{k+1} (\ub; [u^{NH}, +\infty])  
\lesssim {}& \mathbb{I}_{N+1}^2, \quad k \leq N-1. \label{decay-hori-estimate-Fb-thm}
\end{align}
\end{theorem}
 The proof of this theorem will be given in the sections \ref{sec-out-energy-horizon}, \ref{sec-estimate-L-2}, and \ref{sec-top-energy-ndeg}.

After that, we shall make use of Theorem \ref{Thm-decay-low-hori-energy} to  prove the bound for the flux associated to $Y$.
\begin{theorem}\label{Thm-decay-low-hori-energy-flux-Lb}
In $\mathcal{R}_2^{NH}$, $0 - r^\ast_{NH} < u \leq + \infty$, $0 \leq \ub \leq \delta$, letting $u^{NH} = \ub -r^\ast_{NH}$, we have,
\begin{align}
  {}^YF^{ndeg}_{k+1} (u; [0,\ub]) +  {}^Y\Fb^{ndeg}_{k+1} (\ub; [u^{NH}, +\infty]) +  {}^Y\mathcal{S}^{ndeg}_{k+1}(\R_2^{NH})   \lesssim {}&\mathbb{I}_{N+1}^2, \quad  k \leq N-1. \label{decay-hori-estimate-fLb-fL-flux-thm-Y-Cu}
\end{align}
\end{theorem}
Theorem \ref{Thm-decay-low-hori-energy-flux-Lb} will be proved in Section \ref{sec-high-de-hori}.

At last, the estimate for ${}^tF^{ndeg}_{k+1} (u; [0,\ub]), \, k \leq N-1$ will be deduced in Section \ref{sec-recover-Lb-deg}.

To facilitate our estimates, we present some preliminary estimates which follow from the bootstrap assumptions \eqref{bt-forward-energy-Cu-deg}-\eqref{bt-forward-flux-out-nondeg} and \eqref{bt-decay-energy-spacetime-fLb-fL-deg}.
\begin{proposition}\label{Prop-Sobolev-R2}
In $\R_2$, we have the non-degenerate estimates:
\begin{align*}
 \delta^{\frac{1}{2}} \| L\varphi_k \|_{L^{4}(S_{\ub,u})} + \delta^{-\frac{1}{4}} \| \Db \varphi_k\|_{L^{4}(S_{\ub,u})}
 &\lesssim M  , \quad k \leq N-1,\\
 \delta^{\frac{1}{2}}  \|L\varphi_j\|_{L^\infty(\R_2)} + \delta^{-\frac{1}{4}} \|\Db \varphi_j\|_{L^\infty(\R_2)} 
 &\lesssim M  , \quad j \leq N-2,
 \end{align*}
and the degenerate decay estimates:
 \begin{align*}
\delta^{\frac{1}{2}} \| \eta^{\frac{1}{2}} L\varphi_k \|_{L^{4}(S_{\ub,u})} + \delta^{-\frac{1}{4}} \| \eta^{\frac{1}{2}}\Db \varphi_k\|_{L^{4}(S_{\ub,u})} & \lesssim   |u|^{-\beta} M  , \quad k \leq N-1,\\
 \delta^{\frac{1}{2}}  \|\eta^{\frac{1}{2}} L\varphi_j\|_{L^\infty(\R_2)} + \delta^{-\frac{1}{4}} \|\eta^{\frac{1}{2}} \Db \varphi_j\|_{L^\infty(\R_2)} 
 &\lesssim   |u|^{-\beta} M  , \quad j \leq N-2.
 \end{align*}
\end{proposition}

\begin{proof}
The proof is based on the Sobolev inequalities on $C_u$ and $S_{\ub, u}$ \eqref{Sobolev Inequlities-S2}-\eqref{Sobolev Inequlities-Cu-deg}. 
As the proof leading to Proposition \ref{proposition L infinity and L4 estimates in R1}, we will only address the case for $\eta^{\frac{1}{2}} \Lb\varphi_k$. By \eqref{Sobolev Inequlities-Cu-deg}, there is, for $k \leq N-1$,
\begin{align*}
r^{\frac{1}{2}} \| \eta^{\frac{1}{2}}\Lb \varphi_k \|_{L^{4}(S_{\ub,u})} &\lesssim \| \eta^{\frac{1}{2}} L \Lb \varphi_k \|^{\frac{1}{2}}_{L^{2}(C_{u})} (\|\eta^{\frac{1}{2}} \Lb \varphi_k \|^{\frac{1}{2}}_{L^{2}(C_{u})} + \|\eta^{\frac{1}{2}} r \nablaslash \Lb \varphi_k \|^{\frac{1}{2}}_{L^{2}(C_{u})}) \\
 &\lesssim M^{\frac{1}{2}}  \langle u \rangle^{-\frac{\beta}{2}} \cdot  \delta^{\frac{1}{4}} M^{\frac{1}{2}} \langle u \rangle^{-\frac{\beta}{2}} \lesssim \delta^{\frac{1}{4}}  \langle u \rangle^{-\beta} M,
 \end{align*}
where we note that for $k \leq N-1$, $\| \eta^{\frac{1}{2}} L \Lb \varphi_k \|_{L^{2}(C_{u})}$ is controlled by ${}^tF_{k+1}^{deg} (u, \ub)$, while the bound of $\| \eta^{\frac{1}{2}} r \nablaslash \Lb \varphi_k \|_{L^{2}(C_{u})}$ should be related to  the bootstrap assumption for ${}^{\Lb}F_{k+1}^{deg} (u, \ub)$, see \eqref{bt-forward-flux}.
The $L^\infty$ estimate $\| \eta^{\frac{1}{2}} \Lb \varphi_j \|_{L^\infty(S_{\ub, u})}, \, j \leq N-2$ follows from \eqref{Sobolev Inequlities-S2} and the above $L^4$ estimates.
\end{proof}

\begin{remark}\label{rk-Prop-Sobolev-R2}
It is worth to mention that, better estimates for lower order derivatives of $Y\varphi_k$ or $\Lb \varphi_k$ (which will not be used throughout our proof) can be derived:
 \begin{align*}
\| |u|^{\beta} \eta^{\frac{1}{2}}  \Lb \varphi_p \|_{L^{4}(S_{\ub,u})} + \| Y \varphi_p \|_{L^{4}(S_{\ub,u})} &\lesssim \delta^{\frac{1}{2}} M, \quad p \leq N-2, \\
\|  |u|^{\beta} \eta^{\frac{1}{2}}  \Lb  \varphi_q\|_{L^\infty(\R_2)} + \| Y \varphi_q\|_{L^\infty(\R_2)} &\lesssim \delta^{\frac{1}{2}}  M, \quad q \leq N-3.
 \end{align*}
The $\delta^{\frac{1}{4}}$ lose in the estimates for the top order $\| Y \varphi_{N-1} \|_{L^{4}(S_{\ub,u})}$ and $\| |u|^{\beta} \eta^{\frac{1}{2}}  \Lb \varphi_{N-1} \|_{L^{4}(S_{\ub,u})}$ is due to the weaker assumption for the top order energy  $\|Y \varphi_N\|_{L^2(C_u)}$ and $\|\eta^{\frac{1}{2}}\Lb \varphi_N\|_{L^2(C_u)}$, or equivalently $\| \nablaslash Y \varphi_{N-1}\|_{L^2(C_u)}$ and $\|\eta^{\frac{1}{2}} \nablaslash \Lb \varphi_{N-1}\|_{L^2(C_u)}$. As shown in \eqref{bt-forward-flux-nondeg} and \eqref{bt-forward-flux}, compared to the lower order bootstrap assumption $\||u|^\beta   \eta^{ \frac{1}{2}} Y \varphi_k\|_{L^2(C_u)} + \|Y\varphi_k\|_{L^2(C_u)} \leq \delta M, \, k \leq N-1$, the one for the top order case $ \||u|^\beta \eta^{\frac{1}{2}}  \nablaslash \Lb \varphi_{N-1}\|_{L^2(C_u)} + \| \nablaslash Y \varphi_{N-1} \|_{L^2(C_u)} \leq\delta^{\frac{1}{2} } M$ is weaker.

In contrast to \cite{Wang-Yu-16}, the Sobolev inequality on $\Cb_{\ub}$ is not good enough for application here, because
\begin{equation*}
r^{\frac{1}{2}}\| \Lb \varphi \|_{L^{4}(S_{\ub,u})} \lesssim r^{\frac{1}{2}}_0 \| \Lb \varphi \|_{L^{4}(S_{\ub,1})} + \|\Lb^2 \varphi \|^{\frac{1}{2}}_{L^{2}(\Cb^{[1, u]}_{\ub})}(\| \Lb \varphi \|^{\frac{1}{2}}_{L^{2}(\Cb^{[1, u]}_{\ub})} + \|\Omega \Lb \varphi \|^{\frac{1}{2}}_{L^{2}(\Cb^{[1, u]}_{\ub})}),
 \end{equation*}
does not offer any decay rates in terms of $|u|$, since $r$ is finite  in $\R_2$. Here $r^\ast_0 = \ub -1$.
\end{remark}

\subsection{Degenerate energy in $\R_2$}\label{sec-deg-energy}
At the first stage, we devote ourselves to the degenerate energy estimates: Theorem \ref{Thm-decay-low-deg-energy} and Theorem \ref{prop-Lb2-deg}.
Let $1 \leq u_1 \leq u_2 \leq +\infty$, and $0\leq \ub_1 \leq \ub_2 \leq \delta$. 
We should remind ourselves that $r$ has a uniformly upper bound  in $\R_2$ and $r \geq 2m$.

\subsubsection{The multiplier in the region $\R_2$}\label{sec-deg-multiplier-II}

Let us consider the multiplier $\xi =\xi_2 := \eta L +  \delta^{-1} (1+\mu) \Lb$. That is, we choose $f_1=\eta$ and $f_2= \delta^{-1} (1+\mu)$, so that 
\begin{align*}
\p_u f_1  g^{u \ub} |L\psi|^2&= \frac{m}{r^2} |L\psi|^2 >0, \\
\p_{\ub} f_2 g^{u \ub} |\Lb \psi|^2 &= \delta^{-1} \frac{m}{r^2} |\Lb \psi|^2 >0, \\
- \frac{1}{2} \left(\p_u f_2- \frac{\mu f_2}{r} \right) |\nablaslash \psi|^2 &= \delta^{-1} \frac{2m^2}{r^3} |\nablaslash \psi|^2 >0.
\end{align*}
Therefore, by virtue of \eqref{current of fL} and the energy identity \eqref{energy-identity}, we get some extra positive spacetime integrals which is crucial in the proof. The energy inequality takes the following form (irrelevant constants are ignored),
\begin{equation}\label{energy-ineq-fLb-fL-deg-psi}
\begin{split}
& \int_{C_{u_2}^{[\ub_1, \ub_2]}} \left( |L \psi|^2 + \delta^{-1} |\nablaslash \psi|^2\right) \eta \di \mu_{C_u} + \int_{\Cb_{\ub_2}^{[u_1, u_2]}} \left(\eta^2  |\nablaslash \psi|^2 +  \delta^{-1} |\Lb \psi|^2 \right)\di \mu_{\Cb_{\ub}} \\
& + \doubleint_{\Do}  \left( \delta^{-1} | \Lb \psi|^2 +\delta^{-1} |\nablaslash \psi|^2 + |L \psi|^2 \right) \eta \di \mu_{\D} \\
\lesssim &  \int_{C_{u_1}^{[\ub_1, \ub_2]}} \left( |L \psi|^2 + \delta^{-1} |\nablaslash \psi|^2 \right) \eta \di \mu_{C_u} + \int_{\Cb_{\ub_1}^{[u_1, u_2]}} \left( \eta^2  |\nablaslash \psi|^2 +  \delta^{-1} |\Lb \psi|^2\right) \di \mu_{\Cb_{\ub}}\\
& \quad \quad \quad \quad + \mathcal{C} (\psi) + \F (\psi),
\end{split}
\end{equation}
where the current $\mathcal{C} (\psi)$ is given by
\begin{equation}\label{def-C-L-Lb-psi}
 \mathcal{C}(\psi) =   \doubleint_{\Do} \left(  |\nablaslash \psi|^2 + \delta^{-1}  |L \psi \Lb \psi|  \right) \eta \di \mu_{\D},
\end{equation}
and the nonlinear error term $\F (\psi)$ is given as below,
\begin{equation}\label{def-Er-L-Lb-psi}
\F (\psi) =  \doubleint_{\Do} |\Box_g \psi| \cdot \left(| L \psi| \eta^2 + \delta^{-1}  |\Lb \psi | \eta \right) \di \mu_{\D}.
\end{equation}

For the current $\mathcal{C} (\psi)$,
\begin{equation}\label{error-quad-nabla-psi}
 \doubleint_{\Do}  |\nablaslash \psi|^2 \eta \di \mu_{\D} \leq \delta \doubleint_{\Do} \delta^{-1} |\nablaslash \psi|^2  \eta \di \mu_{\D},
\end{equation}
can be absorbed by the spacetime integrals on the left hand side of \eqref{energy-ineq-fLb-fL-deg-psi};
\begin{equation}\label{error-quad-LbL-psi}
\begin{split}
& \doubleint_{\Do} \delta^{-1}  |L \psi \Lb \psi|  \eta \di \mu_{\D}\\
 \lesssim {}& c \doubleint_{\Do}  |L \psi|^2 \eta^2 \di \mu_{\D} +  \int_{\ub_1}^{\ub_2} c^{-1} \delta^{-1} \di \ub  \int_{\Cb_{\ub}^{[u_1, u_2]}} \delta^{-1} | \Lb \psi|^2  \di \mu_{\Cb_{\ub}}.
\end{split}
\end{equation}
Here $c$ is a small constant to be determined.
Meanwhile, we estimate $\F (\psi)$ by
\begin{equation}\label{source-psi}
\begin{split}
& |\F  (\psi)|  \lesssim  \doubleint_{\Do} c^{-1} |\Box_g \psi|^2 \eta^3 \di \mu_{\D} + c \doubleint_{\Do}  |L \psi|^2 \eta \di \mu_{\D} \\
& \quad + c\doubleint_{\Do}  |\Box_g \psi|^2 \eta^2 \di \mu_{\D} +  \int_{\ub_1}^{\ub_2} c^{-1} \delta^{-1} \di \ub  \int_{\Cb_{\ub}^{[u_1, u_2]}} \delta^{-1} | \Lb \psi|^2  \di \mu_{\Cb_{\ub}}. 
\end{split}
\end{equation}
We choose $c \ll 1$ so that $c \iint_{\Do}  |L \psi|^2 \eta \di \mu_{\D}$ can be absorbed by the positive integrals on the left hand side of \eqref{energy-ineq-fLb-fL-deg-psi}, while the last terms in \eqref{error-quad-LbL-psi} and \eqref{source-psi} can be handled by the Gr\"{o}nwall's inequality.  
As a consequence, we deduce
\begin{equation}\label{energy-ineq-fLb-fL-deg}
\begin{split}
& E^{deg} [\psi] (u_2; [\ub_1,\ub_2])  + \Eb^{deg} [\psi] (\ub_2; [u_1, u_2]) + \mathcal{S}^{deg} [\psi] (\Do)\\
\lesssim {}& E^{deg} [\psi] (u_1; [\ub_1,\ub_2])  + \Eb^{deg} [\psi] ( \ub_1, [u_1, u_2] ) + \F_1(\psi) + \F_2(\psi),
\end{split}
\end{equation}
where $c \ll 1$ is a constant to be determined  and
\begin{equation}\label{def-F-1-2-psi}
\F_1(\psi) \lesssim \doubleint_{\Do} c^{-1} |\Box_g \psi|^2 \eta^3 \di \mu_{\D}, \quad \F_2(\psi) \lesssim  c \doubleint_{\Do}  |\Box_g \psi|^2 \eta^2 \di \mu_{\D}.
\end{equation}
This kind of energy inequality \eqref{energy-ineq-fLb-fL-deg}-\eqref{def-F-1-2-psi} will come into play in the energy estimate for the top order case, see Section \ref{sec-top-energy-deg}. 
Alternatively, without lost of generality, we have as well
\begin{equation}\label{energy-ineq-fLb-fL-deg-psi-sign}
\begin{split}
& E^{deg} [\psi] (u_2; [\ub_1,\ub_2])  + \Eb^{deg} [\psi] (\ub_2; [u_1, u_2]) + \mathcal{S}^{deg} [\psi] (\Do) \\
\lesssim {}& E^{deg} [\psi] (u_1; [\ub_1,\ub_2])  + \Eb^{deg} [\psi] ( \ub_1, [u_1, u_2] ) +\doubleint_{\Do}  |\Box_g \psi|^2 \eta^2 \di \mu_{\D}.
\end{split}
\end{equation}

\subsubsection{Energy estimates for $E^{deg}_k (u; [0,\ub])$, $\mathcal{S}^{deg}_k(\D^{u,+\infty}_{0,\ub})$ and $\Eb^{deg}_k (\ub; [u,+\infty]), \, k \leq N-1$}\label{sec-deg-energy-Cu}
Taking $\psi = \varphi_k, \, k \leq N-1$ in \eqref{energy-ineq-fLb-fL-deg-psi-sign}, we obtain the energy inequality,
\begin{equation}\label{energy-ineq-fLb-fL-varphi-deg}
\begin{split}
& E^{deg}_k(u_2; [\ub_1,\ub_2])  + \Eb^{deg}_k(\ub_2; [u_1, u_2]) \\
& + \doubleint_{\Do}  \left( \delta^{-1} | \Lb \varphi_k|^2 +\delta^{-1} |\nablaslash \varphi_k|^2 + |L \varphi_k|^2 \right) \eta \di \mu_{\D} \\
\lesssim {}& E^{deg}_k(u_1; [\ub_1,\ub_2])  + \Eb^{deg}_k( \ub_1, [u_1, u_2] )+ \doubleint_{\Do}  |\Box_g \varphi_k|^2 \eta^2 \di \mu_{\D}. 
\end{split}
\end{equation}
The last term is split as $ \iint_{\Do}  |\Box_g \varphi_k|^2 \eta^2 \di \mu_{\D}= S^k_1 + \cdots + S^k_4$, where $S^k_j,\, j=1:4$ are defined as \eqref{def-V-S-1}-\eqref{def-V-S-4} with $\D=\Do, \,V=1$, $i=2$, i.e., for $p+q\leq k \leq N-1$,  $p \leq q$,
\begin{subequations}
\begin{align}
S^k_1 &=  \doubleint_{\Do} |D \varphi_{p}|^2  |Y \varphi_{q}|^2 \eta^2 \di \mu_{\D}, \label{def-S-k-i-1}\\
 S^k_2 &=  \doubleint_{\Do}    |\Db \varphi_{p}|^2 |L \varphi_{q}|^2 \eta^2 \di \mu_{\D},\label{def-S-k-i-2} \\
  S^k_3 &=  \doubleint_{\Do}   |\Db \varphi_{p}|^2 |\nablaslash \varphi_{q}|^2 \eta^2 \di \mu_{\D},\label{def-S-k-i-3} \\
 S^k_4 &=  \doubleint_{\Do}   |L\varphi_{p}|^2 |\nablaslash\varphi_{q}|^2  \eta^2 \di \mu_{\D}.\label{def-S-k-i-4}
\end{align}
\end{subequations}
Note that, we have chosen $N\geq 6,$ so that $p +2 \leq [\frac{N}{2}]+2 \leq N-1$. Hence, we can apply $L^\infty, L^\infty, L^2, L^2$ to the four factors in each term of \eqref{def-S-k-i-1}-\eqref{def-S-k-i-4}.

For $S^k_1$, due to the $L^\infty$ estimate  $| D \varphi_{p}|^2 \lesssim \delta^{-1} M^2$, 
\begin{equation}\label{estimate-Sk1}
|S^k_1| \lesssim  \int_{\ub_1}^{\ub_2}  M^2 \di \ub  \int_{\Cb_{\ub}^{[u_1, u_2]}} \delta^{-1} |\Lb \varphi_q|^2 \di \mu_{\Cb_{\ub}}, \quad q \leq k.
\end{equation}
for which the Gr\"{o}nwall's inequality applies.

Knowing that $|\Db \varphi_{p}|  \lesssim \delta^{\frac{1}{2}} M^2$, and by the bootstrap assumption for $\mathcal{S}_k^{deg}(\D_{0, \ub}^{u, +\infty})$, there is
\begin{equation*}
|S^k_2| + |S^k_3| \lesssim  \delta^{\frac{1}{2}} M^2 \doubleint_{\Do}  \left( |L \varphi_q|^2 + |\nablaslash \varphi_q|^2 \right) \eta^2 \di \mu_{\D}  \lesssim \delta^{\frac{1}{2}} M^4 |u_1|^{-2\beta}, \quad q \leq k.
\end{equation*}

For $S^k_4$, we note that $p \leq q \leq k\leq N-1$, thus we can all apply $L^4$ norm to the four factors. Knowing that $\|\eta^{\frac{1}{2}} L \varphi_{p} \|^2_{L^4(S_{u, \ub})} \lesssim \delta^{-1}  |u|^{-2\beta} M^2$, $\| \eta^{\frac{1}{2}} \nablaslash \varphi_{q} \|^2_{L^4(S_{u, \ub})} \lesssim \delta^{\frac{1}{2}}  |u|^{-2\beta} M^2,$
\begin{equation}\label{esti-S4}
\begin{split}
|S^k_4| \lesssim & \int_{u_1}^{u}   \int_{\ub_1}^{\ub_2}  \|\eta^{\frac{1}{2}} L \varphi_{p} \|^2_{ L^4(S_{u, \ub})} \|\eta^{\frac{1}{2}} \nablaslash \varphi_{q}\|^2_{L^4(S_{u, \ub})} \di u \di \ub \\
 \lesssim{}& \delta^{\frac{1}{2}} M^{4}  |u_1|^{-4\beta + 1}, \quad p \leq q \leq k\leq N-1.
\end{split}
\end{equation}
We here remark that, for the top order case: $k=N$, $\| \eta^{\frac{1}{2}} \nablaslash \varphi_{N} \|^2_{L^4(S_{u, \ub})}$ is not bounded  because of the regularity and hence the estimate \eqref{esti-S4} is no longer valid if $k=N$.

All the above estimates together with the Gr\"{o}nwall's inequality lead to: in the case of $k \leq N-1,$
\begin{equation}\label{energy-ineq-fLb-fL-lower-1}
\begin{split}
&E^{deg}_k(u_2; [\ub_1, \ub_2])  + \Eb^{deg}_k(\ub_2; [u_1, u_2]) + \mathcal{S}^{deg}_k(\Do)\\
\lesssim {}&E^{deg}_k(u_1; [\ub_1, \ub_2])  + \Eb^{deg}_k(\ub_1; [u_1, u_2])  +   \delta^{\frac{1}{2}} M^{4}  |u_1|^{-2\beta },
\end{split}
\end{equation}
where we consider $\beta \geq \frac{1}{2}$, so that $-4\beta+1 \leq -2\beta$.
In particular, letting $\ub_1 =0$ ($\varphi \equiv 0$ on $\Cb_0$) and $0 <\ub_2 =u \leq \delta$, we have for any $1 \leq u_1 < u_2,$
\begin{equation}\label{energy-ineq-fLb-fL-lower-E}
\begin{split}
&E^{deg}_k(u_2; [0, \ub]) + \int_{u_1}^{u_2} E^{deg}_k(u; [0, \ub]) \di u\\
\lesssim {}&E^{deg}_k(u_1; [0, \ub]) +   \delta^{\frac{1}{2}} M^{4}  |u_1|^{-2\beta }, \quad k \leq N-1.
\end{split}
\end{equation}
By the pigeon-hole principle (see Lemma \ref{lema-pigeonhole-1}), we achieve that for any $\beta \geq \frac{1}{2}$ and $1 \leq u, \, 0 \leq \ub \leq \delta$,
\begin{equation}\label{decay-estimate-L-Cu-lower}
E^{deg}_k(u; [0, \ub]) \lesssim  \mathbb{I}_N^2  |u|^{-2\beta}, \quad  k \leq N-1.
\end{equation}
Letting $u_1 = u, \, u_2 \rightarrow + \infty$ and $\ub_1=0, \, 0<\ub_2 =\ub \leq \delta$  in \eqref{energy-ineq-fLb-fL-lower-1} gives rise to
\begin{align*}
& \Eb^{deg}_k(\ub; [u, +\infty]) + \mathcal{S}^{deg}_k(\D^{u,+\infty}_{0,\ub}) \\
 \lesssim {}& E^{deg}_k(u; [0, \ub]) +  \delta^{\frac{1}{2}} M^{4}  |u|^{-2\beta }, \quad k \leq N-1.
\end{align*}
Substituting \eqref{decay-estimate-L-Cu-lower} into the above formula, we deduce
\begin{equation}\label{decay-estimate-Lb-Cub-lower-deg}
 \Eb^{deg}_k(\ub; [u, +\infty]) + \mathcal{S}^{deg}_k(\D^{u,+\infty}_{0,\ub}) 
\lesssim  \mathbb{I}_N^2|u|^{-2\beta}, \quad k \leq N-1.
\end{equation}
\eqref{decay-estimate-L-Cu-lower} together with \eqref{decay-estimate-Lb-Cub-lower-deg} asserts estimates for  the lower order energy $E^{deg}_k(u; [0, \ub])$, $\Eb^{deg}_k(\ub; [u, +\infty]) $ and $\mathcal{S}^{deg}_k(\D^{u,+\infty}_{0,\ub})$, $k \leq N-1$.

By the way, we can insert the estimate for $\Eb^{deg}_k(\ub; [u, +\infty])$ \eqref{decay-estimate-Lb-Cub-lower-deg} into \eqref{estimate-Sk1} and derive
\begin{equation}\label{estimate-Box-varphi-low-deg}
\doubleint_{\Do}  |\Box_g \varphi_k|^2 \eta^2 \di \mu_{\D}
\lesssim   \delta^{\frac{1}{2}} M^4 |u_1|^{-2\beta}, \quad k \leq N-1,
\end{equation}

\subsubsection{Energy estimates for ${}^LF^{deg}_{k+1} (u; [0,\ub])$, ${}^L\mathcal{S}^{deg}_{k+1}(\D^{u,+\infty}_{0,\ub})$ and ${}^L\Fb^{deg}_{k+1} (\ub; [u, +\infty])$, $k \leq N-1$}\label{sec-deg-flux}

We take $\psi = \delta L\varphi_k$, $k \leq N-1$ in \eqref{energy-ineq-fLb-fL-deg-psi-sign} to derive
\begin{equation}\label{energy-ineq-fLb-fL-L}
\begin{split}
& {}^LF^{deg}_{k+1} (u_2; [\ub_1,\ub_2]) + {}^L\Fb^{deg}_{k+1} (\ub_2; [u_1,u_2]) \\ 
& + \doubleint_{\Do}  \left( \delta | \Lb L \varphi_k|^2 + \delta |\nablaslash L \varphi_k|^2 + \delta^2 |L^2 \varphi_k|^2 \right) \eta \di \mu_{\D} \\
\lesssim &  {}^LF^{deg}_{k+1} (u_1; [\ub_1,\ub_2]) + {}^L\Fb^{deg}_{k+1} (\ub_1; [u_1,u_2]) + {}^LW^{k} + \doubleint_{\Do} \delta^2 |L \Box_g \varphi_k|^2 \eta^2 \di \mu_{\D},
\end{split}
\end{equation}
where ${}^LW^{k} $ associated to $ [\Box_g, \delta L] \varphi_k$ is given by
\begin{equation}\label{def-W-L-1}
{}^LW^{k} =  \doubleint_{\Do} \delta^2   \left( |L\varphi_k|^2 +|\Lb\varphi_k|^2 +|\laplacianslash \varphi_k |^2 +  |\Box_g \varphi_k|^2 \right) \eta^2 \di \mu_{\D},
\end{equation}
and the last term can be split as: $$\doubleint_{\Do} \delta^2 |L \Box_g \varphi_k|^2 \eta^2 \di \mu_{\D} = {}^L\mathcal{S}^k + {}^L\mathcal{G}^k + {}^\delta\mathcal{L}^k.$$ Here ${}^L\mathcal{S}^k, {}^L\mathcal{G}^k $ take the forms of \eqref{def-V-mS-k}-\eqref{def-V-mG-k} with $\D= \Do, V=\delta L$, $i=2$; ${}^\delta\mathcal{L}^k$ is defined as \eqref{def-V-mS-k} with $\D= \Do, V=\delta$, $i=2$.

At first, \eqref{estimate-Box-varphi-low-deg} tells that ${}^\delta\mathcal{L}^k \lesssim \delta^{\frac{5}{2}}   M^4 |u_1|^{-2\beta }, \, k \leq N-1.$

For the error terms ${}^L\mathcal{S}^k$, we make the further splitting: ${}^L\mathcal{S}^k = {}^LS^k_1 + \cdots + {}^LS^k_4$,  where $ {}^LS^k_j,\, j=1:4$ are defined as \eqref{def-V-S-1}-\eqref{def-V-S-4} with $\D= \Do, V=\delta L, \, i=2$.
The estimates for the ${}^LS^k_j, \, j =1 : 3$ are the same as that for $S^k_j, \, j=1: 3$ \eqref{def-S-k-i-1}-\eqref{def-S-k-i-3}, except that $\varphi_{q}$ therein is replaced now by $\delta L \varphi_{q}$.
For the remaining one ${}^LS^k_4$, it  reads
\begin{align*}
{}^LS^k_4 &=  \doubleint_{\Do}  \sum_{p+q\leq k, p\leq q} \delta^2 |L\varphi_{p}|^2 |\nablaslash L\varphi_{q}|^2  \eta^2 \di \mu_{\D}, \quad k \leq N-1.
\end{align*}
By the bootstrap assumption \eqref{bt-forward-energy-Cu-deg}, noting that $q+1 \leq N$, 
and the $L^\infty$ estimate $ \eta^\frac{1}{2} |L\varphi_{p}|  \lesssim \delta^{-\frac{1}{2}} M |u|^{-\beta},$ $p \leq N/2 \leq N-3$, 
\begin{align*}
{}^LS^k_4 &\lesssim  \int_{u_1}^{u_2} \delta^2 \| \eta^\frac{1}{2} L\varphi_{p}\|^2_{L^\infty} \|  \eta^\frac{1}{2}  L\varphi_{q+1}\|^2_{L^2(C_u)}  \di u \\
&\lesssim \int_{u_1}^{u_2} \delta  M^{4}   |u|^{-4\beta} \di u \lesssim  \delta  M^{4} |u_1|^{-4\beta + 1}, \quad k \leq N-1.
\end{align*}

As for ${}^L\mathcal{G}^k$, we make the following splitting: ${}^L\mathcal{G}^k= {}^LG^k_1 + \cdots + {}^LG^k_4$, where ${}^LG^k_j, \, j=1:4$ are defined as \eqref{def-V-G-i-1}-\eqref{def-V-G-i-4} with $\D= \Do, \, V=\delta L$, $i=2$, i.e., for $p+q \leq k \leq N-1, \, p<q$,
\begin{subequations}
\begin{align}
{}^LG^k_1 &=  \doubleint_{\Do}  \delta^2   |D\varphi_{q}|^2  |Y L\varphi_{p}|^2  \eta^2 \di \mu_{\D}, \label{L-Gk-1} \\
{}^LG^k_2 &=  \doubleint_{\Do}   \delta^2 |\Db \varphi_{q}|^2 | L^2\varphi_{p}|^2 \eta^2  \di \mu_{\D}, \label{L-Gk-2} \\
{}^LG^k_3 &=  \doubleint_{\Do}  \delta^2  |\Db\varphi_{q}|^2 |\nablaslash L\varphi_{p}|^2  \eta^2 \di \mu_{\D}, \label{L-Gk-3} \\
{}^LG^k_4 &=  \doubleint_{\Do}  \delta^2 |L\varphi_{q}|^2 |\nablaslash L\varphi_{p} |^2  \eta^2 \di \mu_{\D}. \label{L-Gk-4} 
\end{align}
\end{subequations}
We note that $p+q \leq k \leq N-1, \, p<q,$ then $q \leq N-1, \, p \leq k-1 \leq N-2$. We can always perform $L^4$ norm to the four factors in each term above.

For ${}^LG^k_1$, by the a-priori estimate $\| D \varphi_{q} \|_{L^4(S_{\ub,u})} \lesssim \delta^{-\frac{1}{2}} M,$ $q \leq N-1,$ 
\begin{align*}
|{}^LG^k_1|  & \lesssim  \int_{\ub_1}^{\ub_2} \int_{u_1}^{u_2} \delta^{-1}M^2 \cdot \delta^2 \|\Lb L \varphi_{p} \|^2_{L^4(S_{\ub,u})} \di u \di \ub \\
&\lesssim  \delta M^2  \int_{\ub_1}^{\ub_2} \int_{u_1}^{u_2}    \sum_{p \leq i\leq p+1} \|\Lb L \varphi_{i} \|^2_{L^2(S_{\ub,u})} \di u \di \ub \\
& \lesssim  \int_{\ub_1}^{\ub_2}  M^2 \di \ub \int_{\Cb_{\ub}} \sum_{p \leq i \leq p+1} \delta |\Lb L \varphi_{i} |^2 \di \mu_{\Cb_{\ub}}, \quad p \leq k-1,
\end{align*}
where we have used the Sobolev inequality \eqref{Sobolev-L2-L4-S2} and the fact that $r$ is finite in Region $\R_2$ in the second inequality. Hence we can apply the Gr\"{o}nwall's inequality.

For ${}^LG^k_2, {}^LG^k_3$, knowing that $ \|\Db \varphi_{q} \|_{L^4(S_{\ub,u})} \lesssim \delta^{\frac{1}{4}} M $, $q \leq N-1$, and $p \leq N-2$, we have similarly,
\begin{align*}
|{}^LG^k_2| + |{}^LG^k_3|  \lesssim & \int_{\ub_1}^{\ub_2} \int_{u_1}^{u_2}  \delta^{\frac{5}{2}}M^2  \left(  \| L^2\varphi_{p} \|^2_{L^4(S_{\ub,u})} + \| \nablaslash L \varphi_{p} \|^2_{L^4(S_{\ub,u})} \right) \eta^2 \di u \di \ub  \\
\lesssim {}& \delta^{\frac{1}{2}} M^2 \doubleint_{\Do}  \sum_{k\leq N-1} \delta^2 \left( | L^2 \varphi_{k}|^2 + | \nablaslash L \varphi_{k}|^2 \right) \eta^2 \di \mu_{\D}  \lesssim \delta^{\frac{1}{2}} M^4 |u_1|^{-2\beta}, 
\end{align*}
where we have used the bootstrap assumption for ${}^L\mathcal{S}^{deg}_{k+1}(\D^{u,+\infty}_{0,\ub})$, $k \leq N-1$.

For the last one, by the a-priori estimate $\eta \| L \varphi_{q} \|^2_{L^4(S_{\ub,u})} \lesssim \delta^{-1} M^2 |u|^{- 2\beta},$ $q \leq N-1$, and $p \leq N-2$,
\begin{align*}
|{}^LG^k_4| & \lesssim  \int_{u_1}^{u_2} \int_{\ub_1}^{\ub_2} \delta^2 \|\eta^{\frac{1}{2}} L \varphi_{q} \|^2_{L^4(S_{\ub,u})}   \|\eta^{\frac{1}{2}} \nablaslash L \varphi_{p} \|^2_{L^4(S_{\ub,u})} \di \ub \di u \\
&\lesssim   \int_{u_1}^{u_2}  \int_{\ub_1}^{\ub_2} \delta M^2 |u|^{- 2\beta}   \sum_{p \leq i\leq p+1}   \|\eta^{\frac{1}{2}} \nablaslash L \varphi_{i} \|^2_{L^2(S_{\ub,u})} \di \ub \di u \\
& \lesssim \int_{u_1}^{u_2}\delta M^2 |u|^{- 2\beta} \int_{C_u}  \sum_{p \leq i\leq p+2}   | L \varphi_{i} |^2 \eta \di \mu_{C_u}   \lesssim \delta M^4 |u_1|^{-4\beta+1},
\end{align*}
where in the last inequality, the bootstrap assumption \eqref{bt-forward-energy-Cu-deg} is used.

Finally, noting the bootstrap assumption for $\mathcal{S}^{deg}_l (\D_{0, \ub}^{u, +\infty})$, $l\leq N$ \eqref{bt-decay-energy-spacetime-fLb-fL-deg} and the estimate \eqref{estimate-Box-varphi-low-deg}, 
\begin{align*}
 |{}^LW^k |  \lesssim {}&  \delta^2 \doubleint_{\Do} \left( |L \varphi_k|^2 +  |\Lb \varphi_k|^2 + | \nablaslash \varphi_{k+1}|^2 + |\Box_g \varphi_k|^2  \right)  \eta^2 \di \mu_{\D}  \\
\lesssim {}& \delta^{2} M^2 |u_1|^{-2 \beta} + \delta^{\frac{5}{2}} M^4 |u_1|^{-2 \beta}, \quad k \leq N-1. 
\end{align*}

In conclusion,  we have proved: for $k \leq N-1$, $0 \leq \ub_1 < \ub_2 \leq \delta$, $1 \leq u_1 < u_2 < +\infty$, and $\beta \geq \frac{1}{2}$,
\begin{equation*}
\begin{split}
& {}^LF^{deg}_{k+1} (u_2; [\ub_1,\ub_2]) + {}^L\Fb^{deg}_{k+1} (\ub_2; [u_1,u_2]) + {}^L\mathcal{S}^{deg}_{k+1} (\Do)\\
\lesssim &{}^LF^{deg}_{k+1} (u_1; [\ub_1,\ub_2]) + {}^L\Fb^{deg}_{k+1}(\ub_1; [u_1,u_2]) +  \delta^{\frac{1}{2}}  M^4 |u_1|^{-2\beta}.
\end{split}
\end{equation*}
Following the argument for \eqref{decay-estimate-L-Cu-lower}, \eqref{decay-estimate-Lb-Cub-lower-deg},  we can deduce the estimates for the flux ${}^LF^{deg}_{k+1} (u; [0,\ub])$, ${}^L\mathcal{S}^{deg}_{k+1}(\D^{u,+\infty}_{0,\ub})$ and ${}^L\Fb^{deg}_{k+1} (\ub; [u, +\infty])$, $k \leq N-1$, i.e., \eqref{decay-deg-estimate-fLb-fL-L-thm} and \eqref{decay-deg-Lb-L-Cub}.

Now, the improvement for $E^{deg}_k(u; [0,\ub]), \, k \leq N-1$ \eqref{decay-estimate-L-Cu-lower} together with the improved flux ${}^LF^{deg}_{k+1} (u; [0,\ub])$, $k \leq N-1$ \eqref{decay-deg-estimate-fLb-fL-L-thm} 
 yields the  enhanced $L^\infty$ estimate, 
\begin{equation}\label{L-infty-L-phi-deg-improved-1}
 \|\eta^{\frac{1}{2}} L \varphi_k\|_{L^\infty(\R_2)} \lesssim \delta^{-\frac{1}{2}} \mathbb{I}_{k+3}  |u| ^{-\beta}, \quad k \leq N-3,
\end{equation}
which will help to estimate the top order energy $E^{deg}_N(u; [0,\ub])$ and $\mathcal{S}^{deg}_N (\D_{0, \ub}^{u, +\infty})$, $\Eb^{deg}_N(\ub; [u,+\infty])$.

\subsubsection{Energy estimates for $E^{deg}_N(u; [0,\ub])$ and $\mathcal{S}^{deg}_N (\D_{0, \ub}^{u, +\infty})$, $\Eb^{deg}_N(\ub; [u,+\infty])$}\label{sec-top-energy-deg}

As explained before, the estimate for $S^k_4, \, k\leq N-1$ \eqref{esti-S4} is not allowed when 
$k=N$. However, we can combine the improvement \eqref{L-infty-L-phi-deg-improved-1} with the refined energy inequality \eqref{energy-ineq-fLb-fL-deg}-\eqref{def-F-1-2-psi} to linearize $S_4^N$. 
We take $\psi = \varphi_N$ in \eqref{energy-ineq-fLb-fL-deg}-\eqref{def-F-1-2-psi}, then the error terms are
\begin{equation*}
\begin{split}
\F_1(\varphi_N) \lesssim &\doubleint_{\Do} c^{-1} |\Box_g \varphi_N|^2 \eta^3 \di \mu_{\D}, \quad \F_2(\varphi_N) \lesssim  c\doubleint_{\Do}  |\Box_g \varphi_N|^2 \eta^2 \di \mu_{\D},
\end{split}
\end{equation*}
where $c \ll 1$ is a constant to be determined. Analogous to the case of $k \leq N-1$, there is the decomposition \eqref{def-S-k-i-1}-\eqref{def-S-k-i-4} for $\F_1(\varphi_N)$, $\F_2(\varphi_N)$. And $S_i^N,$ $i=1:3$, can be handled in the same way as \eqref{def-S-k-i-1}-\eqref{def-S-k-i-3} previously, while $S_4^N$ taking the form of
 \begin{align*}
S^N_4= & \doubleint_{\Do}  \sum_{p+q\leq N, p\leq q} c^{-1} |\eta^{\frac{1}{2}} L \varphi_{p}|^2 |\nablaslash \varphi_{q} |^2 \eta^2 \di \mu_{\D} \\
&+ \doubleint_{\Do}  \sum_{p+q\leq N, p\leq q} c | \eta^{\frac{1}{2}} L \varphi_{p}|^2 |\nablaslash \varphi_{q} |^2 \eta \di \mu_{\D},
\end{align*}
should be treated differently.  In view of the improvement  
\eqref{L-infty-L-phi-deg-improved-1},
\begin{align*}
|S^N_4| \lesssim & \int_{\ub_1}^{\ub_2} c^{-1} \delta^{-1} \mathbb{I}_{N}^2 \di \ub \int_{\Cb_{\ub}^{[u_1, u_2]}}  |\nablaslash \varphi_{N}|^2 \eta^2 \di \mu_{\Cb_{\ub}} + \doubleint_{\Do}  c\delta^{-1} \mathbb{I}_{N}^2 |\nablaslash \varphi_{N}|^2 \eta \di \mu_{\D} .
\end{align*}
We additionally require $c \ll 1$ so that the second term can be absorbed by $\mathcal{S}^{deg}_N(\Do)$ on the left hand side the top order energy inequality, and the first term can be handled by the Gr\"{o}nwall's inequality. We here recall \eqref{def-mathbb-I} for the definition of $ \mathbb{I}_{N}^2$ and note that $\eta^3$ in $\F_1(\varphi_N)$  is crucial.

Therefore, we end up with the energy inequality ($\beta \geq \frac{1}{2}$)
\begin{equation}\label{energy-ineq-fLb-fL-top-deg}
\begin{split}
&E^{deg}_N(u_2; [\ub_1,\ub_2] )  + \Eb^{deg}_N(\ub_2; [u_1, u_2] ) + \mathcal{S}^{deg}_N(\Do)\\
\lesssim {}&E^{deg}_N(u_1; [\ub_1,\ub_2] ) + \Eb^{deg}_N(\ub_1; [u_1, u_2] )  +   \delta^{\frac{1}{2}} M^{4} |u_1|^{-2\beta}.
\end{split}
\end{equation}
Proceeding in an analogous way as that in Section \ref{sec-deg-energy-Cu} and taking the previously lower order results into account, we complete the energy estimates \eqref{decay-deg-estimate-fLb-fL-thm}, \eqref{decay-estimate-Lb-Cub} and  Theorem \ref{Thm-decay-low-deg-energy}.

As a consequence, the $L^\infty$ estimate \eqref{L-infty-L-phi-deg-improved-1} is upgraded as: for $k\leq N-1, \, j \leq N-2$, there is
\begin{equation}\label{L-4-L-phi-deg-improved}
 \|\eta^{\frac{1}{2}} L \varphi_j\|_{L^\infty(\R_2)} + \|\eta^{\frac{1}{2}} L \varphi_k\|_{L^4(S_{\ub, u} \cap \R_2)} \lesssim \delta^{-\frac{1}{2}} \mathbb{I}_{N+1} |u|^{-\beta}.
\end{equation}

\subsubsection{Energy estimates for ${}^{\Lb}F_{k+1}^{deg} (u; [0, \ub] )$, ${}^{\Lb}\mathcal{S}^{deg}_{k+1} (\D_{0, \ub}^{u, +\infty})$ and $ {}^{\Lb}\Fb_{k+1}^{deg}(\ub; [u, +\infty] )$, $k \leq N-1$}\label{sec-Lb-energy-away-hori}
In this section, we will make use of Theorem \ref{Thm-decay-low-deg-energy} and the resulted improvement \eqref{L-4-L-phi-deg-improved} to prove Theorem \ref{prop-Lb2-deg}.
 
 \begin{proof}[Proof of Theorem \ref{prop-Lb2-deg}]
We take $\psi = \Lb \varphi_k$, $k \leq N-1$ in \eqref{energy-ineq-fLb-fL-deg-psi-sign} to derive, 
\begin{align*}
& {}^{\Lb}F_{k+1}^{deg} (u_2; [\ub_1,\ub_2]) + {}^{\Lb}\Fb_{k+1}^{deg} (\ub_2; [u_1,u_2]) \\ 
& + \doubleint_{\Do}  \left( \delta^{-1} | \Lb^2 \varphi_k|^2 +\delta^{-1} |\nablaslash \Lb \varphi_k|^2 + |L \Lb \varphi_k|^2 \right) \eta \di \mu_{\D} \\
\lesssim {}& {}^{\Lb}F_{k+1}^{deg} (u_1; [\ub_1,\ub_2]) + {}^{\Lb}\Fb_{k+1}^{deg} (\ub_1; [u_1, u_2]) + {}^{\Lb}\mathcal{W}^k +  \doubleint_{\Do} | \Lb \Box_g \varphi_k|^2 \eta^2 \di \mu_{\D},
\end{align*}
where ${}^{\Lb}\mathcal{W}^k$ is associated to $ [\Box_g, \Lb] \varphi_k$,
\begin{equation*}
{}^{\Lb}\mathcal{W}^k = \doubleint_{\Do}  \left( |L\varphi_k|^2 +|\Lb\varphi_k|^2 +|\laplacianslash \varphi_k |^2 +  |\Box_g \varphi_k|^2 \right) \eta^2 \di \mu_{\D},
\end{equation*}
and
\begin{equation*}
\doubleint_{\Do} | \Lb \Box_g \varphi_k|^2 \eta^2 \di \mu_{\D} = {}^{\Lb} \mathcal{S}^k + {}^{\Lb} \mathcal{G}^k  +  \mathcal{L}^k.
\end{equation*}
Here ${}^{\Lb}\mathcal{S}^k, {}^{\Lb}\mathcal{G}^k $ take the forms of \eqref{def-V-mS-k}-\eqref{def-V-mG-k} with $\D=\Do, V=\Lb$, $i=2$, and $\mathcal{L}^k$ is defined as \eqref{def-V-mS-k} with $\D=\Do, V=1,\, i=2$.
 
Appealing to \eqref{estimate-Box-varphi-low-deg}, we get $ \mathcal{L}^k \lesssim \delta^{\frac{1}{2}} M^4 |u_1|^{-4\beta+1}, \, k \leq N-1.$
 
We next turn to ${}^{\Lb}\mathcal{S}^k$ and ${}^{\Lb}\mathcal{G}^k$. 
$^{\Lb}\mathcal{S}^k$ can be split as: ${}^{\Lb}\mathcal{S}^k = {}^{\Lb}S^k_{1} + \cdots + {}^{\Lb}S^k_{4}$,  where ${}^{\Lb}S^k_j, \, j=1:4$ are defined as \eqref{def-V-S-1}-\eqref{def-V-S-4} with $\D=\Do, \, V=\Lb$, $i=2$. The estimates for ${}^{\Lb} S^k_j, \, j=1:3$, resemble those for $S^k_j, \,  j=1:3$ \eqref{def-S-k-i-1}-\eqref{def-S-k-i-3}, with only $\varphi_{q}$ therein being replaced by $\Lb \varphi_{q}$. We are left with ${}^{\Lb}S^k_{4}$, which reads, 
\begin{equation*}
{}^{\Lb}S^k_{4} =  \doubleint_{\Do} \sum_{p+q\leq k \leq N-1, \, p<q}  |L\varphi_{p}|^2 |\nablaslash \Lb \varphi_{q}|^2 \eta^2  \di \mu_{\D}, \quad q \leq N-1, \, p \leq N-2.
\end{equation*}
By virtue of the upgraded $\|\eta^{\frac{1}{2}} L \varphi_{p} \|_{L^\infty}$, $p \leq N-2$ \eqref{L-4-L-phi-deg-improved} and $\mathcal{S}^{deg}_l  (\D_{0, \ub}^{u, +\infty}), \, l \leq N$ \eqref{decay-deg-estimate-fLb-fL-thm}, there is
\begin{equation*} 
 {}^{\Lb}S^k_{4} \lesssim   \doubleint_{\Do}  \sum_{i \leq N} \delta^{-1} \mathbb{I}_{N}^2 |u|^{-2\beta}  | \Lb \varphi_{i}|^2  \eta \di \mu_{\D} \lesssim \mathbb{I}_{N+1}^4 |u_1|^{-4\beta}.
\end{equation*}
For $ {}^{\Lb}\mathcal{G}^k$, we make the following splitting: ${}^{\Lb}\mathcal{G}^k= {}^{\Lb}G^k_{1} + \cdots + {}^{\Lb}G^k_{4}$, where ${}^{\Lb}G^k_j, \, j=1:4$ is defined as \eqref{def-V-G-i-1}-\eqref{def-V-G-i-4} with $\D=\Do, \, V=\Lb,\, i=2$.  The estimates for ${}^{\Lb}G^k_j, \, j=1:3$ are similar to that for ${}^LG^k_j, j = 1:3$ \eqref{L-Gk-1}-\eqref{L-Gk-3}.
As for ${}^{\Lb}G^k_{4}$, we take advantage of the enhanced $L^4$ estimate \eqref{L-4-L-phi-deg-improved} and \eqref{decay-deg-estimate-fLb-fL-thm} to deduce ($p+q \leq k \leq N-1, \, p<q$, hence $q \leq N - 1$ and $p \leq N-2$)
\begin{equation}\label{estimate-Lb-G4-deg}
\begin{split}
{}^{\Lb}G^k_{4}  & \lesssim  \int_{u_1}^{u_2} \int_{\ub_1}^{\ub_2} \|\eta^{\frac{1}{2}} L \varphi_{q} \|^2_{L^4(S_{\ub,u})}   \|\eta^{\frac{1}{2}} \nablaslash \Lb \varphi_{p} \|^2_{L^4(S_{\ub,u})} \di \ub \di u \\
& \lesssim \doubleint_{\Do}  \delta^{-1} \mathbb{I}_{N+1}^2  |u|^{-2\beta}  \sum_{p \leq i \leq p+2}  |  \Lb \varphi_{i}|^2 \eta \di \mu_{\D} \lesssim \mathbb{I}_{N+1}^4 |u_1|^{-4\beta}.
\end{split}
\end{equation}

For ${}^{\Lb}\mathcal{W}^k $, the improved $\mathcal{S}^{deg}_l  (\D_{0, \ub}^{u, +\infty}), \, l \leq N$ \eqref{decay-deg-estimate-fLb-fL-thm} and \eqref{estimate-Box-varphi-low-deg} yield ${}^{\Lb}\mathcal{W}^k \lesssim  \mathbb{I}_{N+1}^2  |u_1|^{-2\beta}$.

We end up with the following energy bound: for $k \leq N-1,$ $\beta \geq \frac{1}{2}$,
\begin{align*}
& {}^{\Lb}F_{k+1}^{deg}  (u_2; [\ub_1,\ub_2]) +  {}^{\Lb}\Fb_{k+1}^{deg} (\ub_2; [u_1,u_2]) + {}^{\Lb}\mathcal{S}_{k+1}(\Do)\\
\lesssim {}& {}^{\Lb}F_{k+1}^{deg} (u_1; [\ub_1,\ub_2]) +  {}^{\Lb}\Fb_{k+1}^{deg}(\ub_1; [u_1,u_2])  + \mathbb{I}_{N+1}^2 |u_1|^{-2\beta}.
\end{align*}
By analogy with the argument presented in Section \ref{sec-deg-energy-Cu}, we prove Theorem \ref{prop-Lb2-deg}.  
\end{proof}

\subsection{Non-degenerate energy near the future horizon}\label{sec-horizon}
In this section, we will prove the non-degenerate energy estimates near the horizon $\R_2^{NH}=\R_2 \cap \{ 2m \leq r \leq r_{NH}\}$, i,e., Theorem  \ref{Thm-decay-low-hori-energy} and Theorem \ref{Thm-decay-low-hori-energy-flux-Lb}.

Consider the region $r\leq 1.2r_{NH}$, and take $0 \leq \ub_{1} \leq \ub_{2} \leq \delta$. 
Let $u^e_i$ be the $u$ value of the intersecting sphere  $\{ r=1.2r_{NH}\} \cap \Cb_{\ub_i}$, and $u^{NH}_i$ be the $u$ value of the intersecting sphere $\{ r=r_{NH} \} \cap \Cb_{\ub_i}$. That is, $u^e_i = \ub_i - (1.2r_{NH})^\ast, \, u^{NH}_i = \ub_i - r_{NH}^\ast, \, i=1,2$.
 In the domain of $\{ r \leq 1.2 r_{NH}\} \cap \R_2$, i.e., $u^e_1 < u^{NH}_2 \leq u \leq + \infty, \, 0\leq \ub_{1} < \ub_{2} \leq \delta$, we define the following exterior  and interior region
\begin{equation*}
\begin{split}
\mathcal{D}^e := & \{r_{NH} < r\leq 1.2 r_{NH}\} \cap \{ \ub_{1} < \ub < \ub_{2} \}, \\
 \mathcal{D}^h := & \{r \leq   r_{NH}\} \cap \{ \ub_{1} < \ub < \ub_{2} \}, \\
 \Cb^{NH}_{\ub} :=& \Cb_{\ub} \cap \D^h, \quad C_u^{e} := C_{u} \cap \D^e.
\end{split}
\end{equation*}
We will also use the notation: $\Cb_{\ub}^{NH}=\Cb_{\ub} \cap \{ r\leq r_{NH}\} = \Cb_{\ub}^{[u^{NH}, +\infty]}$, where $u^{NH} := \ub - r_{NH}^\ast$, and $\Cb_{\ub}^{e}=\Cb_{\ub} \cap  \{r_{NH} < r\leq 1.2 r_{NH}\}$, if there is no room for confusion.

\subsubsection{The multiplier near the horizon}\label{sec-multiplier-horizon}

We choose $y_1(r^\ast)>0, \, y_2(r^\ast)>0$ that are supported in $r < 1.2 r_{NH}$, with $y_1 \big|_{\hori} =1$, $y_2 \big|_{\hori} =0$, and $\p_{r^\ast} y_1 >0$, $\p_{r^\ast} y_2 > 0$ if $2m < r \leq r_{NH}$. 
An example is given by \cite{D-R-13} (we notice that $|r^\ast| = -r^\ast$ near the horizon)
\begin{equation*}
y_1 = \xi_{r_{NH}} (r^\ast) (1+|r^\ast|^{-\epsilon}), \quad y_2 = c \xi_{r_{NH}} (r^\ast) |r^\ast|^{-1-\epsilon},
\end{equation*}
where $\epsilon$ is a small positive constant, $ \xi_{r_{NH}}$ is a cutoff function such that $ \xi_{r_{NH}} =1$ for $r\leq r_{NH}$ and $ \xi_{r_{NH}} = 0$ for $r \geq 1.2 r_{NH}$. One has then $y_2|_{\hori}=0, \p_{r^\ast} y_2|_{\hori}=0$, $y_1|_{\hori}=1, \, \p_{r^\ast} y_1|_{\hori}=0$. To carry out the estimates near horizon, we will consider the following vector field
\begin{equation}\label{def-VF-N}
N^h=\left( 1+y_2(r^\ast) \right) L + \delta^{-1}  y_1 (r^\ast)  Y.
\end{equation}

We take the multiplier $\xi =N^h$ \eqref{def-VF-N} and apply the energy identity to the wave equation for $\psi$. 
 In addition, we split up the error integrals into  exterior and interior parts to obtain 
\begin{equation}\label{energy-ineq-fL-fLb-hori-psi}
\begin{split}
& E^{ndeg} [\psi] (u; [\ub_{1}, \ub_{2}]) +  \Eb^{ndeg} [\psi]( \ub_{2}; [u^e_1, u]) \\
& \quad+ \doubleint_{\D^h}  \left(\delta^{-1} \eta^{-1} |\Lb \psi|^2 +\delta^{-1} \eta |\nablaslash \psi|^2 + \eta |L \psi|^2\right) \di \mu_{\D} \\
\lesssim {}& E^{ndeg} [\psi](u^e_1; [\ub_{1},\ub_{2}]) +  \Eb^{ndeg} [\psi]( \ub_{1}; [u^e_1, u]) \\
& +{}^h\mathcal{C} (\psi) + {}^h\mathcal{F} (\psi) + {}^e\mathcal{C} (\psi) + {}^e\mathcal{F} (\psi),
\end{split}
\end{equation}
where ${}^h\mathcal{C} (\psi), {}^e\mathcal{C} (\psi)$ are the exterior and interior currents respectively,
\begin{subequations}
\begin{align}
{}^h\mathcal{C} (\psi) &= \doubleint_{\D^h} \left( \eta |\nablaslash \psi|^2 +  \delta^{-1} |L \psi \Lb \psi| \right) \di \mu_{\D}, \label{def-current-hori-h}\\
{}^e\mathcal{C} (\psi)&= \doubleint_{\D^e} \left(\delta^{-1} |\Lb \psi|^2 +\delta^{-1}  |\nablaslash \psi|^2 +  |L \psi|^2 +  \delta^{-1} |L \psi \Lb \psi| \right) \di \mu_{\D}, \label{def-current-hori-e}
\end{align}
\end{subequations}
and $ {}^h\mathcal{F} (\psi) $, ${}^e\mathcal{F} (\psi)$ are the exterior and interior source terms,
\begin{subequations}
\begin{align}
 {}^h\mathcal{F} (\psi)&=  \doubleint_{\D^h} |\Box_g \psi| \left(  |L \psi|  + \delta^{-1} |Y \psi | \right) \eta \di \mu_{\D},\label{def-error-hori-h} \\
 {}^e\mathcal{F} (\psi)&= \doubleint_{\D^e} |\Box_g \psi| \left( | L \psi| + \delta^{-1} |  \Lb \psi | \right) \di \mu_{\D}. \label{def-error-hori-e}
\end{align}
\end{subequations}
 The interior current ${}^h\mathcal{C} (\psi)$ can be estimated in the same way as $\mathcal{C} (\psi)$ \eqref{error-quad-nabla-psi}-\eqref{error-quad-LbL-psi}: the first term in ${}^h\mathcal{C} (\psi)$ can be absorbed, while the second term is bounded by 
\begin{equation*}
\doubleint_{\D^h} c |L \psi|^2 \eta \di \mu_{\D} + \int_{\ub_{1}}^{\ub_{2}}  c^{-1} \delta^{-1} \di \ub  \int_{\Cb_{\ub}^{NH}} \delta^{-1} \eta^{-1} | \Lb \psi|^2  \di \mu_{\Cb_{\ub}}.
\end{equation*}
In a similar manner, there is,
\begin{align*}
 | {}^h\mathcal{F} (\psi) | \lesssim &  \doubleint_{\D^h} c |L \psi|^2 \eta \di \mu_{\D} +  \int_{\ub_{1}}^{\ub_{2}}  \delta^{-1} \di \ub  \int_{\Cb_{\ub}^{NH}} \delta^{-1} \eta^{-1}  | \Lb \psi|^2  \di \mu_{\Cb_{\ub}} \\
 &+   \doubleint_{\D^h} (c^{-1} +1)  |\Box_g \psi|^2 \eta \di \mu_{\D}.
\end{align*}
 As   \eqref{energy-ineq-fLb-fL-deg-psi-sign} in Section \ref{sec-deg-multiplier-II}, we choose $c \ll 1$, so that $\iint_{\D^h} c |L \psi|^2 \eta \di \mu_{\D}$ can be absorbed by the left hand side of \eqref{energy-ineq-fL-fLb-hori-psi}. After applying the Gr\"{o}nwall's inequality, there is
\begin{equation}\label{energy-ineq-fLb-fL-deg-psi-hori}
\begin{split}
& E^{ndeg} [\psi] (u_2; [\ub_1,\ub_2])  + \Eb^{ndeg} [\psi] (\ub_2; [u^e_1, u_2]) + \mathcal{S}^{ndeg} [\psi](\D^h)  \\
\lesssim {}& E^{ndeg} [\psi] (u^e_1; [\ub_1,\ub_2])  + \Eb^{ndeg} [\psi] ( \ub_1, [u^e_1, u_2] ) \\
& +  \doubleint_{\D^h}  |\Box_g \psi|^2 \eta \di \mu_{\D}+ {}^e\mathcal{C} (\psi) + {}^e\mathcal{F} (\psi),
\end{split}
\end{equation}
where ${}^e\mathcal{C} (\psi)$, ${}^e\mathcal{F} (\psi)$ are defined by \eqref{def-current-hori-e} and \eqref{def-error-hori-e}. In applications, ${}^e\mathcal{C} (\psi)$ and ${}^e\mathcal{F} (\psi)$ will be controlled by using the result of Theorem \ref{Thm-decay-low-deg-energy} and Theorem \ref{prop-Lb2-deg} (the degenerate case).

Specifically,   we have to rely on the degenerate spacetime integrated bound: $\mathcal{S}_l^{deg} (\D_{0, \ub}^{u, +\infty})  \lesssim \mathbb{I}_{N+1}^2$, $l \leq N$ of \eqref{decay-deg-estimate-fLb-fL-thm} to continue the non-degenerate energy estimates in the following sections \ref{sec-out-energy-horizon}--\ref{sec-high-de-hori}.

\subsubsection{Energy estimates for $E^{ndeg}_{k} (u; [0,\ub])$, $\mathcal{S}^{ndeg}_k( \R_2^{NH})$ and $\Eb^{ndeg}_{k} (\ub; [u,+\infty]),\, k \leq N-1$}\label{sec-out-energy-horizon}

We take $\psi = \varphi_k$, $k \leq N-1$ in \eqref{energy-ineq-fLb-fL-deg-psi-hori} to derive
\begin{equation}\label{energy-ineq-fL-fLb-hori}
\begin{split}
& E^{ndeg}_k(u; [\ub_{1}, \ub_{2}]) +  \Eb^{ndeg}_k( \ub_{2}; [u^e_1, u]) \\
& \quad+ \doubleint_{\D^h}  \left(\delta^{-1} \eta^{-1} |\Lb \varphi_k|^2 +\delta^{-1} \eta |\nablaslash \varphi_k|^2 + \eta |L \varphi_k|^2\right) \di \mu_{\D} \\
\lesssim {}& E^{ndeg}_k(u^e_1; [\ub_{1},\ub_{2}]) +  \Eb^{ndeg}_k( \ub_{1}; [u^e_1, u]) +  \doubleint_{\D^h}  |\Box_g \varphi_k|^2 \eta \di \mu_{\D}\\
& + {}^e\mathcal{C} (\varphi_k) + {}^e\mathcal{F} (\varphi_k),
\end{split}
\end{equation}
where ${}^e\mathcal{C} (\varphi_k)$ and ${}^e\mathcal{F} (\varphi_k)$ are defined by \eqref{def-current-hori-e}, \eqref{def-error-hori-e}, and $$\doubleint_{\D^h}  |\Box_g \varphi_k|^2 \eta \di \mu_{\D} = {}^hS^k_1 + \cdots + {}^hS^k_4.$$
Here ${}^hS^k_j, \, j=1:4$ are defined as \eqref{def-V-S-1}-\eqref{def-V-S-4} with $\D=\D^h, \, V=1$, $i=1$, i.e., for $p+q \leq k \leq N-1, \, p \leq q$,
\begin{subequations}
\begin{align}
{}^hS^k_1 &=  \doubleint_{\D^h}     |D \varphi_{p}|^2 |Y \varphi_{q}|^2 \eta  \di \mu_{\D}, \label{h-Sk-1} \\
{}^hS^k_2 &=  \doubleint_{\D^h}  | \Db \varphi_{p}|^2  |L \varphi_{q}|^2 \eta \di \mu_{\D},  \label{h-Sk-2}\\
{}^hS^k_3 &=  \doubleint_{\D^h}   | \Db \varphi_{p}|^2 |\nablaslash \varphi_{q}|^2 \eta \di \mu_{\D}, \label{h-Sk-3} \\
{}^hS^k_4 &=  \doubleint_{\D^h}   |L\varphi_{p}|^2 |\nablaslash \varphi_{q}|^2 \eta \di \mu_{\D}. \label{h-Sk-4}
\end{align}
\end{subequations}

The estimates for ${}^hS^k_j, \, j=1:4$ are analogous to the degenerate case. 
As \eqref{def-S-k-i-1}-\eqref{def-S-k-i-4}, we apply $L^\infty, L^\infty, L^2, L^2$ to the four factors in each of $ {}^hS^k_j, \, j=1:4$. Consequently,
\begin{align*}
|{}^hS^k_1| \lesssim {}&  \int_{\ub_{1}}^{\ub_{2}} M^2 \di \ub  \int_{\Cb^{NH}_{\ub}} \delta^{-1} \eta^{-1} |\Lb \varphi_{q}|^2 \di \mu_{\Cb_{\ub}}, \quad q \leq k, \\
|{}^hS^k_2| + |{}^hS^k_3| \lesssim {}& \delta^{\frac{1}{2}}M^2 \doubleint_{\D^h} \left( |L \varphi_q|^2 + |\nablaslash \varphi_q|^2 \right) \eta \di \mu_{\D}, \quad q \leq k,
\end{align*}
where the first line can be treated by the Gr\"{o}nwall's inequality, while the second one can be absorbed by the left hand side of \eqref{energy-ineq-fL-fLb-hori}. 
And for ${}^hS^k_4$, we note that $p \leq q \leq N-1$ and $ \|\nablaslash \varphi_{q} \|_{L^4(S_{\ub, u})} \lesssim \delta^{\frac{1}{4}}M$. 
Then for $\beta \geq \frac{1}{2}$,
\begin{equation}\label{esti-S4-hori}
\begin{split}
|{}^hS^k_4| \lesssim {}& \int_{u^e_{1}}^u \int_{\ub_{1}}^{\ub_{2}}  \|L\varphi_{p}\|^2_{ L^4(S_{\ub, u})} \| \nablaslash \varphi_{q}\|^2_{L^4(S_{\ub, u})} \eta \di u \di \ub \\
\lesssim{}& \delta^{\frac{1}{2}} M^2 \doubleint_{\D^h} \sum_{p \leq i \leq p+1} | L\varphi_{i}|^2  \eta \di \mu_{\D}, \quad p \leq N-1,\\
\lesssim {}& \delta^{\frac{1}{2}} \mathbb{I}_{N+1}^2 M^{2} \lesssim \delta^{\frac{1}{2}} M^{2},
\end{split}
\end{equation}
where the degenerate spacetime integrated estimate in \eqref{decay-deg-estimate-fLb-fL-thm} is used in the last inequality.

In the exterior region $\D^e$, $u^e_1  \leq u \leq u^{NH}_{2}$ and $1-\mu \sim 1$. 
Viewing the degenerate integrated decay estimate \eqref{decay-deg-estimate-fLb-fL-thm} and the improved one $\iint_{\D_{0, \ub}^{u, +\infty}} \delta^{-2}  |\Lb \varphi_l|^2 \di \mu_{\D} \lesssim \mathbb{I}_{N+1}^2 |u|^{-2\beta}$, $l \leq N$, see \eqref{decay-spacetime-estimate-Lb-LbL-Lb2} in Remark \ref{rk-thm-deg-low}, we derive
\begin{align*}
|{}^e\mathcal{C}(\varphi_k)|  \lesssim
& \doubleint_{\D^e} \left( |L \varphi_k|^2 + \delta^{-1} |\Db \varphi_k|^2 + \delta^{-2} |\Lb \varphi_k|^2 \right) \di \mu_{\D} \lesssim \mathbb{I}_{N}^2, \quad k \leq N-1.
\end{align*}
Besides, making use of Theorem \ref{Thm-decay-low-deg-energy} and following the proof leading to \eqref{energy-ineq-fLb-fL-lower-1}, we can also conclude
\begin{equation*}
|{}^e\mathcal{F}(\varphi_k) | \lesssim \mathbb{I}_{N}^2, \quad k \leq N-1.
\end{equation*}

In summary, we have accomplished: for any $\ub_1< \ub_2$ and $u > u^{NH}_{1}$, $k\leq N-1$,
\begin{equation}\label{energy-ineq-fL-fLb-hori-2}
\begin{split}
& E^{ndeg}_k(u; [\ub_{1}, \ub_{2}]) +  \Eb^{ndeg}_k( \ub_{2}; [u^e_{1}, u]) + \mathcal{S}^{ndeg}_k( \D^h)\\
\lesssim {}& E^{ndeg}_k(u^e_{1}; [\ub_{1}, \ub_{2}]) +  \Eb^{ndeg}_k( \ub_{1}; [u^e_{1}, u])  + \mathbb{I}_{N}^2.
\end{split}
\end{equation}
Noticing that, $\{u = u_1^e\} \cap \R_2$ is always away from the horizon, hence by Theorem \ref{Thm-decay-low-deg-energy},
\begin{equation}\label{energy-ext-t1-Cu}
E^{ndeg}_k(u^e_{1}; [\ub_{1}, \ub_{2}]) \sim E^{deg}_k(u^e_{1}; [\ub_{1}, \ub_{2}])  \lesssim \mathbb{I}^2_N, \quad k \leq N-1.
\end{equation}
Substituting \eqref{energy-ext-t1-Cu} into \eqref{energy-ineq-fL-fLb-hori-2}, and letting $\ub_1 =0$,  we obtain that for all $u \geq u^{NH}_{0} > u^{e}_{0}$ where $u_0^{NH}:=-r_{NH}^\ast$, $u^{e}_{0}:= - (1.2r_{NH})^\ast$ and $0 \leq \ub \leq \delta$,
\begin{equation}\label{estimate-E-ndeg-low}
 E^{ndeg}_k(u; [\ub_{1}, \ub_{2}]) +  \Eb^{ndeg}_k( \ub; [u^e_{0}, u]) \lesssim \mathbb{I}^2_N, \quad k \leq N-1.
\end{equation}
Letting $u \rightarrow +\infty$, $\ub_1=0$, $\ub_2 = \ub \leq \delta$ in \eqref{energy-ineq-fL-fLb-hori-2}, and taking \eqref{energy-ext-t1-Cu} into account, we have 
\begin{equation}\label{decay-double-integral-hori-low}
 \Eb^{ndeg}_k( \ub; [u^{NH}_{0}, +\infty])+ \mathcal{S}^{ndeg}_k ( \R_2^{NH})
 \lesssim   \mathbb{I}_N^2, \quad k \leq N-1.
\end{equation}
Analogous to \eqref{estimate-Box-varphi-low-deg}, there is the by-product as well
\begin{equation}\label{estimate-Box-phi-deg}
 \doubleint_{\D^h}  |\Box_g \varphi_k|^2 \eta \di \mu_{\D} \lesssim \delta^{\frac{1}{2}} M^{4}, \quad k \leq N-1.
\end{equation}

\subsubsection{Energy estimates for ${}^LF^{ndeg}_{k+1} (u; [0,\ub])$ and ${}^L\mathcal{S}^{ndeg}_{k+1} (\R_2^{NH})$, $k \leq N-1$}\label{sec-estimate-L-2}
We take $\psi = \delta L \varphi_k, \, k \leq N-1$ in \eqref{energy-ineq-fLb-fL-deg-psi-hori}, then
\begin{equation}\label{energy-ineq-fL-L-phi-hori}
\begin{split}
& {}^LF^{ndeg}_{k+1} (u; [\ub_{1},\ub_{2}]) +  {}^L\Fb^{ndeg}_{k+1} (\ub_{2}; [u^e_{1},u]) \\
& \quad+ \doubleint_{\D^h}  \left(\delta \eta^{-1}  |\Lb L \varphi_k|^2  +\delta \eta |\nablaslash L \varphi_k|^2 + \delta^2 \eta |L^2 \varphi_k|^2\right) \di \mu_{\D} \\
\lesssim & {}^LF^{ndeg}_{k+1} (u^e_{1}; [\ub_{1},\ub_{2}]) +  {}^L\Fb^{ndeg}_{k+1} (\ub_{1}; [u^e_{1},u]) +  \doubleint_{\D^h} \delta^2 |L \Box_g \varphi_k|^2 \eta \di \mu_{\D}  \\
& + {}^{hL}\mathcal{W}^{k} + {}^e\mathcal{C} (\delta L\varphi_k) + {}^e\mathcal{F} ( \delta L \varphi_k),
\end{split}
\end{equation}
where ${}^e\mathcal{C} (\delta L \varphi_k)$ and ${}^e\mathcal{F} ( \delta L \varphi_k)$ are defined by \eqref{def-current-hori-e}, \eqref{def-error-hori-e}, ${}^{hL}\mathcal{W}^{k}$ is related to $\delta [\Box_g, L] \varphi_k$,
\begin{equation*}
{}^{hL}\mathcal{W}^{k}  =  \doubleint_{\D^h} \delta^2  \left( |L\varphi_k|^2 +|\Lb\varphi_k|^2 +|\laplacianslash \varphi_k |^2 +  |\Box_g \varphi_k|^2 \right)  \eta  \di \mu_{\D},
\end{equation*}
and  $$\doubleint_{\D^h} \delta^2 |L \Box_g  \varphi_k|^2 \eta \di \mu_{\D} = {}^{hL}\mathcal{S}^k + {}^{hL}\mathcal{G}^k + {}^{hL}\mathcal{L}^k.$$ Here ${}^{hL}\mathcal{S}^k, {}^{hL}\mathcal{G}^k$ are defined as  \eqref{def-V-mS-k}-\eqref{def-V-mG-k} with $\D=\D^h, \,V=\delta L$, $i=1$, ${}^{hL}\mathcal{L}^k$ is defined as  \eqref{def-V-mS-k} with $\D=\D^h, \, V=\delta$, $i=1$.
We will estimate these error terms one by one. 

To begin with, there is ${}^{hL}\mathcal{L}^k \lesssim \delta^{\frac{5}{2}} M^{4}, \, k \leq N-1$, by \eqref{estimate-Box-phi-deg}.

For ${}^{hL}\mathcal{S}^k$, it is split into: ${}^{hL}\mathcal{S}^k = {}^{hL}S^k_1 + \cdots + {}^{hL}S^k_4$, where ${}^{hL}S^k_j, \, j=1:4$ are defined as \eqref{def-V-S-1}-\eqref{def-V-S-4} with $\D=\D^h, \, V=\delta L, \, i=1$.
The estimates for ${}^{hL}S^k_j, \, j=1:3$ resemble those for ${}^hS^k_j, \, j=1:3$ \eqref{h-Sk-1}-\eqref{h-Sk-3}, and hence we omit the details here. The remaining ${}^{hL}S^k_4$ reads
\begin{equation*}
{}^{hL}S^k_4 =  \doubleint_{\D^h}  \sum_{p + q \leq k, p \leq q}  \delta^2  |L\varphi_{p}|^2 |\nablaslash L\varphi_{q}|^2 \eta \di \mu_{\D}, \quad k \leq N-1.
\end{equation*}
Note that,  $ |L\varphi_{p}|  \lesssim \delta^{-\frac{1}{2}} M, \, p \leq N/2 \leq N-3$, and $q \leq N-1$, thus,
\begin{equation}\label{esti-H4-L-hori}
|{}^{hL}S^k_4|  \lesssim  \delta M^2 \doubleint_{\D^h} |L\varphi_{q+1}|^2  \eta \di \mu_{\D} \lesssim  \delta  \mathbb{I}_{N+1}^2 M^2,
\end{equation}
where the degenerate spacetime estimate \eqref{decay-deg-estimate-fLb-fL-thm} is used in the second inequality.

For ${}^{hL}\mathcal{G}^k,$ we make the following splitting: ${}^{hL}G^k= {}^{hL}G^k_1 + \cdots + {}^{hL}G^k_4$, where $ {}^{hL}G^k_j,\,j=1:4$ are defined as \eqref{def-V-G-i-1}-\eqref{def-V-G-i-4} with $\D=\D^h, \, V=\delta L$, $i=1$, i.e., for $p+q \leq k \leq N-1, \, p<q$,
\begin{subequations}
\begin{align}
{}^{hL}G^k_1 &=  \doubleint_{\D^h} \delta^2  |D\varphi_{q}|^2 |Y L\varphi_{p}|^2 \eta  \di \mu_{\D}, \label{hL-Gk-1} \\
{}^{hL}G^k_2 &=  \doubleint_{\D^h}   \delta^2  | \Db \varphi_{q}|^2 | L^2\varphi_{p}|^2  \eta \di \mu_{\D}, \label{hL-Gk-2} \\
{}^{hL}G^k_3 &=  \doubleint_{\D^h}   \delta^2 | \Db\varphi_{q}|^2  |\nablaslash L\varphi_{p}|^2 \eta \di \mu_{\D}, \label{hL-Gk-3} \\
{}^{hL}G^k_4 &=  \doubleint_{\D^h}  \delta^2 |L\varphi_{q}|^2 |\nablaslash L\varphi_{p} |^2  \eta \di \mu_{\D}. \label{hL-Gk-4}
\end{align}
\end{subequations}
Note that $  k \leq N-1$, then $q \leq N-1, \, p \leq k-1 \leq N-2$.  They can be estimated in the same manner as ${}^LG^k_j, \, j=1: 4$ \eqref{L-Gk-1}-\eqref{L-Gk-4}. Hence, we only sketch the calculations here.
\begin{align*}
{}^{hL}G^k_1 & \lesssim  \int_{\ub_{1}}^{\ub_{2}}  \int_{u^{NH}_{1}}^{u} \delta^{-1} M^2 \delta^2  \cdot \eta^{-1} \| \Lb L \varphi_{p} \|^2_{L^{4}(S_{\ub,u^\prime})}  \di u^\prime \di \ub, \\
&\lesssim  \int_{\ub_{1}}^{\ub_{2}}  M^2 \di \ub  \int_{\Cb_{\ub}^{NH}}  \sum_{p \leq i \leq p+1} \delta \eta^{-1} | \Lb L \varphi_{i}|^2 \di \mu_{\Cb_{\ub}}, \quad p \leq k-1,
\end{align*}
which can handled by the Gr\"{o}nwall's inequality.
For ${}^{hL}G^k_2$, ${}^{hL}G^k_3$,
\begin{align*}
{}^{hL}G^k_2 + {}^{hL}G^k_3 & \lesssim  \int_{\ub_{1}}^{\ub_{2}}  \int_{u^{NH}_{1}}^{u} \delta^{\frac{1}{2}} M^2 \delta^2 \eta \left( \| L^2\varphi_{p} \|^2_{L^{4}(S_{\ub,u^\prime})}  + \|\nablaslash L \varphi_{p} \|^2_{L^{4}(S_{\ub,u^\prime})} \right) \di u^\prime \di \ub, \\
&\lesssim \delta^{\frac{1}{2}} M^2\doubleint_{\D^h} \sum_{p \leq i \leq p+1} \delta^2 \left( | L^2\varphi_{i}|^2 + |\nablaslash L \varphi_{i} |^2  \right)  \eta \di \mu_{\D}, \quad p \leq k-1,
\end{align*}
which can be absorbed by the left hand side of \eqref{energy-ineq-fL-L-phi-hori}. 
Similarly for ${}^{hL}G^k_4$, we have, by  \eqref{decay-deg-estimate-fLb-fL-thm},
\begin{equation}\label{estimate-hL-G4}
\begin{split}
{}^{hL}G^k_4 & \lesssim  \int_{\ub_{1}}^{\ub_{2}}  \int_{u^{NH}_{1}}^{u}  \delta^{-1} M^2 \delta^2 \eta \|\nablaslash L \varphi_{p}\|^2_{L^4(S_{\ub, u^\prime})} \di u^\prime \di \ub, \\
&\lesssim \delta M^2 \doubleint_{\D^h}   \sum_{p\leq i\leq p+2} \eta | L \varphi_{i}|^2 \di \mu_{\D} \lesssim  \delta \mathbb{I}_{N+1}^2M^2, \quad p \leq N-2.
\end{split}
\end{equation}

Again, by virtue of \eqref{decay-deg-estimate-fLb-fL-thm} and \eqref{estimate-Box-phi-deg}, ${}^{hL}\mathcal{W}^{k}$ is bounded by, 
\begin{align*}
|{}^{hL}\mathcal{W}^{k}| &\lesssim
 \delta^2 \mathbb{I}_{N+1}^2 +  \delta^{\frac{5}{2}} M^{4}, \quad k \leq N-1. 
\end{align*}

Furthermore, as consequence of Theorem \ref{Thm-decay-low-deg-energy}, 
$| {}^e\mathcal{C} (\delta L\varphi_k)| + |{}^e\mathcal{F} ( \delta L \varphi_k)| \lesssim \mathbb{I}_{N+1}^2, \, k \leq N-1$.

Eventually, we arrive at, for $k\leq N-1$,
\begin{equation}\label{energy-ineq-fL-L-phi-hori-1}
\begin{split}
& {}^LF^{ndeg}_{k+1} (u; [\ub_{1},\ub_{2}])  + {}^L\Fb^{ndeg}_{k+1} (\ub_{2}; [u^e_{1}, u]) + {}^L\mathcal{S}^{ndeg}_{k+1} (\D^h) \\
\lesssim & {}^LF^{ndeg}_{k+1} (u^e_{1}; [\ub_{1},\ub_{2}])  + {}^L\Fb^{ndeg}_{k+1}  (\ub_{1}; [u^e_{1}, u]) + \mathbb{I}_{N+1}^2,
\end{split}
\end{equation} 
which yields \eqref{decay-hori-estimate-Fb-thm} and \eqref{decay-hori-estimate-L-lower-flux-thm} via an analogous argument presented in Section \ref{sec-out-energy-horizon}. 

We have accomplished the improvements \eqref{decay-hori-estimate-L-lower-flux-thm}, \eqref{decay-hori-estimate-Fb-thm} and \eqref{estimate-E-ndeg-low}-\eqref{decay-double-integral-hori-low} till now.
Combining the degenerate estimates of Theorem \ref{Thm-decay-low-deg-energy} with the results of \eqref{decay-hori-estimate-L-lower-flux-thm} and \eqref{estimate-E-ndeg-low}, we can upgrade the non-degenerate estimate for $L \varphi_{k}$:
\begin{equation}\label{improved-L-nabla-infty-L-phi}
\|L\varphi_p\|_{L^\infty(\R_2)} + \|L\varphi_q\|_{L^4(S_{\ub, u} \cap \R_2)}   \lesssim \delta^{-\frac{1}{2}} \mathbb{I}_{N},  \quad p\leq N-3, \, q \leq N-2. 
\end{equation}

\subsubsection{Energy estimates for $E^{ndeg}_N(u; [0,\ub])$, $\mathcal{S}^{ndeg}_{N} (\R_2^{NH})$ and $\Eb^{ndeg}_N(\ub, [u^{NH},+\infty])$}\label{sec-top-energy-ndeg}
As explained in the degenerate case, the estimate for ${}^hS_4^k$ \eqref{esti-S4-hori} is illegal if $k=N$. However, the enhanced estimate \eqref{improved-L-nabla-infty-L-phi}  will help to linearize ${}^hS^N_4$. We remind ourselves that,
\begin{equation*}
{}^hS^N_4 =  \doubleint_{\D^h}  \sum_{p+q \leq N,  p \leq q} |L\varphi_{p}|^2 |\nablaslash \varphi_{q} |^2 \eta \di \mu_{\D}.
\end{equation*}
With the aid of \eqref{improved-L-nabla-infty-L-phi} (knowing that $p \leq [\frac{N}{2}] \leq N-3$), and the spacetime estimate for $\mathcal{S}^{deg}_l (\D_{0, \ub}^{u, +\infty})$, $l \leq N$ in \eqref{decay-deg-estimate-fLb-fL-thm}, 
 \begin{equation*}
|{}^hS^N_4| \lesssim  \doubleint_{\D^h} \delta^{-1}\mathbb{I}_{N}^2 |\nablaslash \varphi_N|^2 \eta \di \mu_{\D} \lesssim \mathbb{I}_{N}^2 \mathbb{I}^2_{N+1}.
\end{equation*}
The other terms can be bounded in the same way as that in the lower order cases.
After that, we derived \eqref{decay-hori-estimate-lower-flux-thm}, \eqref{decay-hori-estimate-Eb-thm}. 

Hence, we have carried out the proof for Theorem \ref{Thm-decay-low-hori-energy}. As a consequence, there is, 
\begin{equation}\label{improved-L-nabla-L4-phi-top}
\|L\varphi_j \|_{L^\infty(\R_2)} + \|L\varphi_k\|_{L^4(S_{\ub, u} \cap \R_2)}    \lesssim \delta^{-\frac{1}{2}} \mathbb{I}_{N+1}, \quad k\leq N-1,\, j \leq N-2. 
\end{equation}

\subsubsection{Energy estimates for ${}^YF^{ndeg}_{k+1} (u; [0,\ub])$, $ {}^Y\mathcal{S}^{ndeg}_{k+1}(\R_2^{NH}) $ and $ {}^Y\Fb^{ndeg}_{k+1} (\ub; [u^{NH}, +\infty])$, $k \leq N-1$}\label{sec-high-de-hori}
Thanks to Theorem \ref{Thm-decay-low-hori-energy} and the resulted improvement \eqref{improved-L-nabla-L4-phi-top}, we will prove in this section the energy bound related to $Y$ near the horizon, i.e., Theorem \ref{Thm-decay-low-hori-energy-flux-Lb}.

\begin{proof}[Proof of Theorem \ref{Thm-decay-low-hori-energy-flux-Lb}]
We take $\psi =Y \varphi_k, \, k \leq N-1$ in \eqref{energy-ineq-fLb-fL-deg-psi-hori}, to derive
\begin{align*}
{}& {}^YF^{ndeg}_{k+1} (u; [\ub_{1},\ub_{2}])  + {}^Y\Fb^{ndeg}_{k+1} (\ub_{2}; [u^e_{1}, u]) \\
& + \doubleint_{\D^h} \left( \delta^{-1}\eta^{-1} |\Lb Y \varphi_k|^2+  \eta |L Y \varphi_k|^2 + \delta^{-1} \eta |\nablaslash Y \varphi_k|^2 \right)  \di \mu_{\D} \\
\lesssim {}& {}^{Y}F_{k+1}^{ndeg} (u^e_{1}; [\ub_{1},\ub_{2}]) +{}^{Y}\Fb_{k+1}^{ndeg} (\ub_{1}; [u^e_{1}, u]) +  \doubleint_{\D^h}  |Y \Box_g \varphi_k|^2 \eta \di \mu_{\D} \\
& +  {}^{hY}\mathcal{W}^k  + {}^e\mathcal{C} (Y\varphi_k) + {}^e\mathcal{F} ( Y \varphi_k), 
\end{align*} 
where  ${}^e\mathcal{C} (Y \varphi_k)$, ${}^e\mathcal{F} (Y \varphi_k)$ are defined by \eqref{def-current-hori-e}, \eqref{def-error-hori-e}, ${}^{hY}\mathcal{W}^k $ is related to $[\Box_g, Y] \varphi_k$ and given by
\begin{align*}
{}^{hY}\mathcal{W}^k &=  \doubleint_{\D^h}  \left( |Y^2 \varphi_k|^2 + |\laplacianslash \varphi_k|^2 + |Y \varphi_k|^2 + |L \varphi_k|^2 \right) \eta \di \mu_{\D},
\end{align*}
and $$\doubleint_{\D^h}  |Y\Box_g \varphi_k|^2 \eta \di \mu_{\D} =  {}^{hY}\mathcal{S}^k  + {}^{hY}\mathcal{G}^k + {}^{hY}\mathcal{L}^k.$$ Here $ {}^{hY}\mathcal{S}^k, {}^{hY}\mathcal{G}^k$ are defined as \eqref{def-V-mS-k}-\eqref{def-V-mG-k} with $\D = \D^h, \,V=Y$, $i=1$, while $ {}^{hY}\mathcal{L}^k$ is given by \eqref{def-V-mS-k} with $\D = \D^h, \,V=1$, $i=1$,  

We split ${}^{hY}\mathcal{S}^k$ into: ${}^{hY}\mathcal{S}^k = {}^{hY}S^k_1 + \cdots + {}^{hY}S^k_4$, where ${}^{hY}S^k_j, \, j=1:4$ are defined as \eqref{def-V-S-1}-\eqref{def-V-S-4} with $\D = \D^h, \, V=Y$, $i=1$.  The estimates for ${}^{hY}S^k_{j}, \, j =1:3$ mimic those for ${}^hS^k_j, \, j=1:3$ \eqref{h-Sk-1}-\eqref{h-Sk-3}. Hence, we will only focus on ${}^{hY}S^k_{4}$, which reads,
\begin{align*}
{}^{hY}S^k_{4} &=  \doubleint_{\D^h} \sum_{p+q\leq k, p \leq q}  |L\varphi_{p}|^2 |\nablaslash Y \varphi_{q}|^2 \eta \di \mu_{\D}, \quad k\leq N-1.
\end{align*}
We make use of the improved $L^\infty$ estimate for $L\varphi_{p}, \, p \leq \frac{N}{2} \leq N-3$ \eqref{improved-L-nabla-L4-phi-top}, then
\begin{align*}
|{}^{hY}S^k_{4}| \lesssim & \int_{\ub_{1}}^{\ub_{2}} \mathbb{I}_N^2 \int_{\Cb^{NH}_{\ub}} \delta^{-1}  |\nablaslash Y \varphi_q|^2 \eta \di \mu_{\Cb_{\ub}}, \quad q \leq k,
\end{align*}
can be handled by the Gr\"{o}nwall's inequality.

For ${}^{Yh}\mathcal{G}^k$, there is, ${}^{hY}G^k= {}^{hY}G^k_{1} + \cdots + {}^{hY}G^k_{4}$, where $ {}^{hY}G^k_j, \, j=1:4$ are defined as \eqref{def-V-G-i-1}-\eqref{def-V-G-i-4} with $\D = \D^h, \, V=Y, \, i =1$.  ${}^{hY}G^k_j, \, j=1:3$ can be estimated in a similar manner as ${}^{hL}G^k_j, \, j=1:3$ \eqref{hL-Gk-1}-\eqref{hL-Gk-3}, while the left ${}^{hY}G^k_{4}$ takes the following form
\begin{align*}
{}^{hY}G^k_{4} &=  \doubleint_{\D^h} \sum_{p+q \leq k, p<q} |L\varphi_{q}|^2 |\nablaslash Y \varphi_{p} |^2 \eta  \di \mu_{\D}, \quad k \leq N-1.
\end{align*}
Noticing that $q \leq N-1, \, p \leq k -1 \leq N-2$ and referring to \eqref{estimate-Lb-G4-deg}, we obtain by means of the upgraded $L^4$ estimate in \eqref{improved-L-nabla-L4-phi-top} and the integrated estimate $\mathcal{S}^{ndeg}_l(\D^h)$, $l \leq N$ \eqref{decay-hori-estimate-lower-flux-thm}  in Theorem \ref{Thm-decay-low-hori-energy},
\begin{align*}
{}^{hY}G^k_{4} 
&\lesssim   \doubleint_{\D^h}   \sum_{p \leq i \leq p+2} \delta^{-1}  \mathbb{I}_{N+1}^2 |Y \varphi_{i}|^2 \eta  \di \mu_{\D} \lesssim  \mathbb{I}_{N+1}^4, \quad p \leq N-2.
\end{align*}

In addition, the usage of the spacetime estimate $\mathcal{S}^{ndeg}_l(\D^h)$, $l \leq N$ \eqref{decay-hori-estimate-lower-flux-thm} also leads to 
\begin{align*}
| {}^{hY}\mathcal{W}^k | &\lesssim \mathbb{I}_{N+1}^2 + \int_{\ub_{1}}^{\ub_{2}} \delta  \int_{\Cb_{\ub}}  \delta^{-1} \eta^{-1} |\Lb Y \varphi_k|^2 \di \mu_{\D}, \quad k \leq N-1,
\end{align*}
where the last term can be treated by the Gr\"{o}nwall's inequality.

Finally, there is also $| {}^e\mathcal{C} (Y\varphi_k)| + |{}^e\mathcal{F} ( Y \varphi_k)|  
\lesssim \mathbb{I}_{N+1}^2$, as a result of Theorem \ref{Thm-decay-low-deg-energy} and Theorem \ref{prop-Lb2-deg}.

In the end, we arrive at: for $k \leq N-1$,
\begin{equation}\label{energy-ineq-fL-Lb-phi-hori-whole}
\begin{split}
{}& {}^YF^{ndeg}_{k+1} (u; [\ub_{1},\ub_{2}])  + {}^Y\Fb^{ndeg}_{k+1} (\ub_{2}; [u^e_{1}, u]) + {}^{Y}\mathcal{S}^{ndeg}_{k+1} (\D^h)\\
\lesssim {}&{}^{Y}F_{k+1}^{ndeg}(u^e_{1}; [\ub_{1},\ub_{2}]) + {}^Y\Fb^{ndeg}_{k+1} (\ub_{1}; [u^e_{1}, u])  +\mathbb{I}_{N+1}^2,
\end{split}
\end{equation} 
which gives rise to Theorem \ref{Thm-decay-low-hori-energy-flux-Lb}.
\end{proof}

\subsection{More general energies in $\R_2$}\label{sec-recover-Lb-Y}

In what follows, we will capitalize on Theorem \ref{Thm-decay-low-deg-energy}, Theorem \ref{Thm-decay-low-hori-energy} and the improved $L^\infty$ estimate \eqref{improved-L-nabla-L4-phi-top} to retrieve ${}^tF^{deg}_{k+1} (u; [0,\ub])$, ${}^tF^{ndeg}_{k+1} (\ub; [u^{NH}, +\infty])$, $k \leq N-1$. The proof will be an analogue of the one in Section \ref{sec-T-E}.

\subsubsection{Estimates for ${}^tF^{deg}_{k+1} (u; [0,\ub]), \, k \leq N-1$}\label{sec-recover-Lb-deg}
\begin{proposition}\label{Prop-ext-decay-Lb-phi}
In $\R_2$, given any  real number $\beta \geq \frac{1}{2}$ and $k \leq N-1$, there are
\begin{align}
 \delta^{-1}  \| \eta^{\frac{1}{2}} Y\varphi_k\|^2_{L^2(S_{\ub,u})} + \delta^{-2} \|\eta^{\frac{1}{2}} Y \varphi_k\|^2_{L^2(C_u)} \lesssim {}&  \mathbb{I}^2_{N+1}   |u|^{-2\beta}, \label  {eq-decay-Lb-phi-L2-Cu-deg-prop}\\
\| \eta^{\frac{1}{2}} Y L \varphi_k\|^2_{L^2(C_u)} \lesssim {}& \mathbb{I}^2_{N+1}  |u|^{-2\beta}. \label{improve-L-Lb-Cu-deg}
\end{align}
\end{proposition}

\begin{proof}
Define $\chib^2 [\psi] (u,\ub) =   \int_{S_{\ub, u}} |Y \psi|^2  \eta  r^2 \di \sigma_{S^2}$. Take $\psi = \varphi_k$, $k \leq N-1$,
\begin{align*}
& \p_{\ub} \chib^2 [\varphi_k] (u,\ub)  +   \int_{S_{\ub, u}}   \eta^{-1} \mu r |\Lb \varphi_k|^2  \di \sigma_{S^2} \\
=&  \int_{S_{\ub, u}} 2 r^2 \eta^{-1} \Lb \varphi_k \left( L \Lb \varphi_k + \frac{\eta}{r} \Lb \varphi_k \right)  \di \sigma_{S^2}.
\end{align*}
Appealing to the wave equation and the Cauchy-Schwarz inequality, we integrate along $\p_{\ub}$ to derive (refer to \eqref{esti-ineq-Lpsi-Cu})
\begin{equation}\label{eq-retrieve-Lb-Cu}
\begin{split}
 \chib^2 [\varphi_k] (u,\ub) & \lesssim  \int_{0}^{\ub} \delta^{-1} \chib^2 [\varphi_k] (u,\ub^\prime) \di \ub^\prime \\
 & + \int_{C_{u}}  \delta   \eta \left( |\laplacianslash \varphi_k|^2+ |L \varphi_k|^2 + |\Box_g \varphi_k|^2 \right)   \di \mu_{C_u}.
\end{split}
\end{equation}
We make the following splitting: $ \int_{C_{u}} \delta \eta |\Box_g \varphi_k|^2 \lesssim \sum_{i=1}^{4} F^i_k$, with $F_k^i$ defined as below: for all $p+q\leq k \leq N-1, \, p \leq q$,
\begin{align*}
  F^1_k &: =  \int_{C_{u}}  \delta  |D \varphi_{p}|^2  |Y \varphi_{q}|^2 \eta  \di \mu_{C_u}, 
  & F^2_k : = \int_{C_{u}}  \delta  | \Db \varphi_{p}|^2  |L\varphi_{q}|^2 \eta \di \mu_{C_u}, \\
  F^3_k &: =\int_{C_{u}} \delta  | \Db\varphi_{p}|^2  |\nablaslash\varphi_{q}|^2 \eta \di \mu_{C_u}, 
  & F^4_k : = \int_{C_{u}} \delta  |L\varphi_{p}|^2 |\nablaslash\varphi_{q}|^2 \eta \di \mu_{C_u}.
\end{align*}
In view of the improved estimate for $\|L \varphi_{p} \|_{L^\infty(\R_2)}$ \eqref{improved-L-nabla-L4-phi-top} and  $\| \Db \varphi_{p} \|_{L^\infty(\R_2)} \lesssim \delta^{\frac{1}{4}}M$,  $p \leq \frac{N}{2} \leq N-3$, and Theorem \ref{Thm-decay-low-deg-energy}, $F^i_k$, $i=1:4$, share the following estimates ($q \leq N-1$)
\begin{align*}
 |F^1_k|   &\lesssim \int_0^{\ub}   \mathbb{I}^2_{N} \chib^2 [\varphi_k] (u,\ub^\prime) \di \ub^\prime, \\
 |F^2_k| + |F^3_k| & \lesssim \delta^{ \frac{3}{2}} M^2 \left( \|\eta^{\frac{1}{2}} L \varphi_{q}\|^2_{L^2(C_{u})}+ \|\eta^{\frac{1}{2}} \nablaslash \varphi_{q}\|^2_{L^2(C_{u})} \right) \lesssim \delta^{ \frac{3}{2}} M^2 \mathbb{I}_N^2 |u|^{-2\beta}, \\
 |F^4_k| &  \lesssim \mathbb{I}^2_{N} \|\eta^{\frac{1}{2}} \nablaslash \varphi_{q}\|^2_{L^2(C_{u})} \lesssim \delta \mathbb{I}_N^4 |u|^{-2\beta}.
\end{align*}
Therefore,
\begin{equation}\label{retrive-Lb-Box-deg}
 \delta \| \eta^{\frac{1}{2}}\Box_g \varphi_k\|^2_{L^2(C_u)} \lesssim \int_0^{\ub}   \mathbb{I}^2_{N} \chib^2 [\varphi_k] (u,\ub^\prime) \di \ub^\prime +  \left( \delta \mathbb{I}^2_{N} + \delta^{ \frac{3}{2}} M^2 \right) \mathbb{I}_N^2 |u|^{-2\beta}.
\end{equation}
By the Gr\"{o}nwall's inequality,  \eqref{eq-retrieve-Lb-Cu} turns into
\begin{align*}
 \chib^2 [\varphi_k] (u,\ub)  \lesssim  \delta \mathbb{I}^4_{N}  |u|^{-2\beta} + \delta \left(  \|\eta^{\frac{1}{2}} \nablaslash \varphi_{k+1}\|^2_{L^2(C_u)}  +  \|\eta^{\frac{1}{2}} L \varphi_{k}\|^2_{L^2(C_u)} \right),
\end{align*}
which tells that
\begin{equation}\label{retrive-Lb-S-deg}
 \chib^2 [\varphi_k] (u,\ub)  \lesssim \delta \mathbb{I}^2_{N} \mathbb{I}^2_{N+1} |u|^{-2\beta}, \quad k \leq N-1.
\end{equation}
 Integrating \eqref{retrive-Lb-S-deg} over the interval $\ub \in [0, \delta]$, we prove \eqref{eq-decay-Lb-phi-L2-Cu-deg-prop}.

Meanwhile, based on the wave equation, there is, for $k \leq N-1$,
\begin{align*}
\|\eta^{-\frac{1}{2}} L \Lb\varphi_k\|^2_{L^2(C_u)}  \lesssim {}&  \| \eta^{\frac{1}{2}} \Lb \varphi_{k}\|^2_{L^2(C_u)} + \|\eta^{\frac{1}{2}} L \varphi_{k}\|^2_{L^2(C_u)} \\
&+ \| \eta^{\frac{1}{2}} \laplacianslash \varphi_{k}\|^2_{L^2(C_u)} + \|\eta^{\frac{1}{2}} \Box_g \varphi_{k}\|^2_{L^2(C_u)}.
\end{align*}
Due to Theorem \ref{Thm-decay-low-deg-energy}, \eqref{retrive-Lb-Box-deg} and the proved \eqref{eq-decay-Lb-phi-L2-Cu-deg-prop}, the estimate \eqref{improve-L-Lb-Cu-deg} follows.
\end{proof}

\subsubsection{Estimates for ${}^tF^{ndeg}_{k+1} (\ub; [u^{NH}, +\infty]), \, k \leq N-1$}\label{sec-recover-Lb-ndeg}
\begin{proposition}\label{Prop-nondeg-decay-Lb-phi}
In the region $\R_2 \cap \{r\leq r_{NH}\}$, there are
\begin{align}
\delta^{-1} \|Y \varphi_k\|^2_{L^2(S_{\ub,u})} + \delta^{-2} \|Y \varphi_k\|^2_{L^2(C_u)} \lesssim{}& \mathbb{I}^2_{N+1}, \quad k \leq N-1, \label{decay-Lb-phi-L2-Cu} \\
\|YL \varphi_k\|^2_{L^2(C_u)} \lesssim{}&  \mathbb{I}^2_{N+1}, \quad k \leq N-1.\label{improve-L-Lb-Cu}
\end{align}
\end{proposition}

\begin{proof}
Defining ${}^h\chib^2 [\varphi_k] (u,\ub)= \int_{S_{\ub,u}} |\eta^{-1} \Lb \varphi_k|^2 r^2  \di \sigma_{S^2}, \, k \leq N-1$, we derive,
\begin{equation*}
\begin{split}
\p_{\ub} {}^h\chib^2 [\varphi_k] (u,\ub) &= \int_{S_{\ub,u}} 2 \eta^{-2} \Lb \varphi_k \left( L \Lb \varphi_k + \frac{\eta}{r} \Lb \varphi_k \right) r^2  \di \sigma_{S^2}  \\
&\quad \quad \quad -  \int_{S_{\ub,u}} 2 \eta^{-2} \frac{2m}{r^2} |\Lb \varphi_k|^2 r^2  \di \sigma_{S^2}.
\end{split}
\end{equation*}
Then it follows from the proof leading to Proposition \ref{Prop-ext-decay-Lb-phi} that, for $k \leq N-1$,
\begin{align*}
  {}^h\chib^2 [\varphi_k] (u,\ub)  \lesssim{}& \int_0^{\ub}( \delta^{-1} + \mathbb{I}^2_{N})  {}^h\chib^2 [\varphi_k] (u,\ub^\prime)    \di \ub^\prime +  \delta (\mathbb{I}^4_{N} +\mathbb{I}^2_{N+1}),
\end{align*}
where Theorem \ref{Thm-decay-low-hori-energy} is used.
After applying the Gr\"{o}nwall's inequality, there is
\begin{align*}
{}^h\chib^2 [\varphi_k] (u,\ub) \lesssim \delta \mathbb{I}^2_{N}  \mathbb{I}^2_{N+1}, \quad k \leq N-1.
\end{align*}
Integrating the above formula along $\p_{\ub}$, we have \eqref{decay-Lb-phi-L2-Cu}. Besides, the estimates above also imply
\begin{equation}\label{retrive-Lb-Box-ndeg}
\| \Box_g \varphi_k\|^2_{L^2(C_u)} \lesssim \mathbb{I}^4_{N}  \mathbb{I}^2_{N+1} + \delta^{ \frac{1}{2}} M^2 \mathbb{I}^2_{N}, \quad k \leq N-1.
\end{equation}
Thus, \eqref{improve-L-Lb-Cu} follows from the wave equation, \eqref{retrive-Lb-Box-ndeg} and the proved \eqref{decay-Lb-phi-L2-Cu}.
\end{proof}

\subsubsection{Energy estimates for generally high order derivatives in $\R_2$}\label{sec-high-derivative}
Define 
\begin{align*}
E^{deg}_{l+k}(u; [\ub_1,\ub])  &:=  \sum_{p+q=l} E^{deg}[\delta^p W^{l}_{p,q} \varphi_k](u; [\ub_1,\ub]),\\
 \Eb^{deg}_{l+k}(\ub; [u_1, u]) &:= \sum_{p+q=l} \Eb^{deg}[\delta^p W^{l}_{p,q} \varphi_k](\ub; [u_1, u]),\\
 {}^tF^{deg}_{l+k} (u; [\ub_1,\ub]) &:= \sum_{p+q=l} {}^tF^{deg}[\delta^p W^{l}_{p,q} \varphi_k](u; [\ub_1,\ub]),\\
 \mathcal{S}^{deg}_{l+k}(\D) &:= \sum_{p+q=l} \mathcal{S}^{deg}[\delta^p W^{l}_{p,q} \varphi_k](\D).
\end{align*}
We can similarly define $E^{ndeg}_{l+k}(u; [\ub_1,\ub])$ and $\Eb^{ndeg}_{l+k}(\ub; [u^{NH}, u])$, ${}^tF^{ndeg}_{l+k} (u; [\ub_1,\ub])$, $ \mathcal{S}^{ndeg}_{l+k}(\D)$, where $W^{l}_{p,q}$ is replaced by $Z^{l}_{p,q}$. 

\begin{theorem}\label{Thm-decay-high-energy}
Fix $N \geq 6$. In $\mathcal{R}_2$, there are, for any $\beta \geq \frac{1}{2}$ 
\begin{align*}
|u|^{2\beta} \left( \Eb^{deg}_{l+k}(\ub; [u, +\infty])+ E^{deg}_{l+k}(u; [0,\ub]) + \mathcal{S}^{deg}_{l+k}( \D^{u,+\infty}_{0,\ub}) \right) &\lesssim I_{N+1}^2, \quad l+k \leq N,  \\
 \Eb^{ndeg}_{l+k}(\ub; [u^{NH}, +\infty]) + E^{ndeg}_{l+k}(u, [0,\ub]) + \mathcal{S}^{ndeg}_{l+k}( \D^h) &\lesssim I_{N+1}^2, \quad l+k \leq N,
\end{align*}
where $u^{NH} = \ub -r^\ast_{NH}$ and 
\begin{equation*}
 |u|^{2\beta} \cdot {}^t F^{deg}_{l+k}(u; [0, \ub]) + {}^t F^{ndeg}_{l+k}(u, [0,\ub]) \lesssim   I_{N+1}^2, \quad l+k \leq N-1.
\end{equation*}

\end{theorem}

Theorem \ref{Thm-decay-high-energy} with $l \leq 1,\, l+k \leq N$ has been verified by Theorem \ref{Main-Thm-R2}. The general case  can be proved by an inductive argument  on  $l$ and no new difficulty occurs.
Furthermore, an analogous version of Proposition \ref{Prop-ext-decay-Lb-phi} and Proposition \ref{Prop-nondeg-decay-Lb-phi}, which is collected below,  can be established by induction as well.

\begin{proposition}\label{Prop-ext-decay-Lb-phi-high}
In $\R_2$, given any  real number $\beta \geq \frac{1}{2}$, $l+k \leq N-1$, there are
\begin{align*}
 & \delta^{2p-1}  \| Y Z^l_{p,q}\varphi_k\|^2_{L^2(S_{\ub,u})} + \delta^{2p-2} \|Y Z^l_{p,q}\varphi_k\|^2_{L^2(C_u)} \\
 &\quad  + \delta^{2p} \|Y L Z^{l}_{p,q} \varphi_{k}\|^2_{L^2(C_u)}  \lesssim I^2_{N+1}, \\ 
& \delta^{2p-1}  \| \eta^{\frac{1}{2}}Y W^l_{p,q}\varphi_k\|^2_{L^2(S_{\ub,u})} + \delta^{2p-2} \|\eta^{\frac{1}{2}} Y W^l_{p,q}\varphi_k\|^2_{L^2(C_u)} \\
&\quad + \delta^{2p}  \| \eta^{\frac{1}{2}} Y L  W^{l}_{p,q} \varphi_{k}\|^2_{L^2(C_u)}  \lesssim  I^2_{N+1}  |u|^{-2\beta}. 
\end{align*}
\end{proposition}

\subsection{Smallness on the last cone in $\R_2$}\label{sec-smallness-Cb-2}
We denote $S_{\ub, r}$ the sphere which is the intersection of the hypersurfaces of constant $r$ and constant $\ub$ (in $(r, \ub)$ coordinate), and recall that $\|\cdot \|_{L^2(\Cb_{\ub}^{ND})}$ denotes the $L^2$ norm on $\Cb_{\ub}$ with respect to the non-degenerate volume form $\di \mu_{\Cb_{\ub}^{ND}}  = \eta r^2 \di u \di \sigma_{S^2} =- r^2 \di r \di \sigma_{S^2}$.

\begin{proposition}\label{prop-small-Lb-nabla}
In $\R_2$, we have, for $\beta \geq \frac{1}{2}$,  and $\ub \in [0, \delta]$,
\begin{align*}
\|  |u|^{\beta} \eta^{\frac{1}{2}} \Db^a \Lb^l \varphi_k\|_{L^2(\Cb_{\ub}^{ND})} & \lesssim \delta^{\frac{1}{2}}, \quad  \|\Db^a Y^l \varphi_k\|_{L^2(\Cb_{\ub}^{ND})} \lesssim \delta^{\frac{1}{2}}, &\quad a+l+k \leq N,  \, a \leq 1,\\
 \||u|^{\beta} \eta^{\frac{1}{2}} \Db^a \Lb^l \varphi_k\|_{L^\infty(S_{\ub,u})} & \lesssim \delta^{\frac{1}{2}}, \quad \|\Db^a Y^l \varphi_k\|_{L^\infty(S_{\ub,u})} \lesssim \delta^{\frac{1}{2}}, &\quad a+l+k \leq N-2, \, a \leq 1. 
\end{align*}
\end{proposition}

\begin{proof}
Define $\omega^2 [\psi] (\ub, u)= \int_{S_{\ub,u}}   |\nablaslash\psi|^2 \eta r^2 \di \sigma_{S^2}$.
Take $\psi = \Lb^l \varphi_k$, $l+k \leq N-1$,
\begin{align*}
\p_{\ub} \omega^2 [\Lb^l\varphi_k] (u,\ub)
&= \int_{S_{\ub,u}} \left( 2 \nablaslash \Lb^l \varphi_k  \nablaslash L \Lb^l \varphi_k \eta r^2 + \mu |\nablaslash \Lb^l \varphi_k|^2 \eta r \right) \di \sigma_{S^2} \\
& \lesssim \delta \eta \|L \Lb^l \varphi_{k+1}\|^2_{L^2(S_{\ub,u})} + (\delta^{-1} +1) \omega^2 [\Lb^l\varphi_k] (u,\ub).
\end{align*}
Similarly, define ${}^h\omega^2  [\psi] (u, \ub)= \int_{S_{\ub, u}} |\nablaslash \psi|^2 r^2 \di \sigma_{S^2}$ and take $\psi= Y^l \varphi_k$, then
$$\p_{\ub} {}^h\omega^2  [Y^l\varphi_k] (u,\ub) \lesssim \delta \|L Y^l \varphi_{k+1}\|^2_{L^2(S_{\ub,u})} +\delta^{-1}  {}^h\omega^2 [Y^l\varphi_k] (u,\ub).$$
After applying the Gr\"{o}nwall's inequality, we obtain (since $\varphi \equiv 0$ on $\Cb_0$) for $l + k \leq N-1$,
\begin{align*}
\omega^2  [\Lb^l\varphi_k]  (u,\ub) & \lesssim \delta \|\eta^{\frac{1}{2}} L \Lb^l \varphi_{k+1}\|^2_{L^2(C_u)} \lesssim \delta |u|^{-2\beta},\\
{}^h\omega^2  [Y^l\varphi_k]  (u,\ub)&\lesssim \delta \|L Y^l \varphi_{k+1}\|^2_{L^2(C_u)} \lesssim \delta.
\end{align*}

In order to bound $\|\eta^{\frac{1}{2}} \nablaslash \Lb^l \varphi_k\|_{L^2(\Cb^{ND}_{\ub})}$ and $\|\nablaslash Y^l \varphi_k\|_{L^2(\Cb^{ND}_{\ub})}$, we work in $(r, \ub)$ coordinate system and parametrize $\Cb_{\ub}$ by $\cup_{r} S_{\ub,r}$, and $\omega^2  [\Lb^l\varphi_k]  (u,\ub)$, ${}^h\omega^2  [Y^l\varphi_k]  (u,\ub)$ by $\omega^2  [\Lb^l\varphi_k]  (r,\ub)$, ${}^h\omega^2  [Y^l\varphi_k]  (r,\ub)$, and further integrate  $\omega^2  [\Lb^l\varphi_k]  (r,\ub)$, ${}^h\omega^2  [Y^l\varphi_k]  (r,\ub)$ with respect to the measure $\di r$ on $\Cb_{\ub}$, noting that $r$ is finite in $\R_2$. 

In the same way, defining $h^2 [\psi] (u, \ub)= \int_{S_{\ub,u}}   |\psi|^2 \eta r^2 \di \sigma_{S^2}$ and
${}^hh^2  [\psi] (u, \ub)= \int_{S_{\ub, u}} \psi|^2 r^2 \di \sigma_{S^2}$, we have for $l+k\leq N$,
\begin{align*}
h^2  [\Lb^l\varphi_k]  (u,\ub) & \lesssim \delta \|\eta^{\frac{1}{2}} L \Lb^l \varphi_{k}\|^2_{L^2(C_u)} \lesssim \delta |u|^{-2\beta},\\
{}^hh^2  [Y^l\varphi_k]  (r,\ub)&\lesssim \delta \|L Y^l \varphi_{k}\|^2_{L^2(C_u)} \lesssim \delta.
\end{align*}
And the $L^2$ bound on $\Cb_{\ub}$ follows straightforwardly as before.

Besides, Proposition \ref{Prop-ext-decay-Lb-phi-high} and Theorem \ref{Thm-decay-high-energy} automatically lead to the estimates for $\eta^{\frac{1}{2}} Y \Lb^l \varphi_k$ and $Y\Lb^l \varphi_k$. 

At last,  the $L^\infty$ estimates follow from the Sobolev theorem on $S_{\ub,u}$. Thus, we complete the proof.
\end{proof}

\begin{proposition}\label{prop-small-L-phi}
For any $\beta \geq \frac{1}{2}$, we have on the last cone $\R_2 \cap \Cb_\delta$,
\begin{align*}
\||u|^{\beta} \eta^{\frac{1}{2}} L \Lb^l \varphi_k\|_{L^2(\Cb^{ND}_{\delta})} & \lesssim \delta^{\frac{1}{2}}, \quad  \|L Y^l \varphi_k\|_{L^2(\Cb_{\delta}^{ND})} \lesssim \delta^{\frac{1}{2}}, &\quad l+k \leq N-2, \\ 
 \||u|^{\beta} \eta^{\frac{1}{2}} L \Lb^l \varphi_k\|_{L^\infty(S_{\delta,u})} &\lesssim \delta^{\frac{1}{2}}, \quad \| L Y^l \varphi_k\|_{L^\infty(S_{\delta,u})}  \lesssim \delta^{\frac{1}{2}}, &\quad l+k\leq N-4.
\end{align*}
\end{proposition}

\begin{proof}
Since the proof for general $l$ resembles the case of $l=0$, we will take $L \varphi_k, \, k\leq N-2$ for instance here. The proof is analogous to that of Proposition \ref{prop-small-L-R1}. 

{\bf Degenerate case}: Define $\chi^2 [\psi] (r, \ub)=   \int_{S_{\ub,r}} |L \psi|^2 \eta r^2 \di \sigma_{S^2}.$ 
Take $\psi = \varphi_k, \, k \leq N-2$.
Noting that $\p_r (\eta r^2) =2r-2m = r(2\eta + \mu) >0$, we have 
\begin{align*}
&\p_r \chi^2 [\varphi_k](r,\delta) -  \int_{S_{\delta, r}}  \frac{2\eta+\mu}{r} |L \varphi_k|^2 r^2\di \sigma_{S^2} =   \int_{S_{\delta, r}} 2 L\varphi_k  \p_r L \varphi_k \eta r^2 \di \sigma_{S^2}.
\end{align*}
Integrating over $[r, r_1]$, 
\begin{align*}
& \chi^2 [\varphi_k](r,\delta) +  \int_{\Cb^{[r, r_1]}_{\delta}} \frac{2\eta + \mu}{r}  |L \varphi_k|^2 r^2 \di r \di \sigma_{S^2} \\
= &  \chi^2 [\varphi_k](r_1,\delta)  - \int_{\Cb^{[r, r_1]}_{ \delta}} 2L\varphi_k  \p_r L \varphi_k \eta r^2 \di r \di \sigma_{S^2}.
\end{align*}
We now change to the $(u, \ub)$ coordinate system.
Note that $\p_r = -\eta^{-1} \p_u$, where $\partial_r$ is the coordinate vector field in $(r, \ub)$ coordinate. Thus, the $\p_r L \varphi_k$ above is basically $-\eta^{-1} \Lb L\varphi_k$ in $(u, \ub)$ coordinate. What is more, the volume forms on $\Cb_{\ub}$ in the two coordinate systems are related by $-r^2 \di r \di \sigma_{S^2} =\eta r^2 \di u \di \sigma_{S^2} = \eta \di \mu_{\Cb_{\ub}}$. Consequently,  
\begin{align*}
& \chi^2 [\varphi_k](u,\delta) +  \int_{\Cb^{[u_1, u]}_{\delta}} \frac{\mu}{r}  |L \varphi_k|^2\eta \di \mu_{\Cb_{\ub}} \\
= &  \chi^2 [\varphi_k](u_1,\delta)  - \int_{\Cb^{[u_1, u]}_{ \delta}} 2L\varphi_k \left(\Box_g \varphi_k -\laplacianslash \varphi_k + \frac{\Lb \varphi_k}{r}  \right) \eta^2 \di \mu_{\Cb_{\ub}}.
\end{align*}
Noting the positive term $\int_{\Cb_{\delta}^{[u_1, u]}}  \frac{\mu}{r} |L \varphi_k|^2  \eta \di \mu_{\Cb_{\ub}}$ on the left hand side and applying the Cauchy-Schwarz inequality, we have
\begin{equation}\label{eq-small-L-R2}
\begin{split}
& \sum_{k \leq N-2} \chi^2 [\varphi_k] (u,\delta)+ \int_{\Cb_{\delta}^{[u_1, u]}}  \sum_{k \leq N-2}  |L \varphi_k|^2  \eta \di \mu_{\Cb_{\ub}} \lesssim  \sum_{k \leq N-2} \chi^2 [\varphi_k] (u_1, \delta) \\
& \quad + \int_{\Cb_{\delta}^{[u_1,u]}}  \sum_{k \leq N-2} \left( |\laplacianslash \varphi_k|^2  +| \Lb \varphi_k|^2 +|\Box_g \varphi_k|^2  \right) \eta^3  \di \mu_{\Cb_{\ub}}.
\end{split}
\end{equation}
By the result of Proposition \ref{prop-small-Lb-nabla},
\begin{align*} 
 \int_{\Cb_{\delta}^{[u_1, u]}} \left( | \Lb \varphi_k |^2 +  |\laplacianslash \varphi_k|^2 \right) \eta^3  \di \mu_{\Cb_{\ub}} \lesssim  \int^u_{u_1}  \delta |u^\prime |^{-2\beta} \eta^2 \di u^\prime, \quad k \leq N-2.
\end{align*}
Furthermore, the last term in \eqref{eq-small-L-R2}, is split as $\int_{\Cb_{\delta}} |\Box_g \varphi_k |^2 \eta^3  \di \mu_{\Cb_{\ub}} = {}^sF^k_1 + \cdots + {}^sF^k_3$, where $p+q\leq k \leq N-2, \, p \leq q$ and
\begin{align*}
{}^sF^k_1 &=   \int_{\Cb_{\delta}}   \eta^3 |\Db \varphi_{p}|^2 |\Db \varphi_{q}|^2 \di \mu_{\Cb_{\ub}}, \\
{}^sF^k_2 &=   \int_{\Cb_{\delta}}  \eta^3   |  \Db\varphi_{p}|^2  |L\varphi_{q}|^2 \di \mu_{\Cb_{\ub}}, \\
{}^sF^k_3 &=   \int_{\Cb_{\delta}}   \eta^3 |L\varphi_{p}|^2  |  \Db \varphi_{q}|^2 \di \mu_{\Cb_{\ub}}.
\end{align*}
It is obvious to see that $ {}^sF^k_1 \lesssim   \int^u_{u_1}   \delta^2 |u^\prime |^{-2\beta} \eta^2 \di u^\prime$ and ${}^sF^k_2 \lesssim  \delta \int_{\Cb_{\delta}}  |L \varphi_k|^2 \eta^3 \di \mu_{\Cb_{\ub}}$.
For ${}^sF^k_3$, we apply $L^4$ to all the four factors, since $p\leq [\frac{N-2}{2}] \leq N-4$ and $q \leq N-2$,
\begin{align*}
{}^sF^k_3 &\lesssim   \int_{u_1}^u  \| \eta^{\frac{1}{2}} L \varphi_{p}\|^2_{L^4(S_{\delta,u^\prime})}   \|\eta^{\frac{1}{2}} \Db \varphi_{q}\|^2_{L^4(S_{\delta,u^\prime})} \eta  \di u^\prime \\
& \lesssim  \int_{u_1}^u  \sum_{i \leq p +1} \delta |u|^{-2\beta}  \|L \varphi_{i}\|^2_{L^2(S_{\delta,u^\prime})} \eta  \di u^\prime\\
&  \lesssim \delta  \int_{\Cb_{\delta}}  \sum_{i \leq N-3}  |L\varphi_i|^2 \eta \di \mu_{\Cb_{\ub}},
\end{align*}
where we have used the Sobolev inequalities on $S_{\ub, u}$.
Hence both of ${}^sF^k_2$ and  ${}^sF^k_3$ can be absorbed by the left hand side of \eqref{eq-small-L-R2}.

In a word, we deduce that for any $1 \leq u_1 < u,$ $k \leq N-2$,
\begin{equation}\label{eq-chi-deg-R2}
\begin{split}
&\sum_{k \leq N-2} \chi^2 [\varphi_k] (u,\delta) +   \int_{u_1}^{u} \sum_{k \leq N-2} \chi^2 [\varphi_k] (u^\prime,\delta) \di u^\prime \\
 \lesssim {} & \sum_{k \leq N-2} \chi^2 [\varphi_k] (u_1,\delta)+ \int^u_{u_1}   \delta  |u^\prime |^{-2\beta} \eta^2 \di u^\prime.
\end{split}
\end{equation}
Additionally,  the smallness in Theorem \ref{small-W-R1} tells that $\chi^2 [\varphi_k](1,\delta) \lesssim \delta$.
By the pigeon-hole principle (see Lemma \ref{lema-pigeonhole-2}), we derive that for any $1 \leq u$
\begin{equation}\label{improv-L-Cb-delta-deg}
\sum_{k \leq N-2} \chi^2 [\varphi_k] (u,\delta) 
\lesssim   \delta  |u|^{-2\beta}.
\end{equation}
And integrating \eqref{improv-L-Cb-delta-deg} with respect to $\di r$ gives rise to $\||u|^{\beta} \eta^{\frac{1}{2}}  L\varphi_k\|^2_{L^2(\Cb^{ND}_{\delta})} \lesssim \delta$, $k \leq N-2$.

{\bf Non-degenerate case}: 
Define ${}^h\chi^2 [\psi] (r,\ub)=   \int_{S_{\ub, r}} |L \psi|^2 r^3 \di \sigma_{S^2}$ and take $\psi = \varphi_k, \, k \leq N-2$.  
Noting that $\p_r r^3 =3r^2 >0$, then
\begin{align*}
&\p_r {}^h\chi^2 [\varphi_k](r,\ub) - \int_{S_{\ub,r}} 3r^2  |L \varphi_k|^2  \di \sigma_{S^2} =  \int_{S_{\ub,r}} 2 L\varphi_k \partial_r L \varphi_k  r^3 \di \sigma_{S^2}.
\end{align*}
Integrating  on $\Cb_{\delta}$ along $\p_r$ within the interval $[r, r_{NH}]$, one derives,
\begin{align*}
& {}^h\chi^2 [\varphi_k](r,\delta) + \int_{\Cb^{NH}_{\delta}} 3 |L \varphi_k|^2 r^2 \di r \di \sigma_{S^2}\\
= &  {}^h\chi^2 [\varphi_k](r_{NH},\delta) - \int_{\Cb^{NH}_{\delta}}  2 L\varphi_k \partial_r L \varphi_k  r^3 \di r \di \sigma_{S^2}.
\end{align*}
We now change to the $(u, \ub)$ coordinate, as in the degenerate case, there is,
\begin{align*}
& {}^h\chi^2 [\varphi_k](u,\delta) + \int_{\Cb^{NH}_{\delta}} |L \varphi_k|^2 \di \mu_{\Cb_{\ub}^{ND}}  \lesssim {}^h\chi^2 [\varphi_k] (u^{NH}, \delta) \\
&\quad  \quad \quad + \int_{\Cb_{\delta}^{NH}} \left( |\laplacianslash \varphi_k|^2  +| Y \varphi_k|^2 +|\Box_g \varphi_k|^2  \right)  \di \mu_{\Cb^{ND}_{\ub}},
\end{align*}
where $u>u^{NH}=\delta-r^\ast_{NH}$. 
Analogous to the degenerate case, we can show by the result of Proposition \ref{prop-small-Lb-nabla} that for $u>u^{NH}$,
\begin{equation*}
{}^h\chi^2 [\varphi_k] (u,\delta)+    \|L \varphi_k \|_{L^2(\Cb_{\delta}^{ND})}^2  \lesssim \chi^2 [\varphi_k] (u^{NH},\delta)+  \delta.
\end{equation*}
We finish the proof by further applying the Sobolev theorem on $S_{\delta, u}$.
\end{proof}

Based on Proposition \ref{prop-small-Lb-nabla} and Proposition \ref{prop-small-L-phi}, the smallness for general derivatives of $\varphi$ on $\Cb_{\delta} \cap \R_2$ follows by induction. The proof essentially analogous to that in Theorem \ref{small-W-R1}.
\begin{theorem}\label{small-W-R2}
For any fixed $N\geq 6$ and $\beta \geq \frac{1}{2}$, we have on the last cone $\Cb_\delta \cap \mathcal{R}_2$
\begin{align*}
\| | u|^{\beta} \eta^{\frac{1}{2}} \Db^a W^{l}_{p,q}\varphi_k\|_{L^2(\Cb^{ND}_{\delta})} + \|\Db^a Z^l_{p,q} \varphi_k\|_{L^2(\Cb_{\delta}^{ND})}  &\lesssim \delta^{\frac{1}{2}},  &\quad   a+2l +k \leq N,\, a \leq 1, \\
\|  |u|^{\beta} \eta^{\frac{1}{2}} \Db^a W^{l}_{p,q}\varphi_k\|_{L^\infty(S_{\delta,u})} +  \|\Db^a Z^l_{p,q}\varphi_k\|_{L^\infty(S_{\delta,u})}  &\lesssim  \delta ^{\frac{1}{2}}, &\quad a+2l +k\leq N-2,\, a \leq 1,
\end{align*}
and
\begin{align*}
 \| | u|^{\beta } \eta^{\frac{1}{2}} L W^{l}_{p,q}\varphi_k\|_{L^2(\Cb^{ND}_{\delta})} + \|L Z^l_{p,q} \varphi_k\|_{L^2(\Cb_{\delta}^{ND})}  &\lesssim \delta^{\frac{1}{2}},  &\quad   2l  +k \leq N-2, \\
\|  |u|^{\beta} \eta^{\frac{1}{2}} L W^{l}_{p,q}\varphi_k\|_{L^\infty(S_{\delta,u})} +  \|  L Z^l_{p,q}\varphi_k\|_{L^\infty(S_{\delta,u})} &\lesssim  \delta ^{\frac{1}{2}}, &\quad 2l +k\leq N-4.
\end{align*}
\end{theorem}

\subsection{Small data problem in Region IV}\label{sec-small-data-Cauchy}
For the moment, we specify the small data theorem of \cite[Theorem 1.4]{Luk-15-nonlinear} on the Schwarzschild background.

\begin{theorem}[Luk \cite{Luk-15-nonlinear}, 2013]\label{Small-data-theorem-Luk}
Consider the nonlinear wave equation \eqref{Main Equation} with null quadratic form. There exists an $\epsilon$ such that if the initial data satisfy
\begin{align*}
\sum_{i+j+k \leq 16, l \leq 1} & \int_{\Sigma_{\tau_0} \cap \{ r \geq r_{NH}\}} ( |r D Y^k \partial_{t^\ast}^i \Omega^j S^l \varphi|^2 + |Y^k \partial_{t^\ast}^i \Omega^j S^l \varphi|^2) r^2 \di r^\ast \sigma_{S^2}\\
+ \sum_{i+j+k \leq 16, l \leq 1} & \int_{\Sigma_{\tau_0} \cap \{r < r_{NH}\}}  (|D Y^k \partial_{t^\ast}^i \Omega^j S^l \varphi |^2 + |Y^k \partial_{t^\ast}^i \Omega^j S^l \varphi |^2) r^2 \di r \di \sigma_{S^2} \lesssim \epsilon,
\end{align*}
and 
\begin{align*}
\sum_{l\leq13} (  |r D^l \varphi|_{\Sigma_{\tau_0}} + |r D^l S \varphi|_{\Sigma_{\tau_0}} ) & \lesssim \epsilon.
\end{align*}
Then $\varphi$ exists globally in time. Moreover, for all $\gamma>0$, which we can take sufficiently small such that the solution $\varphi$ obeys the decay estimate
\begin{align*}
& |\varphi| \lesssim \epsilon r^{-1} |u|^{-\frac{1}{2}} |t^{\ast}|^{\gamma}, \,\, |D \varphi| \lesssim \epsilon r^{-1} |u|^{-1} |t^\ast|^\gamma,  \,\, |\bar D \varphi| \lesssim \epsilon r^{-1} |t^\ast|^{-1+\gamma}, \quad r \geq R > r_{NH}, \\
&|\varphi| \lesssim \epsilon r^{-1} |t^\ast|^{-\frac{3}{2}} r^\gamma, \quad |D \varphi| \lesssim \epsilon r^{-1} |t^\ast|^{-\frac{3}{2}} r^{-\frac{1}{2} + \gamma},  \quad  r \leq \frac{t^\ast}{4}.
\end{align*}
\end{theorem}
We now explain some notations in Theorem \ref{Small-data-theorem-Luk}. $t^\ast = t+ \chi (r) 2m \log(r-2m)$, where $\chi (r)$ is a cut off function such that $\chi (r) = 1$ if $r \leq 2m+ \varepsilon$ and $\chi (r) =0$ if $r \geq 2m+2\varepsilon$, with $\varepsilon$ being a fixed and small constant. As a remark, $t^\ast = 2\ub -r + 3m + 2m \log m$, if $r \leq  2m+\varepsilon$.  And here $\Sigma_{\tau} = \{t^\ast = \tau\}$. The commutator $S = t^\ast \partial_{t^\ast} + h(r) \p_r$, where $h(r) = r \eta$ if $r \sim 2m$ and $h(r) = r^\ast \eta$ if $r \geq R$, for some large $R$, and $h(r)$ is interpolated so that it is smooth and non-negative. We note that, $S=\ub L + u \Lb$, if $r\geq R$. Besides, the multiplier $K= \ub^2 L + u^2 \Lb$ is crucial in the proof of \cite{Luk-15-nonlinear}. We will apply this small data theorem to demonstrate the global existence in Region IV.

We prescribe our data on $\Sigma_1 = \{t=1\}$. Set $\Sigma_1^{\text{int}} =\Sigma_1 \cap \{\ub \leq \delta\}$,  $\Sigma_1^{\text{ext}} =\Sigma_1 - \Sigma_1^{\text{int}}$. We may restrict the solution constructed
in Section \ref{sec-past} on $\Sigma_1^{\text{int}}$ to get $(\varphi, \p_t\varphi)|_{\Sigma_1^{\text{int}}} = (\psi^{\text{int}}_0, \psi^{\text{int}}_1)$. According to the estimates derived in Theorem \ref{small-W-R1}, we have the following properties for $(\psi^{\text{int}}_0, \psi^{\text{int}}_1)$: $$\|(\p^k \psi_0^{\text{int}}, \p^{k-1}\psi_1^{\text{int}} )\|_{L^\infty(\p \Sigma_1^{\text{int}} )} \lesssim \delta^{\frac{1}{2}}, \quad  k \leq N/2 -1.$$
 We then apply the Whitney extension theorem (\cite[Theorem 12]{Luli-Whitney-extension} and the references therein, see also the application in \cite{Wang-Yu-13}) to extend $( \psi_0^{\text{int}}, \psi_1^{\text{int}})$ to the entire $\Sigma_1$ to obtain the Cauchy data $(\psi_0, \psi_1)$ verifying 
\begin{align*}
(\psi_0, \psi_1)|_{\Sigma_1^{\text{int}}} &= (\psi_0^{\text{int}}, \psi_1^{\text{int}} ); \\
(\psi_0, \psi_1)_{ \{ x \in \Sigma_1^{\text{ext}} | \text{dis}(x, \Sigma_1^{\text{int}} ) \geq 1\} } &=( 0, 0);\\
\|(\p^k \psi_0, \p^{k-1}\psi_1)\|_{L^\infty( \{ x \in \Sigma_1^{\text{ext}} | \text{dis}(x, \Sigma_1^{\text{int}} ) \leq 1\} )} &\lesssim \delta^{\frac{1}{2}}, \quad k \leq N/2 -2.
\end{align*}
We remark that, this extension is made so that the datum $(\psi_0, \psi_1)|_{ \Sigma_1^{\text{ext}} }$ is small and decays fast enough near infinity, and hence fulfils the requirement in Theorem \ref{Small-data-theorem-Luk}. On the other hand, we should mention that restricting the solution derived in Region III to the $\Sigma_1$ slice does not provide us the desired data, since the decay is not fast enough for the application of Theorem \ref{Small-data-theorem-Luk}.

The global existence in Region IV is reduced to a small data problem, where the data are given on $\Sigma_1^{\text{ext}} \cup (\Cb_\delta \cap \text{Region IV})$. For the data on $\Cb_\delta \cap \text{Region IV}$,  there is by Theorem \ref{small-W-R2} the smallness: for any $N \geq 6$,
\begin{align*}
 & \| D^a S^i W^l_{p,q} \varphi_k\|_{L^2 (\Cb^{AH}_{\delta})} + \|D^a S^i Z^l_{p,q} \varphi_k\|_{L^2 (\Cb^{NH}_{\delta})} \\
  \lesssim {}& \delta^{\frac{1}{2}},  \quad a+2l+k+i \leq N-1, \, a,i \leq 1, \\
 &  \|  r D^a S^i W^l_{p,q} \varphi_k\|_{L^\infty (\Cb^{AH}_{\delta})}  + \| r D^a S^i Z^l_{p,q} \varphi_k\|_{L^\infty (\Cb^{NH}_{\delta})} \\ 
   \lesssim {}& \delta^{\frac{1}{2}}, \quad a+2l+k+i \leq N-3, \,a, i \leq 1,
\end{align*}
where $\Cb_{\delta}^{AH} = \Cb_{\delta} \cap \{r \geq r_{NH}\}, \, \Cb_{\delta}^{NH} = \Cb_{\delta} \cap \{r < r_{NH}\}$ and $S$ is defined as before. We should note that $r$, $\ub$ are finite in $\R_2$ and $u$ is finite in $\{r \geq r_{NH}\} \cap \R_2$ as well. In particular, we notice that the energy associated to $K$ on $\Cb_\delta$: $\int_{\Cb_\delta \cap \{r \geq r_{NH}\} } ( |L S^i W^l_{p,q} \varphi_k|^2 +  |\Db S^i W^l_{p,q} \varphi_k|^2 +  |S^i W^l_{p,q} \varphi_k|^2 ) \di \mu_{\Cb_{\ub}}$  is bounded, which is compatible with the proof of \cite[Theorem 1.4]{Luk-15-nonlinear}, for the multiplier $K$ is used therein.
Meanwhile, the data on $\Sigma_1^{\text{ext}}$, $(\psi_0, \psi_1)|_{ \Sigma_1^{\text{ext}} }$ ($N \geq 30$), satisfy the decay assumptions in Theorem \ref{Small-data-theorem-Luk}.
We can apply Theorem \ref{Small-data-theorem-Luk} to our situation, so that the global existence in Region IV holds true. 
 
The global existence in Region IV together with that in $\R_2$ and Region II leads to Theorem \ref{Thm-froward-whole-region}.

\appendix

\section{Some inequalities}
\subsection{Applications of the pigeon-hole principle}
\begin{lemma}\label{lema-pigeonhole-1}
Suppose $f(t) >0$ satisfies the following inequality: for any $t_2>t_1$ and $\alpha>0$, 
\begin{equation}\label{pigeonhole-ineq}
f(t_2) + \int_{t_1}^{t_2} f(t) \di t \lesssim f(t_1) + t_1^{-\alpha},
\end{equation}
then there exists a universal constant $A$ depending on the initial data $f(t_0)$, such that
\begin{equation*}
f(t)  \lesssim_\alpha A t^{-\alpha}.
\end{equation*}
\end{lemma}

\begin{proof}
Take a dyadic sequence $\{ \tau_i\}$, such that $\tau_i = 1.1^i t_0$. Apply \eqref{pigeonhole-ineq} to the interval $[\tau_i, \tau_{i+1}],$
\begin{equation*}
f(\tau_{i+1}) + \int_{\tau_i}^{\tau_{i+1}} f(t) \di t \lesssim f(\tau_i) + \tau_i^{-\alpha}.
\end{equation*}
By the pigeonhole principle, there exists a sequence $\{\tau_i^\prime\}$ with $\tau_i \leq \tau^\prime_i \leq \tau_{i+1}$, such that 
\begin{equation}\label{pigeon-first-genern}
f(\tau_{i}^\prime)  \lesssim \frac{  f(\tau_i) + \tau_i^{-\alpha} }{\tau_{i+1} - \tau_i}  \lesssim \frac{  f(\tau_i) + \tau_i^{-\alpha} }{ \tau_i}.
\end{equation}
Now, for any $\tau$, there must exist one interval $[\tau_i^\prime, \tau_{i+1}^\prime]$, such that $\tau_i^\prime \leq \tau \leq \tau_{i+1}^\prime.$ Then, applying \eqref{pigeonhole-ineq} to the interval $[\tau^\prime_i, \tau],$ we have
\begin{equation*}
f(\tau)  \lesssim f(\tau^\prime_i) + \tau_i^{\prime-\alpha}.
\end{equation*}
In view of \eqref{pigeon-first-genern} and $\tau_i \leq \tau_i^\prime \leq \tau \leq \tau_{i+1}^\prime \leq \tau_{i+2}=1.1^2\tau_i,$ we have
\begin{align*}
f(\tau)  & \lesssim \frac{  f(\tau_i) + \tau_i^{-\alpha} }{ \tau_i} + \tau_i^{\prime-\alpha} \lesssim  \frac{  f(\tau) + \tau^{-\alpha} }{ \tau} + \tau^{-\alpha}\\
&\lesssim  \frac{  f(\tau_0) + \tau_0^{-\alpha} + \tau^{-\alpha} }{ \tau} + \tau^{-\alpha} \lesssim \tau^{-1} + \tau^{-\alpha}.
\end{align*}
This completes the first generation of iteration.

For any fixed integer $k\in \mathbb{N}$, we can repeat this procedure $k$ times to obtain
 \begin{equation*}
f(\tau) \lesssim_k \tau^{-k} + \tau^{-\alpha}, \quad \text{for any fixed} \,\, k\in\mathbb{N}.
\end{equation*}   
\end{proof}

There is an alternative version of estimate derived from the pigeon-hole principle \cite[Page 859-860]{L-10}.
\begin{lemma}\label{lema-pigeonhole-2}
Suppose $f(t) >0$ satisfies the following inequality: for any $t_2>t_1$ and $\alpha >0$,  
\begin{equation}\label{pigeonhole-ineq-int-2}
f(t_2) + \int_{t_1}^{t_2} f(t) \di t \leq C f(t_1) + B \max\{t_2-t_1, 1\}t_1^{-\alpha},
\end{equation}
where $C$ and $B$ are some universal constants.
Then there exists a universal constant $A$ depending on the initial data $f(t_0)$, such that
\begin{equation*}
f(t)  \lesssim_\alpha A t^{-\alpha}.
\end{equation*}
\end{lemma}

\subsection{Gr\"{o}nwall's inequality}
We recall another version of the Gr\"{o}nwall's inequality \cite{K-R-12}, which will be useful in our proof.
\begin{lemma}\label{lemma-Gronwall}
Let $f(x, y), g(x, y)$ be positive functions defined in the rectangle, $0\leq x\leq x_0, 0\leq y\leq y_0$ which verify the inequality,
\begin{equation*}
f(x, y)+g(x, y)\lesssim J +a\int_{0}^{x}f(x', y)dx'+b \int_{0}^{y}g(x,y')dy'
\end{equation*}
for some nonnegative constants $a, b$ and $J.$ Then, for all $0\leq x\leq x_0, 0\leq y\leq y_0$,
\begin{equation*}
f(x, y), g(x, y)\lesssim Je^{ax+by}.
\end{equation*}
\end{lemma}

\subsection{Sobolev inequality}\label{sec-Soloblev}
The Sobolev inequalities on $S_{\ub ,u}$,
 \begin{equation}\label{Sobolev Inequlities-S2}
\begin{split}
 \|\psi \|_{L^\infty(S_{\ub, u})} &\lesssim r^{-\frac{1}{2}}\| \psi \|_{L^{4}(S_{\ub,u})} + r^{\frac{1}{2}}\| \nablaslash\psi \|_{L^{4}(S_{\ub,u})}, \\
\|\psi\|_{L^p(S_{\ub, u})} &\lesssim r^{\frac{2}{p}} \left( r^{-1} \| \psi \|_{L^{2}(S_{\ub,u})} + \| \nablaslash\psi \|_{L^{2}(S_{\ub,u})} \right), \quad p \in \mathbb{N}.
 \end{split}
 \end{equation}

Referring to  \cite{Christodoulou-09}, there is the Sobolev inequality on the outgoing cone:
For any real function $\psi \equiv 0$ on $\Cb_{0}$,
\begin{equation}\label{Sobolev Inequlities-Cu}
 r^{\frac{1}{2}} \| \psi \|_{L^{4}(S_{\ub,u})} \lesssim \| L \psi \|^{\frac{1}{2}}_{L^{2}(C_{u})}(\| \psi \|^{\frac{1}{2}}_{L^{2}(C_{u})} + \|r\nablaslash \psi\|^{\frac{1}{2}}_{L^{2}(C_{u})}).
 \end{equation} 
Resembling \eqref{Sobolev Inequlities-Cu}, we can prove that,
\begin{equation}\label{Sobolev Inequlities-Cu-deg}
 r^{\frac{1}{2}} \| \eta^{\frac{1}{2}} \psi \|_{L^{4}(S_{\ub,u})} \lesssim \|  \eta^{\frac{1}{2}} L \psi \|^{\frac{1}{2}}_{L^{2}(C_{u})}(\| \eta^{\frac{1}{2}} \psi \|^{\frac{1}{2}}_{L^{2}(C_{u})} + \|r  \eta^{\frac{1}{2}} \nablaslash \psi\|^{\frac{1}{2}}_{L^{2}(C_{u})}). 
 \end{equation}

\end{document}